\pdfoutput=1

\documentclass[11pt]{article}

\usepackage[margin=1cm]{geometry}
\usepackage{fullpage}
\usepackage{amsmath}
\usepackage{accents}
\usepackage{amsthm}
\usepackage{amssymb}
\usepackage[hidelinks]{hyperref}
\usepackage{cleveref}
\usepackage{textcomp}

\newcommand{\mbb}[1]{\mathbb{#1}}
\newcommand{\mbf}[1]{\mathbf{#1}}
\newcommand{\bs}{\boldsymbol}
\newcommand{\tr}{\textup{tr}}
\newcommand{\wt}{\widetilde}
\newcommand{\mc}{\mathcal}
\newcommand{\ac}{\accentset{\circ}}
\newcommand{\diag}{\textup{diag}}
\newcommand{\Tr}[1]{\left\langle#1\right\rangle}
\renewcommand{\det}{\textup{det}}
\renewcommand{\Re}{\textup{Re}}
\renewcommand{\Im}{\textup{Im}}

\numberwithin{equation}{section}

\newtheorem{theorem}{Theorem}

\newtheorem{lemma}{Lemma}[section]

\theoremstyle{remark}

\theoremstyle{definition}

\title{Universality for Weakly Non-Hermitian Matrices: Bulk Limit}
\author{Mohammed Osman}
\date{}

\begin{document}

\maketitle

\abstract{We consider complex, weakly non-Hermitian matrices $A=W_{1}+i\sqrt{\tau_{N}}W_{2}$, where $W_{1}$ and $W_{2}$ are Hermitian matrices and $\tau_{N}=O(N^{-1})$. We first show that for pairs of Hermitian matrices $(W_{1},W_{2})$ such that $W_{1}$ satisfies a multi-resolvent local law and $W_{2}$ is bounded in norm, the bulk correlation functions of the weakly non-Hermitian Gauss-divisible matrix $A+\sqrt{t}B$ converge pointwise to a universal limit for $t=O(N^{-1+\epsilon})$. Using this and the reverse heat flow we deduce bulk universality in the case when $W_{1}$ and $W_{2}$ are independent Wigner matrices with sufficiently smooth density.}

\section{Introduction}
It is often the case that the dynamics of an open quantum system can be described by an effective Hamiltonian that is in some sense a non-Hermitian perturbation of the closed system. A standard example of such a situation is a cavity attached to semi-infinite leads. In this case, the Hamiltonian of the cavity acquires a non-Hermitian part whose norm depends on the strength of the coupling with the leads and whose rank depends on the number of open scattering channels. To describe scattering in open cavities whose closed counterparts are chaotic one would like a model of non-Hermitian random matrices for which we can control the degree of non-Hermiticity. Perhaps the simplest model of this kind is the elliptic Ginibre ensemble, defined by
\begin{align}
    B&=V_{1}+i\sqrt{\tau}_{N}V_{2},\label{eq:ellipticGinUE}
\end{align}
where $V_{1}$ and $V_{2}$ are independent Gaussian Hermitian matrices and $\tau_{N}\in[0,1]$. When $\tau_{N}=1$, the real and imaginary parts of $B$ have the same variance and therefore $B$ can be considered to be maximally non-Hermitian. In this case we say that $B$ belongs to the ensemble of complex Ginibre matrices (GinUE). When $\tau_{N}=0$, then $B=V_{1}$ is Hermitian and belongs to the Gaussian Unitary Ensemble (GUE).

We might expect that for some sequences $\tau_{N}\to0$ as $N\to\infty$, there is a limit in which the eigenvalues of $B$ follow a point process that interpolates between the GUE and GinUE. Fyodorov, Khoruzhenko and Sommers \cite{fyodorov_almost_1997} studied the case in which $\tau_{N}=O(N^{-1})$, which they called the regime of weak non-Hermiticity. Using the method of orthogonal polynomials, they found the following limit for the bulk correlation functions
\begin{align}
    \lim_{N\to\infty}\frac{1}{(N\pi\rho_{sc}(E))^{2k}}\rho^{(k)}_{N}\left(E+\frac{\mbf{z}}{N\pi\rho_{sc}(E)}\right)&=\det\left[K_{\tau_{E}}(z_{j},z_{l})\right],
\end{align}
where $E\in(-2,2)$, $\rho_{sc}$ is the semicircle density, $\tau_{E}=\pi\rho_{sc}(E)\lim_{N\to\infty}N\tau_{N}\in(0,\infty)$ and the kernel is given by (eq. (8) in \cite{fyodorov_almost_1997} with a different scaling)
\begin{align}
    K_{\tau}(z_{j},z_{l})&=\frac{1}{(2\pi)^{3/2}\tau^{1/2}}e^{-\frac{1}{4\tau}\left[(\Im z_{j})^{2}+(\Im z_{l})^{2}\right]}\int_{-1}^{1}e^{-2\tau\lambda^{2}-i\lambda(z_{j}-\bar{z}_{l})}d\lambda.\label{eq:Kbulk}
\end{align}
The same authors also gave heuristic arguments based on the supersymmetry method for the universality of this limit in \cite{fyodorov_universality_1998}. A kind of universality was proven by Akemann, Cikovic and Venker \cite{akemann_universality_2018}, who showed that the same limit governs the bulk statistics in ensembles of Gaussian elliptic matrices with hard or soft constraints on the trace $\tr B^{*}B$.

In this paper we prove that this limit holds in the weak sense when we take $V_{1}$ and $V_{2}$ to be Wigner matrices with sufficiently smooth distributions. We use the three-step strategy introduced in \cite{erdos_bulk_2010} for Hermitian matrices and recently adapted to complex, non-Hermitian matrices in \cite{maltsev_bulk_2023}. Thus our first step is to show that the pointwise limit holds if we take $V_{1}$ and $V_{2}$ to have a small Gaussian component, i.e. to be Gauss-divisible matrices. This is done by an asymptotic analysis of an explicit integral formula for the correlation functions derived by the partial Schur decomposition. For the asymptotic analysis we require $V_{1}$ to satisfy a certain multi-resolvent local law. If $V_{1}$ is a Wigner matrix, then this local law has been proven to hold by Cipolloni, Erd\H{o}s and Schr\"{o}der \cite{cipolloni_optimal_2022}. From this we obtain that the weak limit is universal when $V_{1}$ is a Wigner matrix, $V_{1}$ and $V_{2}$ are independent and both have a small Gaussian component. The Gaussian part can then be removed using the method of time reversal from \cite{erdos_bulk_2010}, which has a straightforward generalisation to non-Hermitian matrices. This step requires both $V_{1}$ and $V_{2}$ to be Wigner matrices with sufficiently smooth distributions.

\paragraph{Notation}
The space of $N\times N$ complex (respectively complex Hermitian) matrices is denoted by $\mbf{M}_{N}$ (respectively $\mbf{M}^{sa}_{N}$). The unitary group of dimension $N$ is denoted by $\mbf{U}_{N}$. Scalar multiples of the identity will be denoted by the scalar alone. We will denote by $\mbf{x}$ the column vector or the diagonal matrix with entries $x_{i}$ depending on the context. For $A\in\mbf{M}_{N}$ we use the following notation: $\Tr{A}:=N^{-1}\tr A$; $\left|A\right|=\sqrt{A^{*}A}$; $\Re A=\frac{1}{2}(A+A^{*})$; $\Im A=\frac{1}{2i}(A-A^{*})$. The operator norm is denoted by $\left\|A\right\|$. The Pauli matrices are defined by
\begin{align}
    \sigma_{X}&=\begin{pmatrix}0&1\\1&0\end{pmatrix},\quad\sigma_{Y}=\begin{pmatrix}0&-i\\i&0\end{pmatrix},\quad\sigma_{Z}=\begin{pmatrix}1&0\\0&-1\end{pmatrix}.\label{eq:Pauli}
\end{align}
The Vandermonde determinant is denoted by $\Delta(\mbf{x})=\Delta(x_{1},...,x_{n})$.

Let $W_{1},W_{2}\in\mbf{M}^{sa}_{N}$ and $G_{z}:=(W_{1}-z)^{-1}$. Let $\ac{W}_{2}:=W_{2}-\langle W_{2}\rangle$ be the traceless part of $W_{2}$. Define the quantities
\begin{align}
    m_{N}(z)&=\Tr{G_{z}},\label{eq:m}\\
    \alpha_{N}(z)&=N\tau_{N}\Tr{(\Re(G_{z})\ac{W}_{2})^{2}},\label{eq:alpha}\\
    \beta_{N}(z)&=N\tau_{N}\Tr{(\Im(G_{z})\ac{W}_{2})^{2}}.\label{eq:beta}
\end{align}
Fix a small constant $\epsilon>0$. Let $n_{\epsilon}=\lceil\frac{48}{\epsilon}\rceil$ and define the domain
\begin{align}
    S_{\epsilon}&=\left\{z\in\mbb{C}:\left|\Re z\right|<10,\,N^{-1+\epsilon}\leq\left|\Im z\right|\leq10\right\}.\label{eq:spectralDomain}
\end{align}
We denote by $\mc{W}_{N,\epsilon}$ the class of matrices $(W_{1},W_{2})$ for which the following conditions hold uniformly in $S_{\epsilon}$:
\paragraph{C0:} There is a constant $c_{0}$ such that
\begin{align}
    \tag{C0}
    \left\|W_{1}\right\|&\leq c_{0},\quad\left\|W_{2}\right\|\leq c_{0}.\label{eq:C0}
\end{align}

\paragraph{C1:} There are positive constants $c_{m},c_{m'}$ such that
\begin{align}
    \tag{C1.1}
    |m_{N}(z)|&\leq c_{m},\label{eq:C1.1}
\end{align}
and, when $|\Re z|<2$,
\begin{align}
    \tag{C1.2}
    1/c_{m}\leq|\Im m_{N}(z)|&\leq c_{m},\label{eq:C1.2}\\
    \tag{C1.3}
    \frac{\Im(z)|m'_{N}(z)|}{\Im m_{N}(z)}&<1-c_{m'}.\label{eq:C1.3}
\end{align}

\paragraph{C2:} There is a constant $c_{\beta}$ such that
\begin{align}
    \tag{C2}
    1/c_{\beta}\leq\beta_{N}(z)&\leq c_{\beta}.\label{eq:C2}
\end{align}

\paragraph{C3:} We have
\begin{align}
    \tag{C3.1}
    \left|\Tr{G_{z}\ac{W}_{2}}\right|&\leq\frac{N^{\epsilon/2}}{N\sqrt{\eta}},\label{eq:C3.1}
\end{align}
and there is a constant $c_{3}$ such that
\begin{align}
    \tag{C3.2}
    \left|\Tr{\prod_{j=1}^{m}G_{z_{j}}\ac{W}_{2}}\right|&\leq \frac{c_{3}}{\eta^{m/2-1}},\label{eq:C3.2}
\end{align}
for $m=2,...,4n_{\epsilon}$, where $\eta=\min_{1\leq j\leq m}\left|\Im z_{j}\right|$.

Note that the lower bound in \eqref{eq:C1.3} identifies $(-2,2)$ as the bulk region.

For each pair $(W_{1},W_{2})\in\mc{W}_{N,\epsilon}$ and a sequence $\tau_{N}\in[0,1]$ we construct an elliptic matrix $A$ by
\begin{align}
    A&=W_{1}+i\sqrt{\tau_{N}}W_{2}.\label{eq:Aelliptic}
\end{align}
The ensemble of Gauss-divisible elliptic matrices is defined by
\begin{align}
    A_{t}&=A+\sqrt{t}B,\label{eq:GaussDivisible}
\end{align}
with $B$ as in \eqref{eq:ellipticGinUE}. For $\tau\in(0,\infty)$ let
\begin{align}
    \rho^{(k)}_{\tau}(\mbf{z})&=\det\left[K_{\tau}(z_{j},z_{l})\right],\label{eq:rhoTau}
\end{align}
with $K_{\tau}$ defined by \eqref{eq:Kbulk}.
\begin{theorem}\label{thm1}
Let $0<\epsilon<1/2$, $(W_{1},W_{2})\in\mc{W}_{N,\epsilon}$ and $N^{-1+2\epsilon}\leq t\leq N^{-\epsilon}$. Then for any $E\in(-2,2)$ there is a unique $\lambda_{E,t}=u_{E,t}+i\eta_{E,t}$ such that $\eta_{E,t}>0$ and
\begin{align}
    \lambda_{E,t}&=E+tm_{N}(\lambda_{E,t}).
\end{align}
Assume that 
\begin{align}
    \tau_{E}=\lim_{N\to\infty}\left[\beta_{N}(\lambda_{E,t})+N\tau_{N}\eta_{E,t}\Im m_{N}(\lambda_{E,t})\right]\label{eq:tauE}
\end{align}
exists and $\tau_{E}\in(0,\infty)$. Let $\rho^{(k)}_{N}$ be the $k$-point correlation function of $A_{t}$. Then for any $k\in\mbb{N}$, with $\rho^{(k)}_{\tau}$ as in \eqref{eq:rhoTau}, we have
\begin{align}
    \lim_{N\to\infty}\left(\frac{t}{N\eta_{E,t}}\right)^{2k}\rho^{(k)}_{N}\left(E+i\sqrt{\tau_{N}}\langle W_{2}\rangle+\frac{t\mbf{z}}{N\eta_{E,t}}\right)&=\rho^{(k)}_{\tau_{E}}(\mbf{z}),
\end{align}
uniformly in compact subsets of $\mbb{C}^{k}$.
\end{theorem}
We remark that by making minor adjustments to the proof one can dispense with the assumption in \eqref{eq:C3.2} if we assume that $t\geq N^{-\omega}$ for sufficiently small $\omega>0$. 

Recall that a Wigner matrix with atom distribution $\nu$ is the random Hermitian matrix whose diagonal entries are independent copies of $\nu/\sqrt{N}$ and whose off-diagonal entries have independent real and imaginary parts which are copies of $\nu/\sqrt{2N}$. We rely on the multi-resolvent local law proven by Cipolloni, Erd\H{o}s and Schr\"{o}der \cite{cipolloni_optimal_2022}, from which we obtain that pairs of independent Wigner matrices $W_{1}$ and $W_{2}$ belong to $\mc{W}_{N,\epsilon}$ with high probability. Using \Cref{thm1} and the reverse heat flow of \cite{erdos_bulk_2010} we can then deduce the following.
\begin{theorem}\label{thm2}
Let $\nu_{L}$ be a probability measure on $\mbb{R}$ such that
\begin{align}
    d\nu_{L}(x)&=e^{-V(x)-x^{2}}dx,\label{eq:nu1}
\end{align}
and
\begin{align}
    V(x)\geq-(1-\delta)x^{2},\quad\sum_{j=1}^{L}|V^{(j)}(x)|&\leq C(1+x^{2})^{l},\label{eq:nu2}
\end{align}
for some $\delta>0$ and $L,l\in\mbb{N}$. Let $W_{1},W_{2}$ be independent Wigner matrices with atom distribution $\nu$ and $A=W_{1}+i\sqrt{\tau_{N}}W_{2}$, with $\tau_{N}\in(0,1)$ a sequence such that
\begin{align}
    \pi\rho_{sc}(E)\lim_{N\to\infty}N\tau_{N}&=\tau_{E}\in(0,\infty),
\end{align}
where $E\in(-2,2)$. Then for any $k\in\mbb{N}$ there is an $L=L(k)$ such that
\begin{align}
    \lim_{N\to\infty}\mbb{E}\left[\sum_{i_{1}\neq\cdots\neq i_{k}}f\left(N\pi\rho_{sc}(E)(z_{i_{1}}-E),...,N\pi\rho_{sc}(E)(z_{i_{k}}-E)\right)\right]&=\int_{\mbb{C}^{k}}f(\mbf{z})\rho^{(k)}_{\tau_{E}}(\mbf{z})d\mbf{z},
\end{align}
where $z_{i}$ are the eigenvalues of $A$.
\end{theorem}
The claim that $(W_{1},W_{2})\in\mc{W}_{N,\epsilon}$ with high probability follows from the multi-resolvent local law if we assume that $W_{1}$ and $W_{2}$ are independent, $W_{1}$ is a Wigner matrix and $W_{2}$ is bounded in norm. Thus the smoothness conditions in \Cref{thm2} arise from our reliance on the method of time reversal. An alternative to this method would be to extend the four moment theorem of Tao and Vu \cite{tao_random_2015} to weakly non-Hermitian matrices.

\section{Preliminaries}
In this section we first state some auxilliary lemmas that are used in subsequent sections. The first lemma relates the resolvent of the projection of a matrix onto a subspace to the resolvent of the original matrix.
\begin{lemma}\label{lem:minorresolvent}
Let $A\in\mbf{M}_{N}$; $U=(U_{k},U_{N-k})\in\mbf{U}_{N}$, where $U_{k}$ and $U_{N-k}$ are $N\times k$ and $N\times(N-k)$ respectively; and $B=U_{N-k}^{*}AU_{N-k}$. Then
\begin{align}\label{eq:minorresolvent}
    U\begin{pmatrix}0&0\\0&B^{-1}\end{pmatrix}U^{*}&=A^{-1}-A^{-1}U_{k}(U_{k}^{*}A^{-1}U_{k})^{-1}U_{k}^{*}A^{-1},
\end{align}
and, if $\Re A>0$,
\begin{align}\label{eq:minornorm}
    \left\|A^{-1}U_{k}(U_{k}^{*}A^{-1}U_{k})^{-1}U_{k}^{*}A^{-1}\right\|&\leq\left\|(\Re A)^{-1}\right\|.
\end{align}
If $\Im A>0$ \eqref{eq:minornorm} holds with $\left\|(\Im A)^{-1}\right\|$ on the right hand side.
\end{lemma}

The following lemma deals with integrals over the Stiefel manifold $\mbf{U}_{N,k}:=\mbf{U}_{N}/\mbf{U}_{N-k}$.
\begin{lemma}\label{lem:sphericalint}
Let $f\in L^{1}(\mbb{C}^{N\times k})$ be continuous on a neighbourhood of $\mbf{U}_{N,k}$ and define $\hat{f}:\mbf{M}^{sa}_{k}\to\mbb{C}$ by
\begin{align}
    \hat{f}(H)&=\frac{1}{\pi^{k^{2}}}\int_{\mbb{C}^{N\times k}}e^{-i\tr HX^{*}X}f(X)dX.
\end{align}
Assume that $\hat{f}\in L^{1}(\mbf{M}^{sa}_{k})$; then
\begin{align}\label{eq:sphericalint}
    \int_{\mbf{U}_{N,k}}f(U)d\mu_{N,k}(U)&=\int_{\mbf{M}^{sa}_{k}}e^{i\tr H}\hat{f}(H)dH.
\end{align}
\end{lemma}

Consider a matrix $M$ with distinct eigenvalues; then we have the partial Schur decomposition 
\begin{align}
    M&=U\begin{pmatrix}z_{1}&\mbf{w}_{1}^{*}V_{1}^{*}&&&\\0&z_{2}&\mbf{w}^{*}_{2}V_{2}^{*}&&\\&0&\ddots&\ddots\\&&\ddots&z_{k}&\mbf{w}_{k}^{*}\\&&&0&M^{(k)}\end{pmatrix}U^{*},\label{eq:partialSchur}
\end{align}
where
\begin{align*}
    U&=\prod_{j=1}^{k}\begin{pmatrix}1_{j-1}&0\\0&R(\mbf{v}_{j})\end{pmatrix},\\
    V_{j}&=\prod_{l=0}^{k-j-1}\begin{pmatrix}1_{k-j-l-1}&0\\0&R(\mbf{v}_{k-l})\end{pmatrix},\\
    R(\mbf{v})&=1-\mbf{v}\mbf{v}^{*},
\end{align*}
$\mbf{z}_{j}\in\mbb{C},\,\mbf{w}_{j}\in\mbb{C}^{N-j},\,\mbf{v}_{j}\in S^{N-j}_{+}=\{\mbf{v}\in S^{N-j}:v_{1}\geq0\}$ for $j=1,...,k$ and $M^{(k)}\in\mbf{M}_{N-k}$. The next lemma gives a formula for the $k$-point correlation function in terms of an integral over these variables.
\begin{lemma}\label{lem:partialSchur}
Let $M$ be a random matrix with density $\rho(M)$ with respect to the Lebesgue measure. Let $\rho^{(k)}_{N}(\mbf{z})$ denote the $k$-point function of $M$. Then we have
\begin{align}
    \rho^{(k)}_{N}(\mbf{z})&=\frac{\left|\Delta(\mbf{z})\right|^{2}}{(2\pi)^{k}}\int_{S^{N-1}}\int_{\mbb{C}^{N-1}}\cdots\int_{S^{N-k}}\int_{\mbb{C}^{N-k}}\int_{\mbf{M}_{N-k}}\rho(M(\mbf{v}_{1},\mbf{w}_{1},...,\mbf{v}_{k},\mbf{w}_{k},M^{(k)}))\nonumber\\
    &\left(\prod_{j=1}^{k}\left|\det(M^{(k)}-z_{j})\right|^{2}\right)dM^{(k)}dS_{N-k}(\mbf{v}_{k})d\mbf{w}_{k}\cdots dS_{N-1}(\mbf{v}_{1})d\mbf{w}_{1}.\label{eq:kpointintegral}
\end{align}
\end{lemma}

Finally, we recall Fischer's inequality \cite{fischer_uber_1908} for the determinant of positive semi-definite block matrices.
\begin{lemma}\label{lem:fischerineq}
Let $A$ be a positive semi-definite block matrix with diagonal blocks $A_{nn}$. Then
\begin{align}\label{eq:fischerineq}
    \det A&\leq\prod_{n}\det A_{nn}.
\end{align}
\end{lemma}

The proofs of the first three lemmas can be found in Section 2 of \cite{maltsev_bulk_2023}.

\section{Properties of $\mc{W}_{N,\epsilon}$}
In this section we derive some properties of the class $\mc{W}_{N,\epsilon}$. In the statements of the following lemmas we fix a pair $(W_{1},W_{2})\in\mc{W}_{N,\epsilon}$ and a sequence $\tau_{N}=O(N^{-1})$. We consider the following elliptic matrix
\begin{align}
    \ac{A}&=W_{1}+i\sqrt{\tau_{N}}\ac{W}_{2},
\end{align}
where $\ac{W}_{2}=W_{2}-\langle W_{2}\rangle$. Note that in contrast to \eqref{eq:Aelliptic} we have taken the traceless part of $W_{2}$. Let $F_{z}:=(\ac{A}-z)^{-1}$. There are two other kinds of resolvents that we will need alongside $G_{z}$ and $F_{z}$. The first kind is a block matrix resolvent: 
\begin{align}
    \mbf{G}_{\mbf{z}}&:=(1_{2k}\otimes W_{1}-\mbf{z}\otimes1_{N})^{-1},\label{eq:mbG}\\
    \mbf{F}_{\mbf{z}}(L)&:=(1_{2k}\otimes W_{1}+i\sqrt{\tau_{N}}L\otimes \ac{W}_{2}-\mbf{z}\otimes1_{N})^{-1},\label{eq:mbF}
\end{align}
where $L\in\mbf{M}^{sa}_{2k}$. Note that $\mbf{G}_{\mbf{z}}$ is a block diagonal matrix whose diagonal blocks are $G_{z_{i}}$.

The second kind is the resolvent of the Hermitisation:
\begin{align}
    \mc{G}_{u}(\eta)&:=\begin{pmatrix}-i\eta&W_{1}-u\\W_{1}-u&-i\eta\end{pmatrix}^{-1},\label{eq:mcG}\\
    \mc{F}_{u}(\eta)&:=\begin{pmatrix}-i\eta&\ac{A}^{*}-u\\\ac{A}-u&-i\eta\end{pmatrix}^{-1}.\label{eq:mcF}
\end{align}
We also define
\begin{align}
    H_{u}(\eta)&:=\left(\eta^{2}+|\ac{A}-u|^{2}\right)^{-1},\label{eq:H}
\end{align}
which appears in the block representation of $\mc{F}_{u}(\eta)$:
\begin{align*}
    \mc{F}_{u}(\eta)&=\begin{pmatrix}i\eta H_{u}(\eta)&H_{u}(\eta)(\ac{A}^{*}-u)\\
    (\ac{A}-u)H_{u}(\eta)&\frac{1-(\ac{A}-u)H_{u}(\eta)(\ac{A}-u)}{\eta^{2}}\end{pmatrix}.
\end{align*}
Our convention is that the letter $F$ stands for resolvents of the elliptic matrix $\ac{A}$ and $G$ for resolvents of $W_{1}$; we will treat the former as a perturbation of the latter.

The following combinations of Pauli matrices are convenient to define:
\begin{align}
    \sigma_{\pm}&:=\frac{1\pm\sigma_{X}}{2},\\
    \rho&:=\frac{\sigma_{Z}+i\sigma_{Y}}{2}.
\end{align}
With $z=u+i\eta$ we can write
\begin{align}
    \mc{G}_{u}(\eta)&=\sigma_{+}\otimes G_{z}-\sigma_{-}\otimes G^{*}_{z},\\
    \mc{G}_{u}(\eta)(\sigma_{Y}\otimes\ac{W}_{2})&=i(\rho\otimes G^{*}_{z}\ac{W}_{2}+\rho^{*}\otimes G_{z}\ac{W}_{2}).
\end{align}

We first note that by contour integration \eqref{eq:C3.1} and \eqref{eq:C3.2} imply
\begin{align}
    \left|\Tr{G_{z}^{1+m}\ac{W}_{2}}\right|&\leq\frac{CN^{\epsilon/2}}{N\eta^{m+1/2}},\label{eq:C3.1'}\\
    \left|\Tr{\prod_{j=1}^{m}G^{1+m_{j}}_{z_{j}}\ac{W}_{2}}\right|&\leq\frac{C}{\eta^{\sum_{j}m_{j}+m/2-1}},\label{eq:C3.2'}
\end{align}
for $\eta>2N^{-1+\epsilon}$ and $m=2,...,4n_{\epsilon}$. Indeed, we can write
\begin{align*}
    G^{1+m}_{z}&=\frac{1}{2\pi i}\int_{\mc{C}}\frac{G_{w}}{(w-z)^{1+m}}dw,
\end{align*}
where $\mc{C}$ is a circle of radius $\eta/2$ centred at $z$. On this circle we have $|\Im w|\geq\eta/2>N^{-1+\epsilon}$ and we can apply \eqref{eq:C3.1} or \eqref{eq:C3.2}.

The main consequences of \eqref{eq:C0}, \eqref{eq:C3.1} and \eqref{eq:C3.2} are the following norm bounds.
\begin{lemma}\label{lem:normBound}
For any $m=1,...,n_{\epsilon}$ and $z_{j}\in S_{\epsilon}$, we have
\begin{align}
    \left\|\big(\prod_{j=1}^{m}G_{z_{j}}\ac{W}_{2}\big)G_{z_{m+1}}\right\|&\leq\frac{C(N\eta)^{m/2n_{\epsilon}}}{\eta^{m/2+1}},\label{eq:productNorm1}
\end{align}
where $\eta=\min_{j}|\Im z_{j}|$.
\end{lemma}
\begin{proof}
We will assume that $z_{j}=z$ for all $j$ to ease notation; the general case is proven in exactly the same way. Using $\|A\|=\sqrt{\|AA^{*}\|}\leq(\tr(AA^{*})^{n})^{1/2n}$, we find for any $n\in\mbb{N}$
\begin{align*}
    \|(G_{z}\ac{W}_{2})^{m}G_{z}\|&\leq\Big(\tr\big((G_{z}\ac{W}_{2})^{m}G_{z}G^{*}_{z}(\ac{W}_{2}G^{*}_{z})^{m}\big)^{n}\Big)^{1/2n}\\
    &=\frac{1}{\eta}\Big(\tr\big(\Im(G_{z})\ac{W}_{2}(G_{z}\ac{W}_{2})^{m-1}\Im(G_{z})\ac{W}_{2}(G^{*}_{z}\ac{W}_{2})^{m-1}\big)^{n}\Big)^{1/2n}.
\end{align*}
The trace involves $2nm$ resolvents; choosing $n=\lfloor 2n_{\epsilon}/m\rfloor$ we have $2nm\leq4n_{\epsilon}$ and so we can apply \eqref{eq:C3.2} to obtain
\begin{align*}
    \|(G_{z}\ac{W}_{2})^{m}G_{z}\|&\leq\left(\frac{c_{3}N}{\eta^{n(m+2)-1}}\right)^{1/2n}\\
    &\leq\frac{C(N\eta)^{1/2n}}{\eta^{m/2+1}}.
\end{align*}
Since $n\geq 2n_{\epsilon}/m-1\geq n_{\epsilon}/m$ we obtain \eqref{eq:productNorm1}.
\end{proof}

From this norm bound we derive the following estimates for the norms of resolvents.
\begin{lemma}[Norm estimates]\label{lem:normEstimates}
For any $z=u+i\eta\in S_{\epsilon}$ we have
\begin{align}
    \|\mc{G}_{u}(\eta)(\sigma_{Y}\otimes\ac{W}_{2})\mc{F}_{u}(\eta)\|&\leq\frac{C(N\eta)^{1/2n_{\epsilon}}}{\eta^{3/2}},\label{eq:mcGWmcFnorm}\\
    \|(\mc{G}_{u}(\eta)(\sigma_{Y}\otimes\ac{W}_{2}))^{2}\mc{F}_{u}(\eta)\|&\leq\frac{C(N\eta)^{1/n_{\epsilon}}}{\eta^{2}},\label{eq:mcGW2mcFnorm}\\
    \|\mc{F}_{u}(\eta)(1\otimes\ac{W}_{2})\mc{F}^{*}_{u}(\eta)\|&\leq\frac{C(N\eta)^{1/2n_{\epsilon}}}{\eta^{3/2}}.\label{eq:mcFWmcFnorm}
\end{align}
For any $\mbf{z}\in S^{2k}_{\epsilon}$ and $L\in\mbf{M}^{sa}_{2k}$ we have
\begin{align}
    \|\mbf{F}_{\mbf{z}}(L)\|&\leq\frac{C}{\eta},\label{eq:mbfNorm}\\
    \|\mbf{G}_{\mbf{z}}(L\otimes\ac{W}_{2})\mbf{F}_{\mbf{z}}(L)\|&\leq\frac{C(N\eta)^{1/2n_{\epsilon}}}{\eta^{3/2}},\label{eq:mbfGWmbfFnorm}
\end{align}
where $\eta=\min|\Im z_{j}|$.
\end{lemma}
\begin{proof}
Let $D=-\sqrt{\tau_{N}}\mc{G}_{u}(\eta)(\sigma_{Y}\otimes\ac{W}_{2})$. The matrices $\rho$ and $\sigma_{\pm}$ satisfy
\begin{align}
    \rho^{*}\rho&=\sigma_{+},\quad\rho\rho^{*}=\sigma_{-},\quad\rho^{2}=\sigma_{+}\sigma_{-}=0,\quad\rho=\rho\sigma_{+}=\sigma_{-}\rho.\label{eq:relations}
\end{align}
From these it follows that
\begin{align}
    D^{2n}&=\tau_{N}^{n}\left(\sigma_{+}\otimes(G_{z}\ac{W}_{2}G^{*}_{z}\ac{W}_{2})^{n}+\sigma_{-}\otimes(G^{*}_{z}\ac{W}_{2}G_{z}\ac{W}_{2})^{n}\right),\label{eq:D2n}\\
    D^{2n+1}&=-\tau_{N}^{n+1/2}\left(\rho\otimes(G^{*}_{z}\ac{W}_{2}G_{z}\ac{W}_{2})^{n}G^{*}_{z}\ac{W}_{2}+\rho^{*}\otimes(G_{z}\ac{W}_{2}G^{*}_{z}\ac{W}_{2})^{n}G_{z}\right)\label{eq:D2n+1}.
\end{align}
Using \eqref{eq:productNorm1} and \eqref{eq:C0} we obtain
\begin{align*}
    \|D^{n}\|&\leq\frac{C\tau_{N}^{n/2}(N\eta)^{(n-1)/2n_{\epsilon}}\|\ac{W}_{2}\|}{\eta^{(n+1)/2}}\\
    &\leq\frac{C\sqrt{N}}{(N\eta)^{\frac{n}{2}\left(1-\frac{1}{n_{\epsilon}}\right)+\frac{1}{2}\left(1+\frac{1}{n_{\epsilon}}\right)}}\\
    &\leq\frac{C}{N^{\frac{\epsilon(n+1)-1}{2}}},
\end{align*}
since $\tau_{N}=O(N^{-1})$ and $\eta>N^{-1+\epsilon}$. In particular, for $n=n_{\epsilon}\geq48/\epsilon$ we have
\begin{align}
    \|D^{n_{\epsilon}}\|&\leq\frac{C}{N^{47/2}}.\label{eq:Dnepsilon}
\end{align}
Therefore we can invert $1-D$ and write
\begin{align*}
    \|\mc{G}_{u}(\eta)(\sigma_{Y}\otimes\ac{W}_{2})\mc{F}_{u}(\eta)\|&=\frac{1}{\sqrt{\tau_{N}}}\|D\mc{F}_{u}(\eta)\|\\
    &\leq\frac{1}{\sqrt{\tau_{N}}}\sum_{m=0}^{n_{\epsilon}-1}\|D^{m+1}\mc{G}_{u}(\eta)\|+\frac{\|D^{n_{\epsilon}}\|\cdot\|D\mc{F}_{u}(\eta)\|}{\sqrt{\tau_{N}}}.
\end{align*}
By \eqref{eq:productNorm1} again we have
\begin{align*}
    \frac{1}{\sqrt{\tau_{N}}}\|D^{m+1}\mc{G}_{u}(\eta)\|&\leq\frac{C(N\eta)^{1/2n_{\epsilon}}}{\eta^{3/2}}\left(\frac{(N\eta)^{1/n_{\epsilon}}\tau_{N}}{\eta}\right)^{\frac{m}{2}},
\end{align*}
from which \eqref{eq:mcGWmcFnorm} follows. The bound in \eqref{eq:mcGW2mcFnorm} is proven in exactly the same way.

For \eqref{eq:mcGWmcFnorm} we write
\begin{align*}
    \mc{F}_{u}(\eta)(1\otimes\ac{W}_{2})\mc{F}^{*}_{u}(\eta)&=\sum_{m_{1},m_{2}=0}^{n_{\epsilon}-1}D^{m_{1}}\mc{G}_{u}(\eta)(1\otimes\ac{W}_{2})\mc{G}^{*}_{u}(\eta)(D^{*})^{m_{2}}\\&+2\Re(D^{n_{\epsilon}}\mc{F}_{u}(\eta)(1\otimes\ac{W}_{2})\mc{F}^{*}_{u}(\eta))-D^{n_{\epsilon}}\mc{F}_{u}(\eta)(1\otimes\ac{W}_{2})\mc{F}^{*}_{u}(\eta)(D^{*})^{n_{\epsilon}}.
\end{align*}
From this we deduce
\begin{align*}
    \left\|\mc{F}_{u}(\eta)(1\otimes\ac{W}_{2})\mc{F}^{*}_{u}(\eta)\right\|&\leq\frac{1}{1-2\|D^{n_{\epsilon}}\|-\|D^{n_{\epsilon}}\|^{2}}\sum_{m_{1},m_{2}=0}^{n_{\epsilon}-1}\left\|D^{m_{1}}\mc{G}_{u}(\eta)(1\otimes\ac{W}_{2})\mc{G}^{*}_{u}(\eta)(D^{*})^{m_{1}}\right\|.
\end{align*}
To each term in the sum we apply \eqref{eq:productNorm1}:
\begin{align*}
    \|D^{m_{1}}\mc{G}_{u}(\eta)(1\otimes\ac{W}_{2})\mc{G}^{*}_{u}(\eta)(D^{*})^{m_{2}}\|&\leq\frac{C(N\eta)^{1/2n_{\epsilon}}}{\eta^{3/2}}\left(\frac{(N\eta)^{1/n_{\epsilon}}\tau_{N}}{\eta}\right)^{\frac{m}{2}}.
\end{align*}

Now let $D=-i\sqrt{\tau_{N}}\mbf{G}_{\mbf{z}}(L\otimes\ac{W}_{2})$; since $D^{n}$ is a sum of $O(k^{n})=O(1)$ terms we have by \eqref{eq:C0} and \eqref{eq:productNorm1} as before 
\begin{align*}
    \|D^{n}\|&\leq\frac{C}{N^{\frac{\epsilon(n+1)-1}{2}}},
\end{align*}
The bound \eqref{eq:mbfGWmbfFnorm} now follows in the same way as \eqref{eq:mcGWmcFnorm}. Finally, we note that by the resolvent identity and \eqref{eq:mbfGWmbfFnorm} we have
\begin{align*}
    \|\mbf{F}_{\mbf{z}}(L)\|&\leq\|\mbf{G}_{\mbf{z}}\|+\sqrt{\tau_{N}}\|\mbf{G}_{\mbf{z}}(L\otimes\ac{W}_{2})\mbf{F}_{\mbf{z}}(L)\|\\
    &\leq\frac{1}{\eta}\left(1+\frac{C}{(N\eta)^{\frac{1}{2}\left(1-\frac{1}{n_{\epsilon}}\right)}}\right)\\
    &\leq\frac{C}{\eta}.
\end{align*}
\end{proof}

We will need to estimate traces of monomials in $\mc{F}$ or $\mbf{F}$ and $\ac{W}_{2}$, which we do by the following lemma.
\begin{lemma}[Trace estimates]\label{lem:traceEstimates}
Let $z\in\mbb{C}$ such that $u=\Re z$ and $\eta=|\Im z|>2N^{-1+\epsilon}$. For any monomial $P$ with degrees $p,q$ and $r$ in $\mc{F}_{u}(\eta),\,\mc{F}^{*}_{u}(\eta)$ and $1\otimes\ac{W}_{2}$ respectively, such that $r\leq p+q\leq 8$ and in which $1\otimes\ac{W}_{2}$ does not appear consecutively, 
\begin{align}
    \langle P\rangle&=\langle Q\rangle+O\left(\frac{1}{N^{1/2}\eta^{p+q-r/2-1/2}}\right),\label{eq:traceEstimate}
\end{align}
where $Q$ is the monomial obtained by replacing $\mc{F}_{u}(\eta)$ and $\mc{F}^{*}_{u}(\eta)$ with $\mc{G}_{u}(\eta)$ and $\mc{G}^{*}_{u}(\eta)$ respectively. For any $\mbf{R}\in\mbf{M}^{sa}_{2k}$ and $\mbf{z}\in \mbb{C}^{2k}$ such that $\eta=\min|\Im z_{j}|>2N^{-1+\epsilon}$, let $P$ be a monomial with degrees $p,q$ and $r$ in $\mbf{F}_{\mbf{z}}(L),\,\mbf{F}^{*}_{\mbf{z}}(L)$ and $\mbf{R}$ respectively such that $r\leq p+q\leq4$. Then we have
\begin{align}
    \langle P\rangle&=\langle Q\rangle+O\left(\frac{\|\mbf{R}\|^{r}}{N\eta^{p+q}}\right),\label{eq:traceEstimate2}
\end{align}
where $Q$ is the monomial obtained by replacing $\mbf{F}_{\mbf{z}}(L)$ and $\mbf{F}^{*}_{\mbf{z}}(L)$ by $\mbf{G}_{\mbf{z}}$ and $\mbf{G}^{*}_{\mbf{z}}$ respectively.
\end{lemma}
\begin{proof}
For ease of notation we write $\mc{F}=\mc{F}_{u}(\eta)$ and $\mc{G}=\mc{G}_{u}(\eta)$. A general monomial $P$ has the form
\begin{align*}
    P&=\mc{F}^{p_{1}}(\mc{F}^{*})^{r_{1}}(1\otimes\ac{W}_{2})\cdots\mc{F}^{p_{r}}(\mc{F}^{*})^{q_{r}}(1\otimes\ac{W}_{2}),
\end{align*}
where $\sum_{j}p_{j}=p,\,\sum_{j}q_{j}=q$. Let $D=-\sqrt{\tau_{N}}\mc{G}(\sigma_{Y}\otimes\ac{W}_{2})$; we replace each $\mc{F}$ and $\mc{F}^{*}$ using the identities
\begin{align}
    \mc{F}&=\sum_{m=0}^{\lceil n_{\epsilon}/p\rceil-1}D^{m}\mc{G}+D^{\lceil n_{\epsilon}/p\rceil}\mc{F},\label{eq:series1}\\
    \mc{F}^{*}&=\sum_{m=0}^{\lceil n_{\epsilon}/q\rceil-1}\mc{G}^{*}(D^{*})^{m}+\mc{F}^{*}(D^{*})^{\lceil n_{\epsilon}/q\rceil}.\label{eq:series2}
\end{align}
By our choices of cut-off $\lceil n_{\epsilon}/p\rceil$ and $\lceil n_{\epsilon}/r\rceil$, each term in the resulting sum will be an alterating product containing no more that $2n_{\epsilon}+r$ factors of $\ac{W}_{2}$. Hence we can estimate each term by \eqref{eq:C3.2'}; additional factors of $D$ reduce the size of the trace by $1/\sqrt{N\eta}$. The condition $p+q\leq 8$ and the lower bound $n_{\epsilon}\geq48/\epsilon$ allow us to use the crude bound $|\langle X\rangle|\leq \|X\|$ to bound the contribution from the second terms in the right hand sides of \eqref{eq:series1} and \eqref{eq:series2}. To see this we note that in the proof of \Cref{lem:normEstimates} we established the bound
\begin{align*}
    \|D^{n}\|&\leq\frac{C}{N^{\frac{\epsilon n-1}{2}}}.
\end{align*}
The monomial $P$ has the bound $\|P\|\leq C\eta^{-(p+q)}$. The term containing $D^{\lceil n_{\epsilon}/p\rceil}$ can thus be bounded by
\begin{align*}
    \frac{C}{N^{\frac{\epsilon n_{\epsilon}}{2p}-\frac{1}{2}}\eta^{p+q}}.
\end{align*}
For this to be smaller than the error term in \eqref{eq:traceEstimate}, we require that
\begin{align*}
    \frac{\epsilon n_{\epsilon}}{2p}-\frac{3}{2}\geq\frac{r+1}{2}.
\end{align*}
Since $p\leq 8$ and $r\leq8$, this will certainly be true if
\begin{align*}
    \epsilon n_{\epsilon}\geq48,
\end{align*}
which is indeed the case by our choice of $n_{\epsilon}$. These are very crude bounds; one could reduce the size of $n_{\epsilon}$ by a more refined analysis. The proof of \eqref{eq:traceEstimate2} follows the same pattern; in this case we set $D=-i\sqrt{\tau_{N}}\mbf{G}_{\mbf{z}}(L\otimes\ac{W}_{2})$.
\end{proof}

We also need to estimate the two determinants contained in the following.
\begin{lemma}[Determinant estimates]\label{lem:detEstimates}
Let $\mbf{z}\in S^{2k}_{\epsilon}$ and $L\in\mbf{M}_{2k}$; then
\begin{align}
    \det\left(1+i\sqrt{\tau_{N}}\mbf{G}_{\mbf{z}}(L\otimes\ac{W}_{2})\right)&=\exp\left\{\frac{N\tau_{N}}{2}\sum_{i,j=1}^{2k}\langle G_{z_{i}}\ac{W}_{2}G_{z_{j}}\ac{W}_{2}\rangle\right\}+O\left(\frac{N^{\epsilon/2}}{\sqrt{N\eta}}\right),\label{eq:det1}\\
    \det\left(1+\sqrt{\tau_{N}}\mc{G}_{u}(\eta)(\sigma_{Y}\otimes\ac{W}_{2})\right)&=\exp\left\{N\tau_{N}\langle G_{z}\ac{W}_{2}G^{*}_{z}\ac{W}_{2}\rangle\right\}+O\left(\frac{1}{N\eta}\right).\label{eq:det2}
\end{align}
\end{lemma}
\begin{proof}
Let $D=-i\sqrt{\tau_{N}}\mbf{G}_{\mbf{z}}(L\otimes\ac{W}_{2})$. From the proof of \Cref{lem:normEstimates} we know that $(1-xD)^{-1}$ exists for all $0<x<1$ and satisfies
\begin{align*}
    \|(1-xD)^{-1}\|&\leq\frac{1}{1-x\|D^{n_{\epsilon}}\|}\sum_{m=0}^{n_{\epsilon}-1}x^{m}\|D^{m}\|\\
    &\leq\frac{C}{\sqrt{N}\eta},
\end{align*}
for all $0\leq x\leq 1$. Using the integral representation of the logarithm we find
\begin{align*}
    \det(1-D)&=\exp\left\{-\sum_{m=1}^{n_{\epsilon}-1}\frac{1}{m}\tr D^{m}-\int_{0}^{1}x^{n_{\epsilon}}\tr D^{n_{\epsilon}}(1-xD)^{-1}\right\}.
\end{align*}
We can bound the remainder by taking norms:
\begin{align*}
    |\tr D^{n_{\epsilon}}(1-xD)^{-1}|&\leq N\|D^{n_{\epsilon}}\|\cdot\|(1-xD)^{-1}\|\\
    &\leq\frac{C}{N^{23}\eta}.
\end{align*}
Using \eqref{eq:C3.1} and \eqref{eq:C3.2} we can estimate each term in the sum. For $m=1$ we have
\begin{align*}
    \tr D&=-i\sqrt{\tau_{N}}\sum_{i=1}^{2k}l_{ii}\tr G_{z_{i}}\ac{W}_{2}=O\left(\frac{N^{\epsilon/2}}{\sqrt{N\eta}}\right),
\end{align*}
and for $m>2$ we have
\begin{align*}
    \tr D^{m}&=(-i)^{m}\tau^{m/2}_{E,t}\sum l_{i_{1}i_{2}}\cdots l_{i_{m}i_{1}}\tr G_{z_{i_{1}}}\ac{W}_{2}\cdots G_{z_{i_{m}}}\ac{W}_{2}=O\left(\frac{1}{(N\eta)^{m/2-1}}\right).
\end{align*}
Therefore only the $m=2$ term remains and \eqref{eq:det1} follows.

The estimate \eqref{eq:det2} is proven in the same way, with $D=-\sqrt{\tau_{N}}\mc{G}_{u}(\eta)(\sigma_{Y}\otimes W_{2})$. The difference in this case is that only even powers of $D$ contribute due to the equality \eqref{eq:D2n+1} and the fact that $\tr\rho=0$. Hence the smaller error term in \eqref{eq:det2}.
\end{proof}

\section{Proof of \Cref{thm1}}
Let $0<\epsilon<1/2$, $t\geq N^{-1+2\epsilon}$, $E\in(-2,2)$ and $W_{1},W_{2}\in\mc{W}_{N,\epsilon}$. Let $A=W_{1}+i\sqrt{\tau_{N}}W_{2}$ and $B=V_{1}+i\sqrt{\tau_{N}}V_{2}$, where $V_{1}$ and $V_{2}$ are independent GUE matrices. We will apply the partial Schur decomposition formula in \eqref{eq:kpointintegral} to the Gauss-divisible matrix
\begin{align*}
    M_{t}&=A+\sqrt{t}B,
\end{align*}
where
\begin{align*}
    B&=V_{1}+i\sqrt{\tau_{N}}V_{2},
\end{align*}
and $V_{1},V_{2}$ are independent GUE matrices. This ensemble has density
\begin{align*}
    \rho(M)&=\frac{1}{Z_{N,t}}e^{-\frac{N}{2t}\tr\left[(\Re M-W_{1})^{2}+\frac{1}{\tau_{N}}(\Im M-W_{2})^{2}\right]},
\end{align*}
with respect to the Lebesgue measure $dM=2^{N(N-1)}d(\Re M)d(\Im M)$, where
\begin{align*}
    Z_{N,t}&=\left(\frac{2\pi t\sqrt{\tau_{N}}}{N}\right)^{N^{2}}.
\end{align*}
Let
\begin{align}
    U^{*}AU&=\begin{pmatrix}a_{1}&\mbf{b}_{1}^{*}V_{1}^{*}&&&\\
    V_{1}\mbf{c}_{1}&a_{2}&\mbf{b}_{2}^{*}V_{2}^{*}&&\\
    &V_{2}\mbf{c}_{2}&\ddots&\ddots&\\
    &&\ddots&a_{k}&\mbf{b}_{k}^{*}\\
    &&&\mbf{c}_{k}&A^{(k)}\end{pmatrix},\label{eq:UAU}
\end{align}
where $U$ is the unitary in the partial Schur decomposition \eqref{eq:partialSchur} of $M_{t}$. We have, with $A^{(0)}=A$,
\begin{align}
    a_{j}&=\mbf{v}_{j}^{*}A^{(j-1)}\mbf{v}_{j},\nonumber\\
    \begin{pmatrix}0\\\mbf{b}_{j}\end{pmatrix}&=R_{j}(1-\mbf{v}_{j}\mbf{v}_{j}^{*})A^{(j-1)*}\mbf{v}_{j},\nonumber\\
    \begin{pmatrix}0\\\mbf{c}_{j}\end{pmatrix}&=R_{j}(1-\mbf{v}_{j}\mbf{v}_{j}^{*})A^{(j-1)}\mbf{v}_{j},\nonumber\\
    \begin{pmatrix}0&0\\0&A^{(j)}\end{pmatrix}&=R_{j}(1-\mbf{v}_{j}\mbf{v}_{j}^{*})A^{(j-1)}(1-\mbf{v}_{j}\mbf{v}_{j}^{*})R_{j}.\label{eq:Arecurrence}
\end{align} 
The matrices $A^{(j)}$, $W^{(j)}:=\Re A^{(j)}$ and $W_{2}^{(j)}:=\Im A^{(j)}$ are projections of $A$, $W$ and $W_{2}$ respectively onto a space of codimension $j$ (the orthogonal complement of the span of the first $j$ columns of $U$). We define
\begin{align}
    \ac{W}^{(j)}_{2}&:=W^{(j)}_{2}-\langle W_{2}\rangle,\label{eq:acWj}\\
    \ac{A}^{(j)}&:=W^{(j)}_{1}+i\sqrt{\tau_{N}}\ac{W}^{(j)}_{2};\label{eq:acAj}
\end{align}
note that we have subtracted the trace of the original matrix $W_{2}$ rather than the trace of $W^{(j)}_{2}$, so $\ac{W}^{(j)}_{2}$ is no longer traceless (it is approximately traceless by interlacing). We denote by superscripts $(j)$ any quantity obtained from that without a superscript by replacing $A$ with $A^{(j)}$, e.g. $G^{(j)}(z)=(W^{(j)}_{1}-z)^{-1}$, $\mc{G}^{(j)}_{u}(\eta)=\sigma_{X}\otimes(W^{(j)}_{1}-u)-i\eta$, etc. Recall that the resolvents $F,\mc{F},$ and $\mbf{F}_{\mbf{z}}(L)$ are defined with respect to $\ac{A}$ and not $A$.

In these coordinates the density has the expression
\begin{align*}
    \rho\big(&M(\mbf{v}_{1},\mbf{w}_{1},...,\mbf{v}_{k},\mbf{w}_{k},M^{(k)})\big)=\\&\frac{1}{Z_{N,t}}\exp\left\{-\frac{N}{2t}\sum_{j=1}^{k}\left[(x_{j}-\Re a_{j})^{2}+\frac{1}{\tau_{N}}(y_{j}-\Im a_{j})^{2}\right]\right.\\&\left.-\frac{N}{2t}\sum_{j=1}^{k}\left[\frac{1}{2}\left\|\mbf{w}_{j}-\mbf{b}_{j}-\mbf{c}_{j}\right\|^{2}+\frac{1}{2\tau_{N}}\left\|\mbf{w}_{j}-\mbf{b}_{j}+\mbf{c}_{j}\right\|^{2}\right]\right.\\
    &\left.-\frac{N}{2t}\tr\left[(\Re M^{(k)}-\Re A^{(k)})^{2}+\frac{1}{\tau_{N}}(\Im M^{(k)}-\Im A^{(k)})^{2}\right]\right\}.
\end{align*}
We need to insert this expression into \eqref{eq:kpointintegral} and then integrate with respect to $\mbf{v}_{j},\mbf{w}_{j},\,j=1,...,k$ and $M^{(k)}$. The integral with respect to $\mbf{w}_{j}$ is Gaussian:
\begin{align*}
    \int_{\mbb{C}^{N-j}}e^{-\frac{N}{4t}\left(\left\|\mbf{w}_{j}-\mbf{b}_{j}-\mbf{c}_{j}\right\|^{2}+\frac{1}{\tau_{N}}\left\|\mbf{w}_{j}-\mbf{b}_{j}+\mbf{c}_{j}\right\|^{2}\right)}d\mbf{w}_{j}&=\left(\frac{4\pi t\tau_{N}}{N(1+\tau_{N})}\right)^{N-j}e^{-\frac{N}{t(1+\tau_{N})}\left\|\mbf{c}_{j}\right\|^{2}}.
\end{align*}

We rewrite the $\mbf{v}_{j}$ integrals by defining the following probability measures $\nu_{j}$ on $S^{N-j}$:
\begin{align}
    d\nu_{j}(\mbf{v}_{j})&=\left(\frac{N}{\pi t(1+\tau_{N})}\right)^{N-j}\frac{1}{K_{j}(z_{j};A^{(j-1)})}e^{-\frac{N}{t(1+\tau_{N})}\left\|\mbf{c}_{j}\right\|^{2}-\frac{N}{2t}\left[(x_{j}-\Re a_{j})^{2}+\frac{1}{\tau_{N}}(y_{j}-\Im a_{j})^{2}\right]}dS_{N-j}(\mbf{v}_{j}),\label{eq:nuj}
\end{align}
where
\begin{align}
    K_{j}(z;A^{(j-1)})&=\left(\frac{N}{\pi t(1+\tau_{N})}\right)^{N-j}\int_{S^{N-j}}e^{-\frac{N}{t(1+\tau_{N})}\left\|\mbf{c}_{j}\right\|^{2}-\frac{N}{2t}\left[(x_{j}-\Re a_{j})^{2}+\frac{1}{\tau_{N}}(y_{j}-\Im a_{j})^{2}\right]}dS_{N-j}(\mbf{v}_{j}).\label{eq:Kj}
\end{align}
Since $A^{(j-1)}$ depends on $\mbf{v}_{l},\,l=1,...,j-1$, so do $K_{j}$ and $\nu_{j}$, and we must integrate over $\nu_{j}$ in descending order of $j$.

We also define the function $\Psi$ by
\begin{align}
    \Psi(\mbf{z};A^{(k)})&=\left(\frac{N}{2\pi t\sqrt{\tau_{N}}}\right)^{(N-k)^{2}}|\Delta(\mbf{z})|^{2}\int_{\mbf{M}_{N-k}}e^{-\frac{N}{2t}\tr\left[(\Re M^{(k)}-\Re A^{(k)})^{2}+\frac{1}{\tau_{N}}(\Im M^{(k)}-\Im A^{(k)})^{2}\right]}\nonumber\\&\prod_{j=1}^{k}\left|\det(M^{(k)}-z_{j})\right|^{2}\,dM^{(k)},\label{eq:F}
\end{align}
and note that $\Psi$ has the form of an expectation value of a product of characteristic polynomials of a Gauss-divisible matrix:
\begin{align*}
    \Psi(\mbf{z};A^{(k)})&=|\Delta(\mbf{z})|^{2}\mbb{E}_{M^{(k)}_{t}}\left[\prod_{j=1}^{k}\left|\det(M^{(k)}_{t}-z_{j})\right|^{2}\right],
\end{align*}
where
\begin{align*}
    M^{(k)}_{t}&=A^{(k)}+\sqrt{t}B^{(k)}
\end{align*}
and $B^{(k)}$ is a complex Gaussian elliptic matrix of dimension $N-k$ and variance $1/N$.

After these definitions we can write down a more compact expression for $\rho^{(k)}_{N}$:
\begin{align}
    \rho^{(k)}_{N}(\mbf{z})&=c_{N,k}\int_{S^{N-1}}\cdots\int_{S^{N-k}}\left(\prod_{j=1}^{k}K_{j}(z_{j};A^{(j-1)})\right)\Psi(\mbf{z};A^{(k)})d\nu_{k}(\mbf{v}_{k})\cdots d\nu_{1}(\mbf{v}_{1}),\label{eq:rhocompact}
\end{align}
where
\begin{align*}
    c_{N,k}&=\frac{1}{(2\pi)^{k}}\left(\frac{N}{2\pi t\sqrt{\tau_{N}}}\right)^{k}.
\end{align*}
This is the desired integral formula for the $k$-point function. 

The first part of \Cref{thm1} states that there exists a unique $\lambda_{E,t}=u_{E,t}+i\eta_{E,t}$ such that $\eta_{E,t}>0$ and
\begin{align*}
    \lambda_{E,t}&=E+t\langle G(\lambda_{E,t})\rangle.
\end{align*}
This follows by the fixed point argument in Section 3.2 of \cite{erdos_bulk_2010}, where the function
\begin{align*}
    \phi(\lambda)&=\frac{(\lambda-E)^{2}}{2t}+\frac{1}{N}\log\det(W_{1}-\lambda)
\end{align*}
is studied (see also Section 3 of \cite{johansson_universality_2001}). In brief, the conditions \eqref{eq:C1.1}, \eqref{eq:C1.2} and \eqref{eq:C1.3} ensure that the function
\begin{align*}
    f(\lambda)&=E+t\langle G_{\lambda}\rangle
\end{align*}
maps the region $\{\lambda\in\mbb{C}:|\Re\lambda-E|\leq Ct,\,t/C\leq \Im\lambda\leq Ct\}$ to itself and is a contraction. Therefore it has a fixed point $\lambda_{E,t}$ which satisfies
\begin{align}
    t/C\leq\eta_{E,t}&\leq Ct,\label{eq:etaBounds}\\
    |u_{E,t}-E|&\leq Ct.\label{eq:uBounds}
\end{align}
We will repeatedly use the fact that $t=N^{-1+2\epsilon}\geq2N^{-1+\epsilon}$ so $\lambda_{E,t}\in S_{\epsilon}$ and in particular the bounds in \eqref{eq:C3.1'} and \eqref{eq:C3.2'} are valid.

Let $\alpha_{E,t}=\alpha_{N}(\lambda_{E,t})$, $\beta_{E,t}=\beta_{N}(\lambda_{E,t})$ and
\begin{align}
    \tau_{E,t}&=\beta_{E,t}+\frac{N\tau_{N}\eta_{E,t}^{2}}{t}.\label{eq:tauEt}
\end{align}
To simplify notation we suppress the argument of resolvents evaluated at $\lambda_{E,t}$, i.e. $G:=G(\lambda_{E,t})$ and $F=F(\lambda_{E,t})$. Likewise we write $\mc{G}:=\mc{G}_{u_{E,t}}(\eta_{E,t})$ and $\mc{F}:=\mc{F}_{u_{E,t}}(\eta_{E,t})$.

We are interested in the bulk limit and so we fix an $E\in(-2,2)$ and define
\begin{align}
    \wt{\rho}^{(k)}_{N}(\mbf{z})&=\left(\frac{t}{N\eta_{E,t}}\right)^{2k}\rho^{(k)}_{E,t}\left(E+i\sqrt{\tau_{N}}\langle W_{2}\rangle+\frac{t\mbf{z}}{N\eta_{E,t}}\right)\nonumber\\
    &=\int_{S^{N-1}}\cdots\int_{S^{N-k}}\left(\prod_{j=1}^{k}\wt{K}_{j}(z_{j};A^{(j-1)})\right)\wt{\Psi}(\mbf{z};A^{(k)})d\nu_{k}(\mbf{v}_{k})\cdots d\nu_{1}(\mbf{v}_{1}),\label{eq:rhoTilde}
\end{align}
where
\begin{align}
    \wt{K}_{j}(z_{j};A^{(j-1)})&=\frac{1}{2\eta^{2}_{E,t}\sqrt{2\pi\tau_{N}}}K_{j}\left(E+i\sqrt{\tau_{N}}\langle W_{2}\rangle+\frac{tz_{j}}{N\eta_{E,t}}\right),\label{eq:KTilde}\\
    \wt{\Psi}(\mbf{z};A^{(k)})&=\frac{t^{k}}{2^{k/2}\pi^{3k/2}N^{k}}\Psi\left(E+i\sqrt{\tau_{N}}\langle W_{2}\rangle+\frac{t\mbf{z}}{N\eta_{E,t}};A^{(k)}\right).\label{eq:PsiTilde}
\end{align}
For ease of notation we introduce 
\begin{align}
    w_{j}&=E+i\sqrt{\tau_{N}}\langle W_{2}\rangle+\frac{tz_{j}}{N\eta_{E,t}},\quad j=1,...,k.
\end{align}

As it stands the integrand in \eqref{eq:rhoTilde} depends on the integration variables $\mbf{v}_{j}$ in a very complicated way through the matrices $A^{(j)}$. However, these matrices are projections of $A$ onto random subspaces of codimension $j$ and we expect that for most subspaces we can approximate $A^{(j)}$ by $A$. This will greatly simplify the behaviour of the integrand with respect to $\mbf{v}_{j}$, and in particular allow us to compute the asymptotics of $\wt{K}_{j}$ and $\wt{\Psi}$. Therefore our second step is to show that a certain quadratic form in $\mbf{v}_{j}$ concentrates around an explicit fixed quantity with respect to the measure $\nu_{j}$.

To state the relevant concentration property, we define
\begin{align}
    \mc{E}_{j}(C)&=\left\{\left|\left|\frac{2\eta_{E,t}\mbf{v}_{j}^{*}G^{(j-1)}\mbf{v}_{j}}{1-t\langle G^{2}\rangle}\right|^{2}-1\right|<\frac{C\log N}{\sqrt{Nt}}\right\}.\label{eq:Ej}
\end{align}
\begin{lemma}\label{lem:nuConc}
There are constants $C,C'$ such that
\begin{align}
    \nu_{j}\left(S^{N-j}\setminus\mc{E}_{j}(C)\right)&\leq e^{-C'\log^{2}N},\label{eq:conc}
\end{align}
uniformly for $\mbf{v}_{l}\in\mc{E}_{l}(C),\,l=1,...,j-1$.
\end{lemma}
The proof of this lemma is given in \Cref{sec:K}. Henceforth we denote $\mc{E}_{j}:=\mc{E}_{j}(C)$. We note that by \eqref{eq:C1.3} we have $|1-t\langle G^{2}\rangle|>C$ and so in $\mc{E}_{j}$ we have
\begin{align}
    |\mbf{v}_{j}^{*}G^{(j-1)}\mbf{v}_{j}|&\geq\frac{C}{\eta_{E,t}}.\label{eq:minorNorm}
\end{align}
From this we derive the following bounds, which are essentially extensions of the norm bounds in \Cref{lem:normEstimates} to resolvents of $A^{(j)}$.
\begin{lemma}\label{lem:minornorms}
Let $1\leq j\leq k$ and $\mbf{v}_{l}\in\mc{E}_{l},\,l=1,...,j$. We have
\begin{align}
    \left\|(V_{j}^{*}\mc{F}^{(j-1)}V_{j})^{-1}\right\|&\leq C\eta_{E,t},\label{eq:minorNorm1}\\
    \left\|\mc{G}^{(j)}(\sigma_{Y}\otimes\ac{W}^{(j)}_{2})\mc{F}^{(j)}\right\|&\leq\frac{C(N\eta_{E,t})^{1/2n_{\epsilon}}}{\eta_{E,t}^{3/2}},\label{eq:mcGWmcFminorNorm}\\
    \left\|(\mc{G}^{(j)}(\sigma_{Y}\otimes\ac{W}^{(j)}_{2})^{2}\mc{F}^{(j)}\right\|&\leq\frac{C(N\eta_{E,t})^{1/n_{\epsilon}}}{\eta_{E,t}^{3/2}},\label{eq:mcGW2mcFminorNorm}\\
    \left\|\mc{F}^{(j)}(1\otimes\ac{W}^{(j)}_{2})\mc{F}^{(j)}\right\|&\leq\frac{C(N\eta_{E,t})^{1/2n_{\epsilon}}}{\eta_{E,t}^{3/2}},\label{eq:mcFWmcFminorNorm}\\
    \left\|(\mc{F}^{(j)}(1\otimes\ac{W}^{(j)}_{2}))^{2}\mc{F}^{(j)}\right\|&\leq\frac{C(N\eta_{E,t})^{1/n_{\epsilon}}}{\eta_{E,t}^{2}}.\label{eq:mcFW2mcFminorNorm}
\end{align}
For any $L\in\mbf{M}^{sa}_{2k}$ and $\bs\lambda\in\{\lambda_{E,t},\bar{\lambda}_{E,t}\}^{2k}$ we have
\begin{align}
    \left\|(\mbf{V}_{j}^{*}\mbf{F}^{(j-1)}_{\bs\lambda}(L)\mbf{V}_{j})^{-1}\right\|&\leq C\eta_{E,t},\label{eq:minorNorm2}\\
    \left\|\mbf{F}^{(j)}_{\bs\lambda}(L)\right\|&\leq\frac{C}{\eta_{E,t}},\label{eq:mbfFminorNorm}\\
    \left\|\mbf{G}^{(j)}_{\bs\lambda}(L\otimes\ac{W}^{(j)}_{2})\mbf{F}^{(j)}_{\bs\lambda}(L)\right\|&\leq\frac{C(N\eta_{E,t})^{1/2n_{\epsilon}}}{\eta_{E,t}^{3/2}}.\label{eq:mbfGWmbfFminorNorm}
\end{align}
For $|u-u_{E,t}|<\sqrt{\frac{t}{N}}\log N$ we have
\begin{align}
    \left\|\sqrt{H^{(j)}_{u}}\ac{W}^{(j)}_{2}\sqrt{H^{(j)}_{u}}\right\|&\leq\frac{C(Nt)^{1/4}}{t}.\label{eq:HWnorm}
\end{align}
\end{lemma}
\begin{proof}
First we note that
\begin{align}
    \left\|(V_{l}^{*}\mc{G}^{(l-1)}V_{l})^{-1}\right\|&=\frac{1}{|\mbf{v}_{l}^{*}G^{(l-1)}\mbf{v}_{l}|^{2}}\left\|\begin{pmatrix}-i\Im\mbf{v}_{l}^{*}G^{(l-1)}\mbf{v}_{l}&\Re\mbf{v}_{l}^{*}G^{(l-1)}\mbf{v}_{l}\\\Re\mbf{v}_{l}^{*}G^{(l-1)}\mbf{v}_{l}&-i\Im\mbf{v}_{l}^{*}G^{(l-1)}\mbf{v}_{l}\end{pmatrix}\right\|\nonumber\\
    &\leq\frac{1}{|\mbf{v}_{l}^{*}G^{(l-1)}\mbf{v}_{l}|}\nonumber\\
    &\leq C\eta_{E,t},\label{eq:VmcGVnorm}
\end{align}
by \eqref{eq:minorNorm}. We will prove the following implications:
\begin{enumerate}
\item
\begin{align}
    &\left\|\mc{G}^{(l-1)}(\sigma_{Y}\otimes\ac{W}^{(l-1)}_{2})\mc{F}^{(l-1)}\right\|\leq\frac{C(N\eta_{E,t})^{1/2n_{\epsilon}}}{\eta_{E,t}^{3/2}}\nonumber\\&\Rightarrow\left\|(V_{l}^{*}\mc{F}^{(l-1)}V_{l})^{-1}\right\|\leq C\eta_{E,t};\label{eq:i1}
\end{align}
\item
\begin{align}
    &\left\|(V_{l}^{*}\mc{F}^{(l-1)}V_{l})^{-1}\right\|\leq C\eta_{E,t}\nonumber\\&\Rightarrow\left\|\mc{G}^{(l)}(\sigma_{Y}\otimes\ac{W}^{(l)}_{2})\mc{F}^{(l)}\right\|\leq C\left\|\mc{G}^{(l-1)}(\sigma_{Y}\otimes\ac{W}^{(l-1)}_{2})\mc{F}^{(l-1)}\right\|.\label{eq:i2}
\end{align}
\end{enumerate}
Then \eqref{eq:minorNorm1} and \eqref{eq:mcGWmcFminorNorm} follow by induction and the fact that for $l=1$ we have
\begin{align*}
    \left\|\mc{G}(\sigma_{Y}\otimes\ac{W}_{2})\mc{F}\right\|&\leq\frac{C(N\eta_{E,t})^{1/2n_{\epsilon}}}{\eta_{E,t}^{3/2}},
\end{align*}
by \eqref{eq:mcGWmcFnorm}. 

The first implication follows by the resolvent identity and \eqref{eq:VmcGVnorm}:
\begin{align*}
    \left\|(V_{l}^{*}\mc{F}^{(l-1)}V_{l})^{-1}\right\|&\leq\left\|(V_{l}^{*}\mc{G}^{(l-1)}V_{l})^{-1}\right\|\\
    &\times\left\|\left(1+i\sqrt{\tau_{N}}(V_{l}^{*}\mc{G}^{(l-1)}V_{l})^{-1}V_{l}^{*}\mc{G}^{(l-1)}(\sigma_{Y}\otimes\ac{W}^{(l-1)}_{2})\mc{F}^{(l-1)}V_{l}\right)^{-1}\right\|\\
    &\leq \frac{C\eta_{E,t}}{1-C\sqrt{\tau_{N}}\eta_{E,t}\left\|\mc{G}^{(l-1)}(\sigma_{Y}\otimes\ac{W}^{(l-1)}_{2})\mc{F}^{(l-1)}\right\|}\\
    &\leq\frac{C\eta_{E,t}}{1-C/(N\eta_{E,t})^{\frac{1}{2}\left(1-\frac{1}{n_{\epsilon}}\right)}}\\
    &\leq C\eta_{E,t}.
\end{align*}

To prove the second implication, we note that by \Cref{lem:minorresolvent} there is a unitary $U_{l}$ such that
\begin{align}
    U_{l}\begin{pmatrix}0&0\\0&\mc{G}^{(l)}\end{pmatrix}U_{l}^{*}&=\mc{G}^{(l-1)}-\mc{G}^{(l-1)}V_{l}(V_{l}^{*}\mc{G}^{(l-1)}V_{l})^{-1}V_{l}^{*}\mc{G}^{(l-1)},\label{eq:mcGminor}\\
    U_{l}\begin{pmatrix}0&0\\0&\mc{F}^{(l)}\end{pmatrix}U_{l}^{*}&=\mc{F}^{(l-1)}-\mc{F}^{(l-1)}V_{l}(V_{l}^{*}\mc{F}^{(l-1)}V_{l})^{-1}V_{l}^{*}\mc{F}^{(l-1)}.\label{eq:mcFminor}
\end{align}
Then we have
\begin{align*}
    \left\|\mc{G}^{(l)}(\sigma_{Y}\otimes\ac{W}^{(l)}_{2})\mc{F}^{(l)}\right\|&\leq\left\|1-\mc{G}^{(l-1)}V_{l}(V_{l}^{*}\mc{G}^{(l-1)}V_{l})^{-1}V_{l}^{*}\right\|\\
    &\times\left\|1-\mc{F}^{(l-1)}V_{l}(V_{l}^{*}\mc{F}^{(l-1)}V_{l})^{-1}V_{l}^{*}\right\|\\
    &\times\left\|\mc{G}^{(l-1)}(\sigma_{Y}\otimes\ac{W}^{(l-1)}_{2})\mc{F}^{(l-1)}\right\|.
\end{align*}
By \eqref{eq:VmcGVnorm} and the bound $\|\mc{G}^{(l-1)}\|\leq1/\eta_{E,t}$ we have
\begin{align*}
    \left\|1-\mc{G}^{(l-1)}V_{l}(V_{l}^{*}\mc{G}^{(l-1)}V_{l})^{-1}V_{l}^{*}\right\|&\leq C.
\end{align*}
Similarly, by the assumption in \eqref{eq:i2} and the bound $\|\mc{F}^{(l-1)}\|\leq 1/\eta_{E,t}$ we have
\begin{align*}
    \left\|1-\mc{F}^{(l-1)}V_{l}(V_{l}^{*}\mc{F}^{(l-1)}V_{l})^{-1}V_{l}^{*}\right\|&\leq C,
\end{align*}
which proves \eqref{eq:i2}.

The bound in \eqref{eq:mcGW2mcFminorNorm} follows in a similar way. Let $x^{(l)}$ denote the right hand side of \eqref{eq:mcGW2mcFminorNorm}; we get an extra term from the second factor of $\mc{G}^{(l)}$:
\begin{align*}
    x^{(j)}&\leq Cx^{(j-1)}+C\left\|\mc{G}^{(l-1)}(\sigma_{Y}\otimes\ac{W}^{(l-1)}_{2})\mc{G}^{(l-1)}V_{l}(V_{l}^{*}\mc{G}^{(l-1)}V_{l})^{-1}V_{l}^{*}\mc{G}^{(l-1)}(\sigma_{Y}\otimes\ac{W}^{(l-1)}_{2})\mc{F}^{(l-1)}\right\|\\
    &\leq Cx^{(j-1)}+C\eta_{E,t}\left\|\mc{G}^{(l-1)}(\sigma_{Y}\otimes\ac{W}^{(l-1)}_{2})\mc{G}^{(l-1)}\right\|\cdot\left\|\mc{G}^{(l-1)}(\sigma_{Y}\otimes\ac{W}^{(l-1)}_{2})\mc{F}^{(l-1)}\right\|.
\end{align*}
For the first factor in the second term we use the resolvent identity:
\begin{align*}
    \left\|\mc{G}^{(l-1)}(\sigma_{Y}\otimes\ac{W}^{(l-1)}_{2})\mc{G}^{(l-1)}\right\|&\leq\left\|\mc{G}^{(l-1)}(\sigma_{Y}\otimes\ac{W}^{(l-1)}_{2})\mc{F}^{(l-1)}\right\|+\sqrt{\tau_{N}}x^{(j-1)}.
\end{align*}
Bounding $\|\mc{G}^{(l-1)}(\sigma_{Y}\otimes\ac{W}^{(l-1)}_{2})\mc{F}^{(l-1)}\|$ by \eqref{eq:mcGWmcFminorNorm} we arrive at
\begin{align*}
    x^{(j)}&\leq Cx^{(j-1)}+\frac{C(N\eta_{E,t})^{1/n_{\epsilon}}}{\eta_{E,t}^{2}}.
\end{align*}
Since
\begin{align*}
    \left\|(\mc{G}(\sigma_{Y}\otimes\ac{W}_{2}))^{2}\mc{F}\right\|&\leq\frac{C(N\eta_{E,t})^{1/n_{\epsilon}}}{\eta_{E,t}^{2}}
\end{align*}
by \eqref{eq:mcGW2mcFnorm}, \eqref{eq:mcGW2mcFminorNorm} follows by induction.

For \eqref{eq:minorNorm2},\eqref{eq:mbfFminorNorm} and \eqref{eq:mbfGWmbfFminorNorm} the proof follows along similar lines. We note that by \eqref{eq:minorNorm},
\begin{align}
    \left\|(\mbf{V}_{l}^{*}\mbf{G}^{(l-1)}_{\bs\lambda}\mbf{V}_{l})^{-1}\right\|&=\frac{1}{|\mbf{v}_{l}^{*}G^{(l-1)}\mbf{v}_{l}|}\nonumber\\
    &\leq C\eta_{E,t},\label{eq:VmbfGVnorm}
\end{align}
and modify the implications to the following:
\begin{enumerate}
\item
\begin{align}
    &\left\|\mbf{G}^{(l-1)}_{\bs\lambda}(L\otimes\ac{W}^{(l-1)}_{2})\mbf{F}^{(l-1)}_{\bs\lambda}(L)\right\|\leq\frac{C(N\eta_{E,t})^{1/2n_{\epsilon}}}{\eta_{E,t}^{3/2}}\nonumber\\
    &\Rightarrow\left\|(\mbf{V}_{l}^{*}\mbf{F}_{\bs\lambda}(L)\mbf{V}_{l})^{-1}\right\|\leq C\eta_{E,t},\quad\left\|\mbf{F}^{(l-1)}_{\bs\lambda}(L)\right\|\leq\frac{C}{\eta_{E,t}};\label{eq:i21}
\end{align}
\item
\begin{align}
    &\left\|(\mbf{V}_{l}^{*}\mbf{F}^{(l-1)}_{\bs\lambda}(L)\mbf{V}_{l})^{-1}\right\|\leq C\eta_{E,t},\quad\left\|\mbf{F}^{(l-1)}_{\bs\lambda}(L)\right\|\leq\frac{C}{\eta_{E,t}}\nonumber\\
    &\Rightarrow\left\|\mbf{G}^{(l)}_{\bs\lambda}(L\otimes\ac{W}^{(l)}_{2})\mbf{F}^{(l)}_{\bs\lambda}(L)\right\|\leq C\left\|\mbf{G}^{(l-1)}_{\bs\lambda}(L\otimes\ac{W}^{(l-1)}_{2})\mbf{F}^{(l-1)}_{\bs\lambda}(L)\right\|.\label{eq:i22}
\end{align}
\end{enumerate}
For the base case we use \eqref{eq:mbfGWmbfFnorm}.

Now consider \eqref{eq:HWnorm}. Recall that $H^{(j)}_{u}=\left(\eta_{E,t}^{2}+|\ac{A}^{(j)}-u|^{2}\right)^{-1}$ and let $M^{(j)}=\sqrt{H^{(j)}_{u}}\ac{W}^{(j)}_{2}\sqrt{H^{(j)}_{u}}$. Since $M^{(j)}$ is Hermitian we have
\begin{align*}
    \|M^{(j)}\|&\leq\left(\tr\left(M^{(j)}\right)^{4}\right)^{1/4}.
\end{align*}
In terms of $\mc{F}^{(j)}$ we have
\begin{align*}
    x^{(j)}&:=\tr(M^{(j)})^{4}\leq\tr\left(\mc{F}^{(j)}_{u}\left(1\otimes \ac{W}^{(j)}_{2}\right)\mc{F}^{(j)*}_{u}\right)^{4}.
\end{align*}
Let us rewrite \eqref{eq:mcFminor} as 
\begin{align*}
    U_{l}\begin{pmatrix}0&0\\0&\mc{F}^{(l)}_{u}\end{pmatrix}U_{l}^{*}&=(1-D^{(l-1)})\mc{F}^{(l-1)}_{u},
\end{align*}
where $D^{(l-1)}=\mc{F}_{u}^{(l-1)}V_{l}(V_{l}^{*}\mc{F}^{(l-1)}_{u}V_{l})^{-1}V_{l}^{*}$. We replace $\mc{F}^{(j)}_{u}$ by the right hand side above, $1\otimes \ac{W}^{(j)}_{2}$ by $1\otimes \ac{W}^{(j-1)}_{2}$ and use the inequality $\tr(XY)^{4}\leq\|Y\|^{4}\tr X^{4}$ for Hermitian $X=\mc{F}^{(j-1)}_{u}(1\otimes W^{(j-1)}_{2})\mc{F}^{(j-1)*}_{u}$ and positive $Y=(1-D^{(j-1)})(1-D^{(j-1)})^{*}$ to obtain
\begin{align*}
    x^{(j)}&\leq\tr\left((1-D^{(j-1)})\mc{F}^{(j-1)}_{u}(1\otimes W^{(j-1)}_{2})\mc{F}^{(j-1)*}_{u}(1-D^{(j-1)*})\right)^{4}\\
    &\leq(1+\|D^{(j-1)}\|)^{8}x^{(j-1)}.
\end{align*}
We can extend \eqref{eq:minorNorm1} to all $u$ satisfying $|u-u_{E,t}|<\sqrt{\frac{t}{N}}\log N$ by the resolvent identity:
\begin{align*}
    \|(V_{j}^{*}\mc{F}^{(j-1)}_{u}V_{j})^{-1}\|&=\|(1+(u-u_{E,t})(V_{j}^{*}\mc{F}^{(j-1)}V_{j})^{-1}V_{j}^{*}\mc{F}^{(j-1)}\mc{F}^{(j-1)}_{u}V_{j})^{-1}(V_{j}^{*}\mc{F}^{(j-1)}V_{j})^{-1}\|\\
    &\leq\frac{C\eta_{E,t}}{1-\frac{C|u-u_{E,t}|}{\eta_{E,t}}}\\&\leq C\eta_{E,t}.
\end{align*}
Therefore we find $\|D^{(j-1)}\|\leq C$ and hence $x^{(j)}\leq Cx^{(0)}$. Estimating $x^{(0)}$ using \eqref{eq:traceEstimate} we obtain 
\begin{align*}
    x^{(j)}&\leq C\tr(G\ac{W}_{2}G^{*})^{4}\\
    &\leq\frac{CN}{t^{5}}.
\end{align*}
\end{proof}

We need to estimate the integral over the complement of the sets $\mc{E}_{j}$, which we do by bounding the integrand uniformly in $\mbf{v}_{l},\,l=1,...,k$.
\begin{lemma}\label{lem:PsiBound}
Let $0<\delta<1$. Then
\begin{align}
    \wt{\Psi}(\mbf{z};A^{(k)})&\leq \frac{Ce^{C(Nt\delta+(Nt\delta)^{-1})}}{(Nt\delta^{2})^{k^{2}}}\prod_{j=1}^{k}\det\left(\eta_{E,t}^{2}+|\ac{A}^{(j-1)}-u_{E,t}|^{2}\right)e^{\frac{N}{t(1-\tau_{N})}\Re(u_{E,t}-E)^{2}-\frac{N}{t}\eta_{E,t}^{2}}.\label{eq:PsiBound}
\end{align}
\end{lemma}

\begin{lemma}\label{lem:KBound}
We have
\begin{align}
    \wt{K}_{j}(z;A^{(j-1)})&\leq \frac{C}{\sqrt{t}}\det^{-1}\left(\eta_{E,t}^{2}+|\ac{A}^{(j-1)}-u_{E,t}|^{2}\right)e^{-\frac{N}{t(1-\tau_{N})}\Re(u_{E,t}-E)^{2}+\frac{N}{t}\eta_{E,t}^{2}}.\label{eq:Kbound}
\end{align}
\end{lemma}

Combining these lemmas we obtain for any $0<\delta<1$ the bound
\begin{align*}
    \left(\prod_{j=1}^{k}\wt{K}_{j}(z_{j};A^{(j-1)})\right)\wt{\Psi}(\mbf{z};A^{(k)})&\leq\frac{Ce^{C(Nt\delta+(Nt\delta)^{-1})}}{(Nt\delta^{2})^{k^{2}}\sqrt{t}}.
\end{align*}
Choosing $\delta=1/Nt$ we obtain
\begin{align}
    \left(\prod_{j=1}^{k}\wt{K}_{j}(z_{j};A^{(j-1)})\right)\wt{\Psi}(\mbf{z};A^{(k)})&\leq \frac{C(Nt)^{k^{2}}}{\sqrt{t}}.\label{eq:upper}
\end{align}
Since this is indepedent of $\mbf{v}_{j},\,j=1,...,k$, and grows as $N^{C}$ for some fixed $C$, the integral over $S^{N-j}\setminus\mc{E}_{j}$ is dominated by the $O(e^{-C\log^{2}N})$ decay of the probability. This allows us to restrict the integral to the sets $\mc{E}_{j}$ (we do this step by step, starting from $j=1$):
\begin{align}
    \wt{\rho}^{(k)}_{N}(\mbf{z};A)&=\int_{\mc{E}_{1}}\cdots\int_{\mc{E}_{k}}\left(\prod_{j=1}^{k}\wt{K}_{j}(z_{j};A^{(j-1)})\right)\wt{\Psi}(\mbf{z};A^{(k)})d\nu_{k}(\mbf{v}_{k})\cdots d\nu_{1}(\mbf{v}_{1})+O(e^{-C\log^{2}N}).\label{eq:rhoRestricted}
\end{align}

When $\mbf{v}_{j}\in\mc{E}_{j}$, we can use the properties that we derived in \Cref{lem:minornorms} to calculate the asymptotics of $\wt{\Psi}$ and $\wt{K}_{j}$. For this we define the following functions for $j=1,...,k$:
\begin{align}
    \psi_{j}(A^{(j-1)})&=|1-t\langle G^{2}\rangle|\left|\det(W_{1}-\lambda_{E,t})\right|^{2}\prod_{l=1}^{j-1}|\mbf{v}_{l}^{*}G^{(l-1)}\mbf{v}_{l}|^{2}\nonumber\\&\exp\left\{\alpha_{E,t}+\beta_{E,t}+\frac{N}{t(1-\tau_{N})}\Re(u_{E,t}-w_{j})^{2}-\frac{N\eta_{E,t}^{2}}{t(1+\tau_{N})}\right\}.\label{eq:psi_j}
\end{align}
The asymptotics are given in the following two lemmas.
\begin{lemma}\label{lem:PsiAsymp}
We have
\begin{align}
    \wt{\Psi}\left(\mbf{z};A^{(k)}\right)&=\left[1+O\left(\frac{1}{(Nt)^{\frac{1}{4}}}\right)\right]\det\left[\frac{1}{(2\pi)^{3/2}}\int_{-1}^{1}e^{-2\tau_{E,t}\lambda^{2}-i(z_{j}-\bar{z}_{l})\lambda}d\lambda\right]\nonumber\\&\prod_{j=1}^{k}\left|\frac{2\eta_{E,t}}{1-t\langle G^{2}\rangle}\mbf{v}_{j}^{*}G^{(j-1)}\mbf{v}_{j}\right|^{2j}\psi_{j}(A^{(j-1)}),\label{eq:PsiAsymp}
\end{align}
uniformly in $\mbf{v}_{j}\in\mc{E}_{j},\,j=1,...,k$ and compact subsets of $\mbb{C}^{k}$.
\end{lemma}

\begin{lemma}\label{lem:KAsymp}
We have
\begin{align}
    \wt{K}_{j}(z_{j};A^{(j-1)})&=\left[1+O\left(\frac{1}{(Nt)^{\frac{1}{4}}}\right)\right]\frac{1}{\psi_{j}(A^{(j-1)})}\frac{e^{-\frac{1}{2\tau_{E,t}}y_{j}^{2}}}{\sqrt{\tau_{E,t}}},\label{eq:KAsymp}
\end{align}
uniformly in $\mbf{v}_{j}\in\mc{E}_{j},\,j=1,...,k$ and compact subsets of $\mbb{C}$.
\end{lemma}

Inserting these asymptotics into \eqref{eq:rhoRestricted} we find
\begin{align*}
    \wt{\rho}^{(k)}_{N}(\mbf{z};A)&=\left[1+O\left(\frac{1}{(Nt)^{\frac{1}{4}}}\right)\right]\rho^{(k)}_{\tau_{E,t}}(\mbf{z})\\
    &\times\int_{\mc{E}_{1}}\left|\frac{2\eta_{E,t}\mbf{v}_{1}^{*}G^{(0)}\mbf{v}_{1}}{1-t\langle G^{2}\rangle}\right|^{2}\cdots\int_{\mc{E}_{k}}\left|\frac{2\eta_{E,t}\mbf{v}_{k}^{*}G^{(k-1)}\mbf{v}_{k}}{1-t\langle G^{2}\rangle}\right|^{2k}d\nu_{k}(\mbf{v}_{k})\cdots d\nu_{1}(\mbf{v}_{1}),
\end{align*}
where $\rho^{(k)}_{\tau}$ was defined in \eqref{eq:rhoTau}. By \Cref{lem:nuConc} we have
\begin{align*}
    \int_{\mc{E}_{j}}\left|\frac{2\eta_{E,t}\mbf{v}_{j}^{*}G^{(j-1)}\mbf{v}_{j}}{1-t\langle G^{2}\rangle}\right|^{2j}d\nu_{j}(\mbf{v}_{j})&=1+O\left(\frac{\log N}{\sqrt{Nt}}\right),
\end{align*}
and so
\begin{align*}
    \wt{\rho}^{(k)}_{N}(\mbf{z};A)&=\left[1+O\left(\frac{1}{(Nt)^{\frac{1}{4}}}\right)\right]\rho^{(k)}_{\tau_{E,t}}(\mbf{z}).
\end{align*}
Since we have assumed that $\tau_{E,t}\to\tau_{E}$, it is clear from the expression for $\rho^{(k)}_{\tau_{E,t}}$ in \eqref{eq:rhoTau} that $\rho^{(k)}_{\tau_{E,t}}\to\rho^{(k)}_{\tau_{E}}$ uniformly in compact subsets of $\mbb{C}^{k}$. In \Cref{sec:Psi} we prove \Cref{lem:PsiBound,lem:PsiAsymp} and in \Cref{sec:K} we prove \Cref{lem:nuConc,lem:KBound,lem:KAsymp}.

\section{Analysis of $\wt{\Psi}(\mbf{z};A^{(k)})$}\label{sec:Psi}
The proofs of \Cref{lem:PsiBound} and \Cref{lem:PsiAsymp} will be based on the following representation of $F$ as an integral over $k\times k$ matrices, which is obtained by the supersymmetry method. This falls into the class of so-called ``duality formulas" which have been proved for expectation values of products of characteristic polynomials for various ensembles (see e.g. \cite{nishigaki_replica_2002,grela_diffusion_2016,liu_phase_2022} for the Ginibre ensemble). For this we introduce the matrices
\begin{align}
    Z&=\diag(\mbf{z},\bar{\mbf{z}}),\label{eq:Z}\\
    L&=\diag(1_{k},-1_{k}).\label{eq:L}
\end{align}
\begin{lemma}\label{lem:Fsuper}
We have the identity
\begin{align}
    \wt{\Psi}\left(\mbf{z};A^{(k)}\right)&=d_{N,k}|\Delta(\mbf{z})|^{2}e^{\frac{N}{t(1-\tau_{N})}\sum_{j=1}^{k}\Re(u_{E,t}-w_{j})^{2}}\int_{\mbf{M}^{sa}_{2k}}e^{-\frac{N}{2t}\tr Q^{2}}D^{(k)}(Q)g(Q)dQ,\label{eq:Fsuper}
\end{align}
where
\begin{align}
    D^{(k)}(Q)&=\det\left[1_{2k}\otimes W_{1}^{(k)}+i\sqrt{\tau_{N}}L\otimes \ac{W}^{(k)}_{2}-(u_{E,t}+iQ)\otimes1_{N-k}\right],\label{eq:Dk}\\
    g(Q)&=\exp\left\{-\frac{iN}{t(1-\tau_{N})}\tr\left(E-u_{E,t}+\frac{t}{N\eta_{E,t}}Z\right)Q\right.\nonumber\\&\left.-\frac{N\tau^{2}_{N}}{2t(1-\tau^{2}_{N})}\tr Q^{2}-\frac{N\tau_{N}}{2t(1-\tau_{N}^{2})}\tr(LQ)^{2}\right\},\label{eq:g}
\end{align}
and
\begin{align}
    d_{N,k}&=\frac{N^{k^{2}}}{(2\pi)^{3k/2}\pi^{2k^{2}}t^{k^{2}}\eta_{E,t}^{k(k-1)}(1-\tau_{N}^{2})^{k^{2}}}.\label{eq:dNk}
\end{align}
\end{lemma}
\begin{proof}
Recall the representation of $\Psi$ in terms of an expectation:
\begin{align*}
    \Psi(\mbf{z};A^{(k)})&=|\Delta(\mbf{z})|^{2}\mbb{E}_{B^{(k)}}\left[\prod_{j=1}^{k}\left|\det(A^{(k)}+B^{(k)}-z_{j})\right|^{2}\right],
\end{align*}
where $B^{(k)}$ has density proportional to
\begin{align*}
    e^{-\frac{N}{2t}\tr(\Re B^{(k)})^{2}-\frac{N}{2t\tau_{N}}\tr(\Im B^{(k)})^{2}}.
\end{align*}
We write each determinant as an integral over anti-commuting variables:
\begin{align*}
    \Psi(\mbf{z};A^{(k)})&=|\Delta(\mbf{z})|^{2}\int e^{-\sum_{j=1}^{k}\left[\zeta_{j}^{*}(A^{(k)}-z_{j})\zeta_{j}+\chi_{j}^{*}(A^{(k)*}-\bar{z}_{j})\chi_{j}\right]}\\
    &\times\mbb{E}_{B^{(k)}}\left[e^{\tr B^{(k)}\sum_{j=1}^{k}\zeta_{j}\zeta_{j}^{*}+\tr B^{(k)*}\sum_{j=1}^{k}\chi_{j}\chi_{j}^{*}}\right]d\zeta d\chi.
\end{align*}
The expectation over $B^{(k)}$ can be calculated using the identity
\begin{align*}
    \mbb{E}_{B^{(k)}}\left[e^{\tr B^{(k)} X+\tr B^{(k)*}Y}\right]&=e^{\frac{t}{2N}\tr\left[(1-\tau_{N})(X^{2}+Y^{2})+2(1+\tau_{N})XY\right]}.
\end{align*}
Setting $X=\sum_{j=1}^{k}\zeta_{j}\zeta_{j}^{*}$ and $Y=\sum_{j=1}^{k}\chi_{j}\chi_{j}^{*}$ we find
\begin{align*}
    \mbb{E}_{B^{(k)}}\left[e^{\tr B^{(k)}X+\tr B^{(k)*}Y}\right]&=e^{-\frac{t}{2N}\tr\left[(1-\tau_{N})(\sigma_{\zeta}^{2}+\sigma_{\chi}^{2})+2(1+\tau_{N})\sigma_{\zeta,\chi}^{*}\sigma_{\zeta,\chi}\right]},
\end{align*}
where $\sigma_{\zeta,ij}=\zeta_{i}^{*}\zeta_{j},\,\sigma_{\chi,ij}=\chi_{i}^{*}\chi_{j}$ and $\sigma_{\zeta,\chi,ij}=\zeta_{i}^{*}\chi_{j}$. Introducing auxilliary integration variables $Q_{11},\,Q_{22}\in\mbf{M}^{sa}_{k}$ and $Q_{12}\in\mbf{M}_{k}$ we have
\begin{align*}
    &e^{-\frac{t}{2N}\tr\left[(1+\tau_{N})(\sigma_{\zeta}^{2}+\sigma_{\chi}^{2})+(1-\tau_{N})\sigma_{\zeta,\chi}^{*}\sigma_{\zeta,\chi}\right]}=\frac{1}{2^{k}}\left(\frac{N}{\pi t\sqrt{1-\tau^{2}_{N}}}\right)^{2k^{2}}\\&\int_{\mbf{M}^{sa}_{k}}\int_{\mbf{M}^{sa}_{k}}\int_{\mbf{M}_{k}}\exp\left\{-\frac{N}{2t(1-\tau_{N})}\tr(Q_{11}^{2}+Q_{22}^{2})-\frac{N}{t(1+\tau_{N})}\tr|Q_{12}|^{2}\right.\\&\left.+i\tr(Q_{11}\sigma_{\zeta}+Q_{22}\sigma_{\chi}+Q_{12}\sigma_{\zeta,\chi}^{*}+Q_{12}^{*}\sigma_{\zeta,\chi})\right\}dQ_{12}dQ_{11}dQ_{22}.
\end{align*}
Collecting the anti-commuting variables into a length $2k(N-k)$ vector $\bs\xi=(\zeta_{1},...,\zeta_{k},\chi_{1},...,\chi_{k})^{T}$, the exponent becomes
\begin{align*}
    \bs\xi^{*}\begin{pmatrix} 1_{k}\otimes A^{(k)}-(\mbf{z}+iQ_{11})\otimes 1_{N-k}&-iQ_{12}\otimes 1_{N-k}\\-iQ_{12}^{*}\otimes1_{N-k}&1_{k}\otimes A^{(k)*}-(\bar{\mbf{z}}+iQ_{22})\otimes1_{N-k}\end{pmatrix}\bs\xi.
\end{align*}
Integrating over $\bs\xi$ gives the determinant of the matrix above. Now we change variables $Q_{11}\mapsto Q_{11}-i(u_{E,t}-\mbf{z})$ and $Q_{22}\mapsto Q_{22}-i(u_{E,t}-\bar{\mbf{z}})$ and shift the diagonal elements $Q_{11}$ and $Q_{22}$ back to the real axis using the analyticity of the integrand and the Gaussian decay at infinity. We can then write the integral in terms of the $2k\times 2k$ Hermitian matrix $Q=\begin{pmatrix}Q_{11}&Q_{12}\\Q_{12}^{*}&Q_{22}\end{pmatrix}$. We obtain \eqref{eq:Fsuper} after replacing $\mbf{z}$ with $\mbf{w}$ and rescaling by the constant factor in \eqref{eq:PsiTilde} that defines $\wt{\Psi}$ in terms of $\Psi$.
\end{proof}

\begin{proof}[Proof of \Cref{lem:PsiBound}]
Let
\begin{align}
    \wt{\Psi}_{upper}&=\prod_{j=1}^{k}\det\left(\eta_{E,t}^{2}+|\ac{A}^{(j-1)}-u_{E,t}|^{2}\right)e^{\frac{N}{t(1-\tau_{N})}\Re(u_{E,t}-w_{j})^{2}-\frac{N}{t}\eta_{E,t}^{2}}.\label{eq:PsiUpper}
\end{align}
We need to show that
\begin{align}
    \wt{\Psi}(\mbf{z};A^{(k)})&\leq\frac{e^{C(Nt\delta+(Nt\delta)^{-1})}}{(Nt\delta^{2})^{k^{2}}}\wt{\Psi}_{upper}.\label{eq:PsiBound1}
\end{align}
Partition $Q$ into $k\times k$ blocks $Q_{11},Q_{22}$ and $Q_{12}$ with elements $b_{jl},d_{jl}$ and $c_{jl}$ respectively. We rewrite the determinant in \eqref{eq:Dk} by exchanging row/column $j$ with $j+N-k$ to obtain
\begin{align*}
    D^{(k)}(Q)&=\det\mbf{M}^{(k)}(Q),
\end{align*}
where $\mbf{M}^{(k)}$ is the $k\times k$ block matrix with diagonal blocks
\begin{align*}
    \mbf{M}^{(k)}_{jj}&=\begin{pmatrix}-ic_{jj}&\ac{A}^{(k)}-u_{E,t}-ib_{jj}\\\ac{A}^{(k)*}-u_{E,t}-id_{jj}&-i\bar{c}_{jj}\end{pmatrix},
\end{align*}
and off-diagonal blocks
\begin{align*}
    \mbf{M}^{(k)}_{jl}&=-i\begin{pmatrix}c_{jl}&b_{jl}\\d_{jl}&\bar{c}_{lj}\end{pmatrix}.
\end{align*}

Using Fischer's inequality we can bound the determinant as follows:
\begin{align*}
    \left|\det\mbf{M}^{(k)}\right|^{2}&\leq\prod_{j=1}^{k}\det\left[(\mbf{M}^{(k)}\mbf{M}^{(k)*})_{jj}\right]\\
    &\leq\prod_{j=1}^{k}\det\left(\theta_{j}^{2}+|\ac{A}^{(k)}-u_{E,t}-ib_{jj}|^{2}\right)\det\left(\zeta^{2}_{j}+|\ac{A}^{(k)}-u_{E,t}+id_{jj}|^{2}\right),
\end{align*}
where we have defined
\begin{align}
    \theta_{j}&=\left|c_{jj}\right|^{2}+\sum_{l\neq j}\left(\left|c_{jl}\right|^{2}+\left|b_{jl}\right|^{2}\right),\\
    \zeta_{j}&=\left|c_{jj}\right|^{2}+\sum_{l\neq j}\left(\left|c_{lj}\right|^{2}+\left|d_{jl}\right|^{2}\right).
\end{align}
Consider one factor in the product; using $\det(1+M)\leq\exp\tr M$ for $M>-1$ we have
\begin{align*}
    \frac{\det\left(\theta^{2}_{j}+|\ac{A}^{(k)}-u_{E,t}-ib_{jj}|^{2}\right)}{\det(\eta^{2}+\left|\ac{A}^{(k)}-u_{E,t}\right|^{2})}&=\det\left[1-\begin{pmatrix}i(\theta_{j}-\eta_{j})&ib_{jj}\\-ib_{jj}&i(\theta_{j}-\eta_{j})\end{pmatrix}\mc{F}_{u_{E,t}}^{(k)}(\eta)\right]\\
    &\leq\exp\left\{b_{jj}\Re\tr\mc{F}^{(k)}_{u_{E,t}}(\eta)(\sigma_{Y}\otimes1_{N-k})+\frac{\theta^{2}_{j}-\eta^{2}+b_{jj}^{2}}{2\eta}\Im\tr\mc{F}^{(k)}_{u_{E,t}}(\eta)\right\}.
\end{align*}
For the first term we use \cref{lem:minorresolvent} to replace $A^{(k)}$ with $A$ and then estimate the trace by \cref{eq:traceEstimate}:
\begin{align*}
    \Re\tr\mc{F}^{(k)}_{u_{E,t}}(\eta)(\sigma_{Y}\otimes 1_{N-k})&=\Re\tr\mc{F}_{u_{E,t}}(\eta)(\sigma_{Y}\otimes1_{N})+O(\eta^{-1})\\
    &=O(\eta^{-1}).
\end{align*}

Up to now $\eta$ has been arbitrary. Let $0<\delta<1$ and set $\eta=\sqrt{1+\delta}\eta_{E,t}$, $\lambda=u_{E,t}+i\eta$. Then since $\frac{t}{N}\tr\left|G\right|^{2}=1$ by definition (recall that $G=G_{\lambda_{E,t}}$) we have
\begin{align*}
    1-\frac{t}{2N\eta}\Im\tr\mc{G}_{u_{E,t}}(\eta)&=\frac{\delta\eta_{E,t}^{2}t}{N}\tr\big(1+\delta\eta_{E,t}^{2}\left|G\right|^{2}\big)^{-1}\left|G\right|^{4}\\
    &\geq\frac{\delta t}{N(1+\delta)}\tr(\Im G)^{2}\\
    &=\frac{\delta}{2(1+\delta)}+O(\delta t).
\end{align*}
By interlacing and \cref{eq:traceEstimate} we then have
\begin{align*}
    \frac{t}{2N\eta}\Im\tr\mc{F}^{(k)}_{u_{E,t}}(\eta)&=\frac{t}{2N\eta}\Im\tr\mc{F}_{u_{E,t}}(\eta)+O\left(\frac{t}{N\eta^{2}}\right)\\
    &=\frac{t}{2N\eta}\Im\tr\mc{G}_{u_{E,t}}(\eta)+O\left(\frac{t}{N\eta^{2}}\right)\\
    &\leq\frac{1-\delta/8}{Nt},
\end{align*}
for sufficiently large $N$. We obtain the bound
\begin{align*}
    \frac{\det(\theta^{2}_{j}+\left|\ac{A}^{(k)}-u_{E,t}-ib_{jj}\right|^{2})}{\det(\eta^{2}+\left|\ac{A}^{(k)}-u_{E,t}\right|^{2})}&\leq C\exp\left\{-\eta\tr|G_{\lambda}|^{2}+\frac{N(1-\delta/8)}{t}(\theta^{2}_{j}+b_{jj}^{2})+\frac{C}{t}\left|b_{jj}\right|\right\}.
\end{align*}
Repeating this for the other factors we obtain the following bound on the determinant:
\begin{align*}
    \left|\det\mbf{M}^{(k)}\right|&\leq \det^{k}\left(\eta^{2}+|\ac{A}^{(k)}-u_{E,t}|^{2}\right)e^{-k\eta^{2}\tr\left|G_{\lambda}\right|^{2}}\\
    &\times\exp\left\{\frac{N(1-\delta/8)}{2t}\tr Q^{2}+\frac{C}{t}\sum_{j=1}^{k}(\left|b_{jj}\right|+\left|d_{jj}\right|)\right\}.
\end{align*}
Since $\tau_{N}=O(N^{-1})$ and $\left\|Z\right\|\leq C$, we have
\begin{align*}
    \log|g(Q)|&=\frac{\tr Q\Im Z}{\eta_{E,t}(1-\tau_{N})}-\frac{N\tau^{2}_{N}}{2t(1-\tau^{2}_{N})}\tr Q^{2}-\frac{N\tau_{N}}{2t(1-\tau^{2}_{N})}\tr(LQ)^{2}\\
    &\leq\frac{C}{t}\sum_{j=1}^{k}\left(|b_{jj}|+|d_{jj}|\right)+\frac{C}{t}\tr Q^{2}.
\end{align*}
Inserting this into the expression for $\wt{\Psi}$ we find
\begin{align*}
    \wt{\Psi}(\mbf{z};A^{(k)})&\leq C\det^{k}\left(\eta^{2}+\left|\ac{A}^{(k)}-u_{E,t}\right|^{2}\right)e^{\frac{Nk}{t(1-\tau_{N})}(u_{E,t}-E)^{2}-k\eta^{2}\tr\left|G_{\lambda}\right|^{2}}\\&\times d_{N,k}\int_{\mbf{M}^{sa}_{2k}}\exp\left\{-\frac{CN\delta}{t}\tr Q^{2}+\frac{C}{t}\sum_{j=1}^{2k}(\left|b_{jj}\right|+\left|d_{jj}\right|)\right\}dQ.
\end{align*}
The second line is bounded by
\begin{align*}
    d_{N,k}\left(\frac{t}{N\delta}\right)^{2k^{2}}e^{C(Nt\delta)^{-1}}&\leq\frac{Ct^{k}}{N^{k^{2}}\delta^{2k^{2}}}e^{C(Nt\delta)^{-1}}.
\end{align*}
In terms of $\wt{\Psi}_{upper}$ we have
\begin{align*}
    \wt{\Psi}(\mbf{z};A^{(k)})&\leq\frac{Ct^{k}e^{C(Nt\delta)^{-1}}}{N^{k^{2}}\delta^{2k^{2}}}\left(\prod_{j=1}^{k}r_{j}\right)\wt{\Psi}_{upper},
\end{align*}
where
\begin{align*}
    r_{j}&:=\frac{\det\left(\eta^{2}+|\ac{A}^{(k)}-u_{E,t}|^{2}\right)}{\det\left(\eta_{E,t}^{2}+|\ac{A}^{(j-1)}-u_{E,t}|^{2}\right)}e^{-\eta^{2}\tr|G_{\lambda}|^{2}+\frac{N}{t}\eta_{E,t}^{2}}.
\end{align*}
We have used the fact that $|\Delta(\mbf{z})|^{2}<C$ on compact sets. Using Cramer's rule we replace $A^{(k)}$ by $A^{(j-1)}$ in the numerator of $r_{j}$:
\begin{align*}
    \frac{\det\left(\eta^{2}+|\ac{A}^{(k)}-u_{E,t}|^{2}\right)}{\det\left(\eta_{E,t}^{2}+|\ac{A}^{(j-1)}-u_{E,t}|^{2}\right)}&=\det\left(1+\delta\eta_{E,t}^{2}H^{(j-1)}\right)\prod_{l=j}^{k}|\det V_{j}^{*}\mc{F}^{(j-1)}_{u_{E,t}}(\eta) V_{j}|.
\end{align*}
Since $\|\mc{F}^{(j-1)}_{u_{E,t}}(\eta)\|\leq C/\eta\leq C/t$ and $V_{j}$ has rank 2 we can bound each factor in the product over $l$ by its norm to find
\begin{align*}
    r_{j}&\leq\frac{C}{t^{2(k-j+1)}}\exp\left\{\delta\eta_{E,t}^{2}\tr H^{(j-1)}-\eta^{2}\tr|G_{\lambda}|^{2}+\frac{N}{t}\eta_{E,t}^{2}\right\},
\end{align*}
where we have used $\det(1+X)\leq e^{\tr X}$ for $X>0$. By interlacing and \eqref{eq:traceEstimate} we find
\begin{align*}
    \tr H^{(j-1)}&=\frac{1}{2\eta_{E,t}}\Im\tr\mc{F}^{(j-1)}\\
    &=\frac{\Im\tr G}{\eta_{E,t}}+O\left(\frac{1}{t^{2}}\right)\\
    &=\frac{N}{t}+O\left(\frac{1}{t^{2}}\right).
\end{align*}
Thus we have
\begin{align*}
    r_{j}&\leq\frac{C}{t^{2(k-j+1)}}\exp\left\{\frac{N\eta^{2}}{t}\left(1-t\langle|G_{\lambda}|^{2}\rangle\right)+O(\delta)\right\}\\
    &=\frac{C}{t^{2(k-j+1)}}\exp\left\{\delta\eta^{2}\eta_{E,t}^{2}\tr\left(1+\delta\eta_{E,t}^{2}|G|^{2}\right)^{-1}|G|^{4}+O(\delta)\right\}\\
    &\leq\frac{C e^{CNt\delta}}{t^{2(k-j+1)}}.
\end{align*}
Taking the product over $j=1,...,k$ we obtain \eqref{eq:PsiBound1}.
\end{proof}

\begin{proof}[Proof of \Cref{lem:PsiAsymp}]
We need to prove that
\begin{align*}
    \wt{\Psi}(\mbf{z};A^{(k)})&=\left[1+O\left(\frac{\log N}{(Nt)^{\frac{1}{2}\left(1-\frac{1}{n_{\epsilon}}\right)}}\right)\right]\wt{\Psi}_{asymp}(\mbf{z};A^{(k)}),
\end{align*}
where
\begin{align}
    \wt{\Psi}_{asymp}(\mbf{z};A^{(k)})&=\det\left[\frac{1}{(2\pi)^{3/2}}\int_{-1}^{1}e^{-2\tau_{E,t}\lambda^{2}-i(z_{j}-\bar{z}_{l})\lambda}d\lambda\right]\nonumber\\&\times\prod_{j=1}^{k}\left|\frac{2\eta_{E,t}\mbf{v}_{j}^{*}G^{(j-1)}\mbf{v}_{j}}{1-t\langle G^{2}\rangle}\right|^{2j}\psi_{j}(A^{(j-1)}),\label{eq:PsiAsymp1}
\end{align}
and $\psi_{j}$ was defined in \eqref{eq:psi_j}.

Let $Q=U\bs\eta U^{*}$ be the eigenvalue decomposition of $Q$. Without loss we assume that $\eta_{j}\geq0$ for $1\leq j\leq k+s$ and $\eta_{j}<0$ for $k+s<j\leq 2k$, i.e. $s$ is the signature of $Q$. Let $\bs\eta_{s}$ be the diagonal matrix whose first $s$ elements are $\eta_{E,t}$ and remaining elements are $-\eta_{E,t}$, and let $\bs\lambda_{s}=u_{E,t}+i\bs\eta_{s}$. Let $R=\bs\eta-\bs\eta_{s}$ and $\mbf{R}^{(j)}=(\bs\eta-\bs\eta_{s})\otimes1_{N-j}$. Let $\mbf{F}^{(j)}_{s}=\mbf{F}^{(j)}_{\bs\lambda_{s}}(U^{*}LU)$ and $\mbf{G}^{(j)}_{s}=\mbf{G}^{(j)}_{\bs\lambda_{s}}$, where $\mbf{G}$ and $\mbf{F}$ are the block resolvents defined in \eqref{eq:mbG} and \eqref{eq:mbF} respectively. Adding and subtracting $\bs\eta_{s}\otimes1_{N-k}$ inside the determinant $D^{(k)}(Q)$ in \eqref{eq:Dk} we find
\begin{align*}
    D^{(k)}(Q)&=\det\left(1_{2k}\otimes W^{(k)}_{1}+i\sqrt{\tau_{N}}U^{*}LU\otimes \ac{W}^{(k)}_{2}-\bs\lambda_{s}\otimes 1_{N-k}\right)\det\left(1-i\mbf{R}^{(k)}\mbf{F}_{s}^{(k)}\right)\\
    &=:D^{(k)}_{1}D^{(k)}_{2}.
\end{align*}

Consider first the second factor $D^{(k)}_{2}$. Taking the square modulus and using the inequality $\log(1+x)\leq x-\frac{3x^{2}}{6+4x}$ for $x>-1$ we find
\begin{align*}
    |D^{(k)}_{2}|^{2}&\leq\exp\left\{\tr M^{(k)}-\frac{3\tr (M^{(k)})^{2}}{6+4\|M^{(k)}\|}\right\},
\end{align*}
where
\begin{align*}
    M^{(k)}&=2\Im(\mbf{R}^{(k)}\mbf{F}_{s}^{(k)})+\mbf{R}^{(k)}\mbf{F}_{s}^{(k)}\mbf{F}_{s}^{(k)*}\mbf{R}^{(k)}.
\end{align*}

To estimate the norm and traces of $M^{(k)}$ we will use \Cref{lem:minorresolvent} to replace $F^{(j)}$ by $F^{(j-1)}$ and \Cref{lem:minornorms} to estimate the resulting errors, recalling that we have assumed $\mbf{v}_{l}\in\mc{E}_{l}$ for $l=1,...,k$. For the norm we find
\begin{align*}
    \|M^{(k)}\|&\leq2\|\bs\eta-\bs\eta_{s}\|\cdot\|\mbf{F}^{(k)}_{s}\|+\|\bs\eta-\bs\eta_{s}\|^{2}\cdot\|\mbf{F}^{(k)}_{s}\|^{2}\\
    &\leq\frac{C\|\bs\eta-\bs\eta_{s}\|}{\eta_{E,t}}+\frac{C\|\bs\eta-\bs\eta_{s}\|^{2}}{\eta_{E,t}^{2}},
\end{align*}
where we have used the bound for $\mbf{F}^{(k)}_{s}$ in \eqref{eq:mbfFminorNorm}. Since the signs of the elements of $\bs\eta$ are the same as those of $\bs\eta_{s}$, we have $|\bs\eta+\bs\eta_{s}|\geq|\bs\eta-\bs\eta_{s}|$ and hence
\begin{align*}
    \|M^{(k)}\|&\leq C\max_{1\leq j\leq 2k}\gamma_{j},
\end{align*}
where $\gamma_{j}=|\eta_{j}^{2}-\eta_{E,t}^{2}|/\eta_{E,t}^{2}$.

Consider
\begin{align*}
    x^{(j)}&=\tr\mbf{R}^{(j)}\mbf{F}^{(j)}\mbf{F}^{(j)*}\mbf{R}^{(j)}.
\end{align*}
Recall that by \Cref{lem:minorresolvent}, there is a unitary $\mbf{U}_{j}$ such that
\begin{align}
    \mbf{U}_{j}\begin{pmatrix}0&0\\0&\mbf{F}_{s}^{(j)}\end{pmatrix}\mbf{U}_{j}^{*}&=\mbf{F}_{s}^{(j-1)}-\mbf{F}_{s}^{(j-1)}\mbf{V}_{j}(\mbf{V}_{j}^{*}\mbf{F}_{s}^{(j-1)}\mbf{V}_{j})^{-1}\mbf{V}_{j}^{*}\mbf{F}_{s}^{(j-1)}.\label{eq:FPert}
\end{align}
Using \eqref{eq:FPert}, \eqref{eq:minorNorm2} and \eqref{eq:mbfFminorNorm} we find
\begin{align*}
    |x^{(j)}-x^{(j-1)}|&\leq\left(2+\|(\mbf{V}_{j}^{*}\mbf{F}_{s}^{(j-1)}\mbf{V}_{j})^{-1}\|\cdot\|\mbf{F}_{s}^{(j-1)}\|\right)\|(\mbf{V}_{j}^{*}\mbf{F}_{s}^{(j-1)}\mbf{V}_{j})^{-1}\|\cdot\|\mbf{F}_{s}^{(j-1)}\|\\&\cdot\|\mbf{F}_{s}^{(j-1)*}(\mbf{R}^{(j-1)})^{2}\mbf{F}_{s}^{(j-1)}\|\\
    &\leq\frac{C}{\eta_{E,t}^{2}}\|\mbf{R}^{(j-1)}\|^{2}\\
    &\leq C\sum_{l=1}^{2k}\gamma_{l}.
\end{align*}
To estimate $x$ we apply \eqref{eq:traceEstimate2}:
\begin{align*}
    x&=\tr\mbf{R}|\mbf{F}_{\mbf{z}}(U^{*}LU)|^{2}\mbf{R}\\
    &=\sum_{j,l=1}^{2k}|R_{jl}|^{2}\tr GG^{*}+O\left(\frac{\|R\|^{2}}{\eta_{E,t}^{2}}\right)\\
    &=\frac{N}{t}\tr R^{2}+O\left(\sum_{j=1}^{2k}\gamma_{j}\right).
\end{align*}

We carry out the same steps to estimate the other traces that appear in $\tr M^{(k)}$ and $\tr(M^{(k)})^{2}$. In general, the error after replacing $\mbf{F}_{s}^{(k)}$ with $\mbf{G}_{s}$ will be $O(\|R\|^{m}/\eta_{E,t}^{m})$ with $m=1,...,4$. This can be bounded in terms of $\gamma_{j}$ using the inequalities
\begin{align*}
    \left(\frac{|\eta_{j}\pm\eta_{E,t}|}{\eta_{E,t}}\right)^{m}&\leq\gamma^{\lceil\frac{m}{2}\rceil}_{j},\quad m=1,...,4\\
    \frac{|\eta_{j}\pm\eta_{E,t}|^{2}}{\eta_{E,t}^{2}}&\leq\gamma^{2}_{j},
\end{align*}
where the $\pm$ sign is taken opposite to the sign of $\eta_{j}$. Therefore, if we let $M_{0}$ denote the matrix obtained from $M$ by replacing $\mbf{F}_{s}$ with $\mbf{G}_{s}$, we have
\begin{align*}
    |\tr M^{(k)}-\tr M_{0}|&\leq C\sum_{j=1}^{2k}\gamma_{j},\\
    |\tr(M^{(k)})^{2}-\tr M^{2}_{0}|&\leq C\sum_{j=1}^{2k}\gamma_{j}^{2},
\end{align*}
We can calculate traces of $M_{0}$ using the definition of $\lambda_{E,t}$. We have
\begin{align*}
    \tr M_{0}&=\tr(Q^{2}-\eta_{E,t}^{2})\tr|G|^{2}=\frac{N}{t}\sum_{j=1}^{2k}(\eta_{j}^{2}-\eta_{E,t}^{2}),
\end{align*}
and
\begin{align*}
    \tr M_{0}^{2}&=\tr(Q^{2}-\eta_{E,t}^{2})^{2}\tr|G|^{4}\geq CNt\sum_{j=1}^{2k}\gamma_{j}^{2}.
\end{align*}
Thus we have obtained the bound
\begin{align*}
    |D^{(k)}_{2}|&\leq C\exp\left\{\frac{N}{2t}\sum_{j=1}^{2k}(\eta_{j}^{2}-\eta_{E,t}^{2})-CNt\sum_{j=1}^{2k}\frac{\gamma_{j}^{2}}{1+\gamma_{j}}\right\}.
\end{align*}
We have used the fact that $x\mapsto\frac{x^{2}}{1+x}$ is increasing on $[0,\infty]$, so that
\begin{align*}
    \frac{1}{1+\max_{j}\gamma_{j}}\sum_{j=1}^{2k}\gamma_{j}^{2}\geq\frac{\max_{j}\gamma_{j}^{2}}{1+\max_{j}\gamma_{j}}=\max_{j}\frac{\gamma^{2}_{j}}{1+\gamma_{j}}\geq\frac{1}{2k}\sum_{j=1}^{2k}\frac{\gamma_{j}^{2}}{1+\gamma_{j}},
\end{align*}
and the fact that
\begin{align*}
    \frac{Nt\gamma^{2}}{1+\gamma}-\gamma\geq C\left(1+\frac{Nt\gamma^{2}}{1+\gamma}\right).
\end{align*}

Now we bound $D^{(j)}_{1}$, whose definition we repeat here:
\begin{align*}
    D^{(j)}_{1}&=\det\left(1_{2k}\otimes W^{(j)}_{1}+i\sqrt{\tau_{N}}U^{*}LU\otimes\ac{W}^{(j)}_{2}-\bs\lambda_{s}\otimes1_{N-j}\right).
\end{align*}
We use Cramer's rule to replace $A^{(j)}$ with $A^{(j-1)}$:
\begin{align*}
    \left|\frac{D^{(j)}_{1}}{D^{(j-1)}_{1}}\right|&=|\det\mbf{V}_{j}^{*}\mbf{F}_{s}^{(j-1)}\mbf{V}_{j}|.
\end{align*}
Using the resolvent identity we find
\begin{align*}
    \frac{\det\mbf{V}_{j}^{*}\mbf{F}^{(j-1)}_{s}\mbf{V}_{j}}{\det\mbf{V}_{j}^{*}\mbf{G}^{(j-1)}_{s}\mbf{V}_{j}}&=\det\left[1+i\sqrt{\tau_{N}}(\mbf{V}_{j}^{*}\mbf{G}^{(j-1)}_{s}\mbf{V}_{j})^{-1}\mbf{V}_{j}^{*}\mbf{G}^{(j-1)}_{s}(U^{*}LU\otimes W^{(j-1)}_{2})\mbf{F}^{(j-1)}_{s}\mbf{V}_{j}\right].
\end{align*}
Bounding the norm of the matrix in the determinant on the right hand side using \eqref{eq:VmbfGVnorm} and \eqref{eq:mbfGWmbfFminorNorm} we find
\begin{align}
    \frac{\det\mbf{V}_{j}^{*}\mbf{F}^{(j-1)}_{s}\mbf{V}_{j}}{\det\mbf{V}_{j}^{*}\mbf{G}^{(j-1)}_{s}\mbf{V}_{j}}&=1+O\left(\frac{1}{(Nt)^{\frac{1}{2}\left(1-\frac{1}{n_{\epsilon}}\right)}}\right).\label{eq:PertRatio}
\end{align}
From this we obtain by induction 
\begin{align*}
    |D^{(k)}_{1}|&\leq CD^{(0)}_{1}\prod_{j=1}^{k}\left|\mbf{V}_{j}^{*}\mbf{G}^{(j-1)}_{s}\mbf{V}_{j}\right|\\
    &= CD^{(0)}_{1}\prod_{j=1}^{k}|\mbf{v}_{j}^{*}G^{(j-1)}\mbf{v}_{j}|^{2k}.
\end{align*}
We factorise $D^{(0)}_{1}$ and use \eqref{eq:det1}:
\begin{align*}
    \left|\frac{D^{(0)}_{1}}{\det^{2k}(W_{1}-w_{E})}\right|&=\left|\det(1+i\sqrt{\tau_{N}}\mbf{G}_{s}(U^{*}LU\otimes \ac{W}_{2}))\right|\leq C.
\end{align*}
Combining the bounds for $D^{(k)}_{1}$ and $D^{(k)}_{2}$ we have
\begin{align*}
    \frac{|D^{(k)}(Q)|}{|\det(W_{1}-w_{E})|^{2k}}&\leq C\exp\left\{\frac{N}{2t}\sum_{j=1}^{2k}(\eta_{j}^{2}-\eta_{E,t}^{2})-CNt\sum_{j=1}^{2k}\frac{\gamma_{j}^{2}}{1+\gamma_{j}}\right\}\prod_{j=1}^{k}|\mbf{v}_{j}^{*}G^{(j-1)}\mbf{v}_{j}|^{2k}.
\end{align*}

Since $|z_{j}|<C$ and $\|\bs\eta_{s}\|<Ct$ we have
\begin{align*}
    |g(Q)|&\leq C\exp\left\{\frac{\|Z\|}{t(1-\tau_{N})}\tr|\bs\eta-\bs\eta_{s}|+\frac{N\tau_{N}}{2t(1-\tau_{N})}\tr|\bs\eta^{2}-\bs\eta_{s}^{2}|\right\}\\
    &\leq C\exp\left\{C\sum_{j=1}^{k}\gamma_{j}\right\}.
\end{align*}
Inserting these bounds into \eqref{eq:Fsuper} we find
\begin{align}
    \frac{\wt{\Psi}(\mbf{z};A^{(k)})}{\wt{\Psi}_{asymp}(\mbf{z};A^{(k)})}&\leq Cd_{N,k}r(\mbf{z})\int_{\mbf{M}^{sa}_{2k}}\exp\left\{-\frac{CNt\gamma_{j}^{2}}{1+\gamma_{j}}\right\}dQ,\label{eq:PsiBound2}
\end{align}
where
\begin{align*}
    r(\mbf{z})&=|\Delta(\mbf{z})|^{2}\det^{-1}\left[\frac{1}{(2\pi)^{3/2}}\int_{-1}^{1}e^{-2\tau_{E,t}\lambda^{2}-i(z_{j}-\bar{z}_{l})\lambda}d\lambda\right].
\end{align*}
Since
\begin{align*}
    r(\mbf{z})&=k!(2\pi)^{3k/2}\left(\int_{-1}^{1}\left|\frac{\det(e^{-\tau_{E,t}\lambda_{l}^{2}-iz_{j}\lambda_{l}})}{\Delta(\mbf{z})}\right|^{2}d\bs\lambda\right)^{-1}
\end{align*}
and the integrand is positive, we have $r(\mbf{z})\leq C$. If $||\eta_{j}|-\eta_{E,t}|>\sqrt{\frac{t}{N}}\log N$, then $\gamma^{2}_{j}>\frac{C\log^{2}N}{Nt}$ and hence the integral in \eqref{eq:PsiBound2} is $O(e^{-\log^{2}N})$. This proves that we can restrict the integral to the region 
\begin{align*}
    \bigcap_{j=1}^{2k}\left\{||\eta_{j}|-\eta_{E,t}|<\sqrt{\frac{t}{N}}\log N\right\}.
\end{align*}

We can split this region into disjoint subsets according to signature of $Q$. For each signature of $Q$ we can choose a reference subset $\Omega_{s}$ in which the first $k+s$ eigenvalues of $Q$ are positive. Since the integrand is symmetric in the eigenvalues, all matrices with a given signature $s$ can be transformed into an element of $\Omega_{s}$. Therefore, defining $\wt{\Psi}_{s}$ to be the restriction of the integral to $\Omega_{s}$, we have
\begin{align}
    \wt{\Psi}&=\sum_{s=-k}^{k}\begin{pmatrix}2k\\k+s\end{pmatrix}\wt{\Psi}_{s}+O\left(e^{-C\log^{2}N}\right).
\end{align}

Let $L_{s}$ be the diagonal matrix whose first $k+s$ and last $k-s$ elements are $+1$ and $-1$ respectively, so that $L_{0}=L$. In $\Omega_{s}$ we make the change of variables
\begin{align*}
    Q&=U\left(\eta_{E,t}L_{s}+\sqrt{\frac{t}{N}}\bs\eta\right)U^{*},
\end{align*}
where $\bs\eta\in[-\log N,\log N]^{2k}$ and $U\in U(2k)/U(1)^{2k}$. We lift the integral over the coset $U(2k)/U(1)^{2k}$ to the full group $U(2k)$ with normalised Haar measure $\mu_{2k}$. The Lebesgue measure $dQ$ transforms as
\begin{align*}
    dQ&=\frac{v_{2k}}{(2\pi)^{2k}(2k)!}\left(\frac{4\eta^{2}_{E,t}t}{N}\right)^{k^{2}}J_{s}(\bs\eta)d\bs\eta d\mu_{2k}(U),
\end{align*}
where $v_{2k}=\text{Vol}(U(2k))=2^{k}\pi^{k(k+1)/2}(\prod_{j=1}^{k-1}j!)^{-1}$ and
\begin{align*}
    J_{s}(\bs\eta)&=\left(\frac{t}{N}\right)^{s^{2}}\Delta^{2}(\eta_{1},...,\eta_{k+s})\Delta^{2}(\eta_{k+s+1},...,\eta_{2k})\prod_{j=1}^{k+s}\prod_{l=1}^{k-s}\left(1+\sqrt{\frac{t}{N}}\frac{\eta_{j}-\eta_{l}}{\eta_{E,t}}\right).
\end{align*}

Since $|\eta_{j}|>Ct$ in $\Omega_{s}$, we can apply Cramer's rule directly to $D^{(k)}(Q)$ without factorising it first as we did before. Defining $\bs\lambda=u_{E,t}+i\left(\eta_{E,t}L_{s}+\sqrt{\frac{t}{N}}\bs\eta\right)$ and $\mbf{F}^{(j)}_{\bs\lambda}=\mbf{F}^{(j)}_{\bs\lambda}(U^{*}LU)$ we have
\begin{align*}
    \frac{D^{(j)}(Q)}{D^{(j-1)}(Q)}&=\det\mbf{V}_{j}^{*}\mbf{F}^{(j-1)}_{\bs\lambda}\mbf{V}_{j}.
\end{align*}
From the resolvent identity we have
\begin{align*}
    \frac{\det\mbf{V}_{j}^{*}\mbf{F}^{(j-1)}_{\bs\lambda}\mbf{V}_{j}}{\det\mbf{V}_{j}^{*}\mbf{F}^{(j-1)}_{s}\mbf{V}_{j}}&=\det\left[1+\sqrt{\frac{t}{N}}(\mbf{V}_{j}^{*}\mbf{F}^{(j-1)}_{s}\mbf{V}_{j})^{-1}\mbf{V}_{j}^{*}\mbf{F}^{(j-1)}_{\bs\lambda}(\bs\eta\otimes1)\mbf{F}^{(j-1)}_{s}\mbf{V}_{j}\right].
\end{align*}
To bound the norm of $\mbf{F}^{(k-1)}_{\bs\lambda}$ we use the resolvent identity again:
\begin{align*}
    \|\mbf{F}^{(j-1)}_{\bs\lambda}\|&\leq\|\mbf{F}^{(j-1)}_{s}\|+\sqrt{\frac{t}{N}}\|\bs\eta\|\cdot\|\mbf{F}^{(j-1)}_{s}\|\cdot\|\mbf{F}^{(j-1)}_{\bs\lambda}\|\\
    &\leq\frac{1}{1-\frac{C\log N}{\sqrt{Nt}}}\|\mbf{F}^{(j-1)}_{s}\|\\
    &\leq\frac{C}{\eta_{E,t}},
\end{align*}
where the last inequality is \eqref{eq:mbfFminorNorm}. Therefore
\begin{align*}
    \frac{\det\mbf{V}_{j}^{*}\mbf{F}^{(j-1)}_{\bs\lambda}\mbf{V}_{j}}{\det\mbf{V}_{j}^{*}\mbf{F}^{(j-1)}_{s}\mbf{V}_{j}}&=1+O\left(\frac{\log N}{\sqrt{Nt}}\right).
\end{align*}
Replacing $\mbf{F}^{(j-1)}_{s}$ with $\mbf{G}^{(j-1)}_{s}$ by \eqref{eq:PertRatio} and repeating these steps $k$ times we find
\begin{align*}
    \frac{D^{(k)}(Q)}{D(Q)}&=\left[1+O\left(\frac{1}{(Nt)^{\frac{1}{2}\left(1-\frac{1}{n_{\epsilon}}\right)}}\right)\right]\prod_{j=1}^{k}\det\mbf{V}_{j}^{*}\mbf{G}^{(j-1)}_{s}\mbf{V}_{j}\\
    &=\left[1+O\left(\frac{1}{(Nt)^{\frac{1}{2}\left(1-\frac{1}{n_{\epsilon}}\right)}}\right)\right]\prod_{j=1}^{k}(\mbf{v}_{j}^{*}G^{(j-1)}\mbf{v}_{j})^{k+s}(\mbf{v}_{j}^{*}G^{(j-1)*}\mbf{v}_{j})^{k-s}.
\end{align*}

Now we factorise $D(Q)$:
\begin{align}
    D(Q)&=\det\left(1+i\sqrt{\tau_{N}}\mbf{G}_{\bs\lambda}(U^{*}LU\otimes \ac{W}_{2})\right)\prod_{j=1}^{2k}\det(W_{1}-\lambda_{j}).\label{eq:Dfactorised}
\end{align}
The first factor can be directly estimated using \eqref{eq:det1}:
\begin{align*}
    \det\left(1+i\sqrt{\tau_{N}}\mbf{G}_{\bs\lambda}(U^{*}LU\otimes\ac{W}_{2})\right)&=\left[1+O\left(\frac{N^{\epsilon/2}}{\sqrt{Nt}}\right)\right]\exp\left\{\frac{\tau_{N}}{2}\tr(\mbf{G}_{\bs\lambda}(U^{*}LU\otimes\ac{W}_{2}))^{2}\right\}.
\end{align*}
Using the resolvent identity we can replace $\lambda_{j}$ with $\lambda_{E,t}$ if $j\leq k+s$ and $-\lambda_{E,t}$ if $j>k+s$:
\begin{align*}
    \mbf{G}_{\bs\lambda}&=\mbf{G}_{s}+i\mbf{G}_{s}((\bs\eta-\eta_{E,t}L_{s})\otimes1)\mbf{G}_{\bs\lambda}.
\end{align*}
The resulting terms can be estimated by \eqref{eq:C3.1'} and \eqref{eq:C3.2'}. We find
\begin{align*}
    &\det\left(1+i\sqrt{\tau_{N}}\mbf{G}_{\bs\lambda}(U^{*}LU\otimes\ac{W}_{2})\right)\\&=\left[1+O\left(\frac{N^{\epsilon/2}}{\sqrt{Nt}}\right)\right]\exp\left\{k\alpha_{E,t}+isN\tau_{N}\langle\Re(G)\ac{W}_{2}\Im(G)\ac{W}_{2}\rangle-\frac{\beta_{E,t}}{2}\tr(UL_{s}U^{*}L)^{2}\right\}.
\end{align*}

Now consider one factor in the product over $j$ on the right hand side of \eqref{eq:Dfactorised}. Assume without loss that $\Im\lambda_{j}>0$ and divide by $\det(W_{1}-\lambda_{E,t})$ (otherwise divide by $\overline{\det}(W_{1}-\lambda_{E,t})$). Using the identity
\begin{align*}
    \log(1+x)&=x-\int_{0}^{1}x^{2}u(1+ux)^{-1}du,\quad x\in\mbb{C}\setminus[-\infty,-1],
\end{align*}
we obtain
\begin{align*}
    \frac{\det(W_{1}-\lambda_{j})}{\det(W_{1}-\lambda_{E,t})}&=\det\left(1-i\sqrt{\frac{t}{N}}\eta_{j}G\right)\\
    &=\exp\left\{-i\sqrt{\frac{t}{N}}\eta_{j}\tr G+\frac{t\eta_{j}^{2}}{2N}\tr G^{2}+\frac{t^{3/2}\eta_{j}^{3}}{N^{3/2}}\int_{0}^{1}u\Re\tr G^{3}\left(1-i\sqrt{\frac{t}{N}}u\eta_{j}G\right)^{-1}du\right\}.
\end{align*}
From the definition of $\lambda_{E,t}$ we have $t\langle G\rangle=\lambda_{E,t}-E$. Moreover, we have
\begin{align*}
    \frac{t^{3/2}|\eta_{j}|^{3}}{N^{3/2}}\left|\Re\tr G^{3}\left(1-i\sqrt{\frac{t}{N}}u\eta_{j}G\right)^{-1}\right|&\leq\frac{Ct^{3/2}|\eta_{j}|^{3}}{N^{3/2}\eta_{E,t}}\tr|G|^{2}\\
    &\leq\frac{C\log^{3}N}{\sqrt{Nt}}.
\end{align*}
Therefore
\begin{align*}
    &e^{-\frac{N}{2t}\left[\left(\eta_{E,t}+\sqrt{\frac{t}{N}}\eta_{j}\right)^{2}-\eta_{E,t}^{2}\right]}\frac{\det(W_{1}-\lambda_{j})}{\det(W_{1}-\lambda_{E,t})}\\&=\left[1+O\left(\frac{\log^{3}N}{\sqrt{Nt}}\right)\right]\exp\left\{-\frac{1}{2}\left(1-t\langle G^{2}\rangle\right)\eta_{j}^{2}-i\sqrt{\frac{N}{t}}(u_{E,t}-E)\eta_{j}\right\},
\end{align*}
for sufficiently large $N$.

For $g(Q)$ we find by similar arguments
\begin{align*}
    g(Q)&=\left[1+O\left(\frac{\log N}{\sqrt{Nt}}\right)\right]\exp\left\{-\frac{2isN\eta_{E,t}(E-u_{E,t})}{t(1-\tau_{N})}+i\sqrt{\frac{N}{t}}(u_{E,t}-E)\sum_{j=1}^{2k}\eta_{j}\right.\\&\left.-i\tr UL_{s}U^{*}Z-\frac{N\tau_{N}\eta_{E,t}^{2}}{2t}\tr(LUL_{s}U^{*})^{2}\right\}.
\end{align*}
Since
\begin{align*}
    \left|\frac{J_{s}(\bs\eta)}{J_{0}(\bs\eta)}\right|&\leq C\left(\frac{t}{N}\right)^{s^{2}}
\end{align*}
and the remaining $s$-dependent terms are $O(1)$ uniformly in $\bs\eta\in\Omega_{s}$, we conclude that the dominant contribution to $\wt{\Psi}$ comes from the term corresponding to $s=0$, i.e.
\begin{align*}
    \left|\frac{\wt{\Psi}_{s}}{\wt{\Psi}_{0}}\right|&\leq C\left(\frac{t}{N}\right)^{s^{2}}.
\end{align*}
For the leading term we have
\begin{align*}
    \wt{\Psi}_{0}&=\left[1+O\left(\frac{1}{(Nt)^{\frac{1}{4}}}\right)\right]f_{N,k}I(\mbf{z})\prod_{j=1}^{k}\psi_{j}(A^{(j-1)})|\mbf{v}_{j}^{*}G^{(j-1)}\mbf{v}_{j}|^{2j},
\end{align*}
where
\begin{align}
    f_{N,k}&=\frac{d_{N,k}v_{2k}}{(2\pi)^{2k}(k!)^{2}}\left(\frac{4\eta_{E,t}^{2}t}{N}\right)^{k^{2}}\left|\int_{\mbb{R}^{k}}e^{-\frac{1-t\langle G^{2}\rangle}{2}\sum_{j=1}^{k}\eta^{2}_{j}}\Delta^{2}(\bs\eta)d\bs\eta\right|^{2},\label{eq:f_Nk}
\end{align}
and
\begin{align}
    I(\mbf{z})&=|\Delta(\mbf{z})|^{2}e^{-k\tau_{E,t}}\int_{U(2k)}e^{\tr ULU^{*}Z-\frac{1}{2}\tau_{E,t}\tr(ULU^{*}L)^{2}}dU.\label{eq:Iz}
\end{align}

By taking the average over the unitary group, we can calculate the integral in \eqref{eq:f_Nk} as follows:
\begin{align*}
    \int_{\mbb{R}^{k}}e^{-\frac{1-t\langle G^{2}\rangle}{2}\sum_{j=1}^{k}\eta_{j}^{2}}\Delta^{2}(\bs\eta)d\bs\eta&=\int_{U(2k)}\int_{\mbb{R}^{k}}e^{-\frac{1-t\langle G^{2}\rangle}{2}\tr(U\bs\eta U^{*})^{2}}\Delta^{2}(\bs\eta)d\bs\eta d\mu_{2k}(U)\\
    &=\frac{k!(2\pi)^{k}}{v_{k}}\int_{\mbf{M}^{sa}_{k}}e^{-\frac{1-t\langle G^{2}\rangle}{2}\tr Q^{2}}dQ\\
    &=\frac{k!(2\pi)^{3k/2}\pi^{k(k-1)/2}}{v_{k}(1-t\langle G^{2}\rangle)^{k^{2}/2}}.
\end{align*}

For $I(\mbf{z})$ we note that since $U$ appears only in the combination $ULU^{*}$, the integral reduces to an integral over the coset space $U(2k)/U(k)\times U(k)$. A similar integral has been computed in \cite{fyodorov_systematic_1999}, where the authors used the following parametrisation of the coset space:
\begin{align*}
    U&=\begin{pmatrix}U_{1}&0\\0&U_{2}\end{pmatrix}\begin{pmatrix}\cos\bs\theta&e^{i\bs\phi}\sin\bs\theta\\-e^{-i\bs\phi}\sin\bs\theta&\cos\bs\theta\end{pmatrix}\begin{pmatrix}U_{1}^{*}&0\\0&U_{2}^{*}\end{pmatrix}.
\end{align*}
Here we have $\bs\theta\in[0,\pi/2]^{k},\,\bs\phi\in[0,2\pi]^{k}$ and $U_{1},U_{2}\in U(k)$. The Haar measure transforms as
\begin{align*}
    d\mu_{2k}(U)&=\frac{\prod_{j=1}^{2k-1}j!}{2^{k(k+1)}(2\pi)^{k}\prod_{j=1}^{k-1}(j!)^{4}}\Delta^{2}(\cos2\bs\theta)d(\cos2\bs\theta)d\bs\phi d\mu_{k}(U_{1})d\mu_{k}(U_{2}).
\end{align*}
The integrand becomes
\begin{align*}
    e^{\tr ULU^{*}Z-\frac{1}{2}\tau_{E,t}\tr(ULU^{*}L)^{2}}&=e^{\tr U_{1}\cos2\bs\theta U_{1}^{*}\mbf{z}-\tr U_{2}\cos2\bs\theta U_{2}^{*}\overline{\mbf{z}}-\tau_{E,t}\tr(2\cos^{2}2\bs\theta-1)}.
\end{align*}
We can integrate over $U_{1}$ and $U_{2}$ using the Harish-Chandra--Itzykson--Zuber integral:
\begin{align*}
    \int_{U(k)}e^{\tr UAU^{*}B}&=\left(\prod_{j=1}^{k-1}j!\right)\frac{\det[e^{a_{j}b_{l}}]}{\Delta(A)\Delta(B)}.
\end{align*}
The Vandermonde in $\cos2\bs\theta$ and $\mbf{z}$ cancel, and after changing variables $\bs\lambda=\cos2\bs\theta$ we arrive at 
\begin{align*}
    I(\mbf{z})&=\frac{\prod_{j=1}^{2k-1}j!}{2^{k(k+1)}\prod_{j=1}^{k-1}(j!)^{2}}\det\left[\int_{-1}^{1}e^{-2\tau_{E,t}\lambda^{2}-i(z_{j}-\bar{z}_{l})\lambda}d\lambda\right].
\end{align*}

From this we obtain
\begin{align*}
    f_{N,k}I(\mbf{z})&=\frac{|1-t\langle G^{2}\rangle|^{k}}{(2\pi)^{3k/2}}\left(\frac{2\eta_{E,t}}{|1-t\langle G^{2}\rangle|}\right)^{k(k+1)}\det\left[\int_{-1}^{1}e^{-2\tau_{E,t}\lambda^{2}-i(z_{j}-\bar{z}_{l})\lambda}d\lambda\right],
\end{align*}
and \cref{lem:PsiAsymp} follows.
\end{proof}

\section{Analysis of $\wt{K}_{j}$}\label{sec:K}
We begin with an alternative representation for $\wt{K}_{j}$.
\begin{lemma}\label{lem:K}
We have
\begin{align}
    \wt{K}_{j}(z_{j};A^{(j-1)})&=\frac{N}{(2\pi)^{3/2}t(1-\tau^{2}_{N})\eta_{E,t}^{2}}e^{\frac{N}{t(1+\tau_{N})}\eta^{2}_{E,t}+\frac{t}{(1-\tau_{N})N\eta_{E,t}}y^{2}}\nonumber\\&\int_{-\infty}^{\infty}\int_{-\infty}^{\infty}e^{-\frac{N}{t(1-\tau_{N})}(u-x)^{2}-\frac{N\tau_{N}}{2t(1-\tau^{2}_{N})}v^{2}-\frac{iyv}{(1-\tau_{N})\eta_{E,t}}}\widehat{K}_{j}(u,v;z_{j})dudv,\label{eq:K}
\end{align}
where
\begin{align}
    \widehat{K}_{j}(u,v;z_{j})&=\left(\frac{N}{\pi t(1+\tau_{N})}\right)^{N-j}\int_{S^{N-j}}e^{-\frac{N}{t(1+\tau_{N})}\mbf{v}^{*}\left(\eta^{2}_{E,t}+\left|\ac{A}^{(j-1)}-u\right|^{2}+iv\Im \ac{A}^{(j-1)}\right)\mbf{v}}dS_{N-j}(\mbf{v}).\label{eq:Khat}
\end{align}
\end{lemma}
\begin{proof}
Recall the expression for $\wt{K}_{j}$:
\begin{align*}
    \wt{K}_{j}(z_{j};A^{(j-1)})&=\left(\frac{N}{\pi t(1+\tau_{N})}\right)^{N-j}\int_{S^{N-j}}e^{-\frac{N}{t(1+\tau_{N})}\left\|\mbf{c}_{j}\right\|^{2}-\frac{N}{2t}\left[\left(\Re w_{j}-\Re a_{j}\right)+\frac{1}{\tau_{N}}\left(\Im w_{j}-\Im a_{j}\right)^{2}\right]}dS_{N-j}(\mbf{v}_{j}),
\end{align*}
where $w_{j}=E+i\sqrt{\tau_{N}}\langle W_{2}\rangle+tz_{j}/N\eta_{E,t}$ and
\begin{align*}
    a_{j}&=\mbf{v}_{j}^{*}A^{(j-1)}\mbf{v}_{j},\\
    \mbf{c}_{j}&=(1-\mbf{v}_{j}\mbf{v}_{j}^{*})A^{(j-1)}\mbf{v}_{j}.
\end{align*}
The norm of $\left\|\mbf{c}_{j}\right\|$ is given by
\begin{align*}
    \left\|\mbf{c}_{j}\right\|^{2}&=\mbf{v}_{j}^{*}|A^{(j-1)}|^{2}\mbf{v}_{j}-\left|a_{j}\right|^{2}.
\end{align*}
The quadratic term in $a_{j}$ can be linearised by introducing two auxiliary Gaussian integrals:
\begin{align*}
    e^{\frac{N(1-\tau_{N})}{2t(1+\tau_{N})}\left((\Re a_{j})^{2}-\frac{1}{\tau_{N}}(\Im a_{j})^{2}\right)}&=\frac{N(1-\tau_{N})\sqrt{\tau_{N}}}{2\pi t(1+\tau_{N})}\int_{-\infty}^{\infty}\int_{-\infty}^{\infty}e^{-\frac{N(1-\tau_{N})}{2t(1+\tau_{N})}(u^{2}+\tau_{N}v^{2}-2u\Re a_{j}-2iv\Im a_{j})}dudv.
\end{align*}
The $\mbf{v}_{j}$-dependent term $NQ(\mbf{v}_{j})/t$ in the exponent is now a quadratic form with
\begin{align*}
    Q(\mbf{v}_{j})&=\frac{1}{1+\tau_{N}}\mbf{v}_{j}^{*}\Big(|A^{(j-1)}|^{2}-\big((1+\tau_{N})\Re w_{j}+(1-\tau_{N})u\big)\Re A^{(j-1)}\\&-\frac{1}{\tau_{N}}\left((1+\tau_{N})\Im w_{j}-i\tau_{N}(1-\tau_{N})v\right)\Im A^{(j-1)}\Big)\mbf{v}_{j}.
\end{align*}
Completing the square we find
\begin{align*}
    Q(\mbf{v}_{j})&=\wt{Q}(\mbf{v}_{j})-\frac{1}{4(1+\tau_{N})}\Big(\big((1+\tau_{N})\Re w_{j}+(1-\tau_{N})u\big)^{2}-\big(i(1+\tau_{N})\Im w_{j}/\tau_{N}+(1-\tau_{N})v\big)^{2}\Big),
\end{align*}
where
\begin{align*}
    \wt{Q}(\mbf{v}_{j})&=\frac{1}{1+\tau_{N}}\mbf{v}_{j}^{*}\Big(A^{(j-1)*}-\frac{(1+\tau_{N})(\Re w_{j}-i\Im w_{j}/\tau_{N})+(1-\tau_{N})(u-v)}{2}\Big)\\&\Big(A^{(j-1)}-\frac{(1+\tau_{N})(\Re w_{j}+i\Im w_{j}/\tau_{N})+(1-\tau_{N})(u+v)}{2}\Big)\mbf{v}_{j}.
\end{align*}
Consider now the term in the exponent depending on $\Re w_{j},\Im w_{j},u,v$, which we denote by $Nq/2t$. We have
\begin{align*}
    q&=\Re w^{2}_{j}+\frac{1}{\tau_{N}}\Im w_{j}^{2}+\frac{1-\tau_{N}}{1+\tau_{N}}(u^{2}+\tau_{N}v^{2})\\&-\frac{1}{2(1+\tau_{N})}\Big(\big((1+\tau_{N})\Re w_{j}+(1-\tau_{N})u\big)^{2}-\big(i(1+\tau_{N})\Im w_{j}/\tau_{N}+(1-\tau_{N})v\big)^{2}\Big)\\
    &=\frac{1-\tau_{N}}{2}\left((u-\Re w_{j})^{2}+(v+i\Im w_{j}/\tau_{N})^{2}\right).
\end{align*}
We now make the change of variables
\begin{align*}
    u&\mapsto\frac{2u-(1+\tau_{N})\Re w_{j}}{1-\tau_{N}},\\
    v&\mapsto\frac{2(v+i\sqrt{\tau_{N}}\langle W_{2}\rangle)-i(1+\tau_{N})\Im w_{j}/\tau_{N}}{1-\tau_{N}},
\end{align*}
and shift $v$ back to the real axis. This does not change the value of the integral since the exponent behaves as $-v^{2}$ for large $v$. After this change of variables we find
\begin{align*}
    \wt{Q}(\mbf{v}_{j})&=\frac{1}{1+\tau_{N}}\mbf{v}_{j}^{*}\Big(\ac{A}^{(j-1)*}-u+v\Big)\Big(\ac{A}^{(j-1)}-u-v\Big)\mbf{v}_{j},\\
    q&=\frac{2}{1-\tau_{N}}\left(\left(u-E-\frac{tx_{j}}{N\eta_{E,t}}\right)^{2}+\left(v-\frac{ity_{j}}{N\eta_{E,t}}\right)^{2}\right),
\end{align*}
and hence, after interchanging the $\mbf{v}_{j}$ integral with the $u$ and $v$ integrals,
\begin{align*}
    \wt{K}_{j}(z_{j};A^{(j-1)})&=\frac{N}{2^{1/2}\pi^{3/2}t(1-\tau^{2}_{N})\eta_{E,t}^{2}}\int_{-\infty}^{\infty}\int_{-\infty}^{\infty}e^{-\frac{N}{t(1-\tau_{N})}\left(\left(u-E-\frac{tx_{j}}{N\eta_{E,t}}\right)^{2}+\left(v-\frac{ity_{j}}{N\eta_{E,t}}\right)^{2}\right)}\\&\left(\frac{N}{\pi t(1+\tau_{N})}\right)^{N-j}\int_{S^{N-j}}e^{-\frac{N}{t(1+\tau_{N})}\mbf{v}_{j}^{*}\big(\ac{A}^{(j-1)*}-u+v\big)\big(\ac{A}^{(j-1)}-u-v\big)\mbf{v}_{j}}dS_{N-j}(\mbf{v}_{j})dudv.
\end{align*}
Since
\begin{align*}
    1&=\exp\left\{\frac{N\eta^{2}}{t(1+\tau_{N})}-\frac{N\eta^{2}}{t(1+\tau_{N})}\right\}=\exp\left\{\frac{N\eta^{2}}{t(1+\tau_{N})}-\frac{N}{t(1+\tau_{N})}\mbf{v}_{j}^{*}\eta^{2}\mbf{v}_{j}\right\},
\end{align*}
we can insert $\eta_{E,t}^{2}$ into the quadratic form and arrive at \eqref{eq:K}.
\end{proof}

\begin{proof}[Proof of \Cref{lem:KBound}]
Let
\begin{align}
    \wt{K}_{j,upper}&=\det^{-1}\left(\eta_{E,t}^{2}+|\ac{A}^{(j-1)}-u_{E,t}|^{2}\right)e^{-\frac{N}{t(1-\tau_{N})}\Re(u_{E,t}-w_{j})^{2}+\frac{N}{t}\eta_{E,t}^{2}}.\label{eq:Kupper}
\end{align}
Our aim is to prove that
\begin{align}
    \wt{K}_{j}(z_{j};A^{(j-1)})&\leq\frac{C}{\sqrt{t}}\wt{K}_{j,upper}.\label{eq:Kbound1}
\end{align}
Dividing $\wt{K}_{j}$ by $\wt{K}_{j,upper}$ we can write
\begin{align}
    \frac{\wt{K}_{j}}{\wt{K}_{j,upper}}&\leq\frac{CN}{t^{5/2}}\int_{-\infty}^{\infty}e^{-N(\phi(u)-\phi(u_{E,t}))}h_{j}(u)r_{j}(u)du,
\end{align}
where
\begin{align}
    \phi(u)&=\frac{(u-E)^{2}}{t}+\frac{1}{N}\log\det\left(\eta_{E,t}^{2}+(W_{1}-u)^{2}\right),\label{eq:phi}\\
    h_{j}(u)&=e^{\frac{C|u-u_{E,t}|}{t}}\int_{-\infty}^{\infty}e^{\frac{iNp}{t(1+\tau_{N})}}\det^{-1}\left(1+ipH^{(j-1)}_{u}\right)dp,\label{eq:h}\\
    r_{j}(u)&=\frac{\det\left(\eta_{E,t}^{2}+(W_{1}-u)^{2}\right)\det\left(\eta_{E,t}^{2}+|\ac{A}^{(j-1)}-u_{E,t}|^{2}\right)}{\det\left(\eta_{E,t}^{2}+(W_{1}-u_{E,t})^{2}\right)\det\left(\eta_{E,t}^{2}+|\ac{A}^{(j-1)}-u|^{2}\right)}.
\end{align}
To arrive at this expression we have inserted absolute values into the integrand in \eqref{eq:K} and \eqref{eq:Khat} and used \cref{lem:sphericalint} to rewrite the latter as the integral over $p$ in \eqref{eq:h}. Inserting absolute values leaves $e^{-\frac{N\tau_{N}}{2t(1-\tau_{N}^{2})}v^{2}}$ as the only $v$-dependent, which after integrating gives a factor $\sqrt{t}$.

For $\phi$, we note that $\phi'(u_{E,t})=0$ by definition and by \eqref{eq:C1.3}
\begin{align*}
    \phi''(u)&=\frac{2}{t}\left(1-t\Re\Tr{G^{2}(u+i\eta_{E,t})}\right)\\
    &=\frac{2+O(t)}{t},
\end{align*}
when $|u|<2$. Since $|E|<2$, we have by \eqref{eq:C1.1}
\begin{align*}
    \phi'(u)&=\frac{2(u-E)}{t}-2\Re\langle G(u+i\eta_{E,t})\rangle\\
    &\geq\frac{C}{t},
\end{align*}
when $u>2$, and likewise
\begin{align*}
    \phi'(u)&\leq-\frac{C}{t},
\end{align*}
when $u<2$. Therefore we obtain
\begin{align}
    \phi(u)-\phi(u_{E,t})&\geq\frac{C}{t}(u-u_{E,t})^{2}.\label{eq:phiBound}
\end{align}

Since the integrand in \eqref{eq:h} decays as $e^{-CNp^{2}\|H^{(j-1)}_{u}\|^{2}}$, we can truncate the integration domain to $|p|<C\|\ac{A}^{(j-1)}-u\|^{2}<C\|\ac{A}-u\|^{2}$, after which we simply bound the integrand by 1 to obtain
\begin{align*}
    h_{j}(u)&\leq C(1+|u-u_{E,t}|^{2})e^{\frac{C|u-u_{E,t}|}{t}}.
\end{align*}
To bound $r_{j}(u)$, we use Cramer's rule to replace $A^{(j-1)}$ with $A$ in the numerator and denominator, arriving at:
\begin{align*}
    r_{j}(u)&=\frac{\det\left(1+i\sqrt{\tau_{N}}G\ac{W}_{2}\right)}{\det\left(1+i\sqrt{\tau_{N}}G(u+i\eta_{E,t})\ac{W}_{2}\right)}\prod_{l=1}^{j-1}\frac{\det V_{l}^{*}\mc{F}^{(l-1)}V_{l}}{\det V_{l}^{*}\mc{F}^{(l-1)}_{u}V_{l}}.
\end{align*}
By \eqref{eq:det1}, both the numerator and denominator in the first factor are $O(1)$ uniformly in $u\in\mbb{R}$. For each term in the product over $l$ we use the resolvent identity:
\begin{align*}
    \frac{\det V_{l}^{*}\mc{F}^{(l-1)}V_{l}}{\det V_{l}^{*}\mc{F}^{(l-1)}_{u}V_{l}}&=\det\left[1-(u-u_{E,t})(V_{l}^{*}\mc{F}^{(l-1)}_{u}V_{l})^{-1}V_{l}^{*}\mc{F}^{(l-1)}(\sigma_{X}\otimes1)\mc{F}^{(l-1)}_{u}V_{l}\right].
\end{align*}
By \Cref{lem:minorresolvent}, the matrix on the right hand side is bounded by $C|u-u_{E,t}|/t$; since it is rank 2 we find
\begin{align}
    r_{j}(u)&\leq C\left(1+\frac{|u-u_{E,t}|}{t}\right)^{2}.\label{eq:rBound}
\end{align}

The upshot of these bounds is that we can restrict $u$ to the region $|u-u_{E,t}|<\sqrt{\frac{t}{N}}\log N$. In this region we revisit the $p$-integral in \eqref{eq:h}. Truncating the integral to $|p|<C$ and using $\log(1+x)\leq x/(1+x)$ we find
\begin{align*}
    |h_{j}(u)|&\leq C\int_{-C}^{C}\exp\left\{-\frac{1}{2}\left(1+\frac{p^{2}}{\eta_{E,t}^{2}}\right)^{-1}p^{2}\tr(H^{(j-1)}_{u})^{2}\right\}dp+O(e^{-CN}).
\end{align*}
By interlacing and \eqref{eq:traceEstimate} we have
\begin{align*}
    \tr(H^{(j-1)}_{u})^{2}&=\tr|G(u+i\eta_{E,t})|^{4}+O\left(\frac{N^{1/2}}{t^{7/2}}\right)\\
    &\geq\frac{CN}{t^{3}},
\end{align*}
from which we obtain
\begin{align}
    |h_{j}(u)|\leq C\sqrt{\frac{t^{3}}{N}}.\label{eq:hBound}
\end{align}
Since we also have $r_{j}(u)\leq C$ in this region, \eqref{eq:Kbound1} follows after bounding $\phi(u)-\phi(u_{E,t})$ below by $C(u-u_{E,t})^{2}/t$ and performing the Gaussian integral over $u$.
\end{proof}

\begin{proof}[Proof of \Cref{lem:KAsymp}]
Assume that $\mbf{v}_{l}\in\mc{E}_{l}$ for $l=1,...,j-1$. We multiply $\wt{K}_{j}$ by $\psi_{j}(A^{(j-1)})$ defined in \eqref{eq:psi_j} and rewrite it in the following way:
\begin{align*}
    \psi_{j}(A^{(j-1)})\wt{K}_{j}(z_{j};A^{(j-1)})&=p_{N}|1-t\langle G^{2}\rangle|\int_{-\infty}^{\infty}e^{-N\left(\phi(u)-\phi(u_{E,t})\right)}f_{j}(u)I_{j}(u)du,
\end{align*}
where $\phi$ was defined in \eqref{eq:phi} and
\begin{align}
    f_{j}(u)&=\exp\left\{\frac{2x_{j}(u-u_{E,t})}{(1-\tau_{N})\eta_{E,t}}-\frac{N\tau_{N}}{t(1-\tau_{N})}((u-E)^{2}-(u_{E,t}-E)^{2})\right\}\nonumber\\&\times e^{\alpha_{E,t}+\beta_{E,t}}\frac{\det\left(\eta_{E,t}^{2}+(W_{1}-u)^{2}\right)}{\det\left(\eta_{E,t}^{2}+|\ac{A}^{(j-1)}-u|^{2}\right)}\prod_{l=1}^{j-1}|\mbf{v}_{l}^{*}G^{(l-1)}\mbf{v}_{l}|^{2},\label{eq:f}\\
    I_{j}(u)&=\det\left(\eta_{E,t}^{2}+|\ac{A}^{(j-1)}-u|^{2}\right)\int_{-\infty}^{\infty}e^{-\frac{N\tau_{N}}{2t(1-\tau^{2}_{N})}v^{2}-\frac{iy_{j}v}{(1-\tau_{N})\eta_{E,t}}}\widehat{K}_{j}(u,v;z_{j})dv,\label{eq:Ij}\\
    p_{N}&=\frac{N}{(2\pi)^{3/2}t(1-\tau_{N}^{2})\eta_{E,t}^{2}}.\label{eq:pN}
\end{align}
First we bound $f_{j}$ and $I_{j}$ uniformly in $u$ so that we can restrict $u$ to a neighbourhood of $u_{E,t}$. Using \eqref{eq:det1} we have
\begin{align*}
    e^{\alpha_{E,t}+\beta_{E,t}}&=\left[1+O\left(\frac{1}{Nt}\right)\right]\det\left(1+i\sqrt{\tau_{N}}\mc{G}(\sigma_{Y}\otimes\ac{W}_{2})\right),
\end{align*}
from which we obtain
\begin{align*}
    f_{j}(u)&=r_{j}(u)\prod_{l=1}^{j-1}\frac{\det V_{l}^{*}\mc{G}^{(l-1)}V_{l}}{\det V_{l}^{*}\mc{F}^{(l-1)} V_{l}}.
\end{align*}
By \Cref{lem:minornorms} we have
\begin{align*}
    \|(V_{l}^{*}\mc{G}^{(l-1)}V_{l})^{-1}\|&\leq C\eta_{E,t},\\
    \|\mc{G}^{(l-1)}(\sigma_{Y}\otimes\ac{W}^{(l-1)}_{2})\mc{F}^{(l-1)}\|&\leq\frac{C(N\eta_{E,t})^{1/2n_{\epsilon}}}{\eta_{E,t}^{3/2}}.
\end{align*}
Then we can bound each term in the product over $l$:
\begin{align*}
    \frac{\det V_{l}^{*}\mc{G}^{(l-1)}V_{l}}{\det V_{l}^{*}\mc{F}^{(l-1)}V_{l}}&=\det^{-1}\left[1+i\sqrt{\tau_{N}}(V_{l}^{*}\mc{G}^{(l-1)}V_{l})^{-1}V_{l}^{*}\mc{G}^{(l-1)}(\sigma_{Y}\otimes\ac{W}_{2})\mc{F}^{(l-1)}V_{l}\right]\\
    &=1+O\left(\frac{1}{(Nt)^{\frac{1}{2}\left(1-\frac{1}{n_{\epsilon}}\right)}}\right).
\end{align*}
Hence by the bound on $r_{j}$ in \eqref{eq:rBound} we have
\begin{align*}
    f_{j}(u)&\leq C\left(1+\frac{|u-u_{E,t}|^{2}}{t}\right)^{2}e^{\frac{C|u-u_{E,t}|}{t}}.
\end{align*}
For $I_{j}$ we insert absolute values inside the integral to obtain
\begin{align*}
    |I_{j}(u)|&\leq C\sqrt{t}h_{j}(u)\leq C\sqrt{t}(1+|u-u_{E,t}|^{2}),
\end{align*}
where $h_{j}$ was defined in \eqref{eq:h} and bounded in \eqref{eq:hBound}. It then follows from \eqref{eq:phiBound} that we can restrict $u$ to the region $|u-u_{E,t}|<\sqrt{\frac{t}{N}}\log N$.

Now we fix a $u$ in this region. Using \Cref{lem:sphericalint} we can rewrite $I_{j}(u)$ as follows
\begin{align}
    I_{j}(u)&=\int_{-\infty}^{\infty}\int_{-\infty}^{\infty}e^{-\frac{N\tau_{N}}{2t(1-\tau^{2}_{N})}v^{2}-\frac{iy_{j}v}{(1-\tau_{N})\eta_{E,t}}+\frac{iNp}{t(1+\tau_{N})}}\det^{-1}\left(1+iM^{(j-1)}_{u}\right)dpdv,\label{eq:I2}
\end{align}
where
\begin{align}
    M^{(j-1)}_{u}&=\sqrt{H^{(j-1)}_{u}}\left(p+v\Im \ac{A}^{(j-1)}\right)\sqrt{H^{(j-1)}_{u}}.\label{eq:Mj}
\end{align}
First we note that the $p$ integral converges uniformly in $v$ while the exponent decays as $e^{-|v|/t}$. Hence we can truncate the $v$-integral to the region $|v|<C$. Then we note that, for $|v|<C$, the determinant in the integrand decays as $e^{-N|p|}$ and so we can truncate the $p$-integral to the region $|p|<C$. 

By \eqref{eq:HWnorm} we can bound the norm of the matrix inside the determinant in the integrand of $I_{j}(u)$ by
\begin{align*}
    \left\|M^{(j-1)}_{u}\right\|&\leq\frac{|p|}{\eta_{E,t}^{2}}+\frac{|v|}{(Nt)^{1/4}t}.
\end{align*}
Using $\log(1+x)\geq x/(1+x)$ we have
\begin{align}
    \left|\det^{-1}\left(1+iM^{(j-1)}_{u}\right)\right|&\leq\exp\left\{-\frac{a^{(j-1)}p^{2}-2b^{(j-1)}|pv|+c^{(j-1)}v^{2}}{1+\alpha p^{2}+2\beta|pv|+\gamma v^{2}}\right\},\label{eq:qFormRatio}
\end{align}
where
\begin{align*}
    a^{(j-1)}&=\tr{(H^{(j-1)}_{u})^{2}},\\
    b^{(j-1)}&=\sqrt{\tau_{N}}\left|\tr{H^{(j-1)}_{u}\ac{W}^{(j-1)}_{2}H^{(j-1)}_{u}}\right|,\\
    c^{(j-1)}&=\tau_{N}\tr{\left(H^{(j-1)}_{u}\ac{W}^{(j-1)}_{2}\right)^{2}}.
\end{align*}
Since we have assumed $\mbf{v}_{l}\in\mc{E}_{l}$ for $l=1,...,j-1$, we can replace $A^{(j-1)}$ by $A$ in each of these quantities and then use the estimate in \eqref{eq:traceEstimate}. Consider $a^{(j-1)}$; rewriting the trace in terms of $\mc{F}^{(j-1)}_{u}$ we find
\begin{align*}
    a^{(j-1)}&=\frac{1}{4\eta_{E,t}^{3}}\Im\tr\mc{F}^{(j-1)}_{u}-\frac{1}{4\eta_{E,t}^{2}}\Re\tr(\mc{F}^{(j-1)}_{u})^{2}.
\end{align*}
By interlacing $\Tr{(\mc{F}^{(j-1)}_{u})^{m}}=\Tr{\mc{F}^{m}_{u}}+O(1/Nt^{m})$ and hence
\begin{align*}
    a^{(j-1)}&=\frac{1}{4\eta_{E,t}^{3}}\Im\tr\mc{F}_{u}-\frac{1}{4\eta_{E,t}^{2}}\Re\tr\mc{F}^{2}_{u}+O\left(\frac{1}{t^{4}}\right)\\
    &=\frac{1}{2\eta_{E,t}^{2}}\tr|\mc{F}_{u}|^{4}+O\left(\frac{1}{t^{4}}\right).
\end{align*}
By \cref{eq:traceEstimate} we can replace $\mc{F}$ by $\mc{G}$:
\begin{align*}
    a^{(j-1)}&=\frac{1}{4\eta_{E,t}^{3}}\Im\tr\mc{G}_{u}-\frac{1}{4\eta_{E,t}^{2}}\Re\tr\mc{G}^{2}_{u}+O\left(\frac{N^{1/2}}{t^{7/2}}\right).
\end{align*}
By the resolvent identity and \eqref{eq:C1.3} we can replace $u$ by $u_{E,t}$:
\begin{align*}
    a^{(j-1)}&=\frac{1}{4\eta_{E,t}^{3}}\Im\tr\mc{G}-\frac{1}{4\eta_{E,t}^{2}}\Re\tr\mc{G}^{2}+O\left(\frac{N^{1/2}}{t^{7/2}}\right)\\
    &=\left[1+O\left(\frac{1}{\sqrt{Nt}}\right)\right]\frac{N(1-t\Re\langle G^{2}\rangle)}{2\eta_{E,t}^{2}t}.
\end{align*}

For $b^{(j-1)}$ we perform similar steps, except that in place of interlacing we use \Cref{lem:minorresolvent} and the bounds from \Cref{lem:minornorms}. Rewriting $b^{(j-1)}$ in terms of $\mc{F}^{(j-1)}_{u}$ we have
\begin{align*}
    b^{(j-1)}&=\frac{\sqrt{\tau_{N}}}{4\eta_{E,t}^{3}}\Im\tr\mc{F}^{(j-1)}_{u}(1\otimes\ac{W}^{(j-1)}_{2})-\frac{\sqrt{\tau_{N}}}{4\eta_{E,t}^{2}}\Re\tr(\mc{F}^{(j-1)}_{u})^{2}(1\otimes\ac{W}^{(j-1)}_{2}).
\end{align*}
Consider the first term; using \eqref{eq:mcFminor} we have
\begin{align*}
    \tr\mc{F}^{(l)}_{u}(1\otimes\ac{W}^{(l)}_{2})&=\tr\mc{F}^{(l-1)}_{u}(1\otimes\ac{W}^{(l-1)}_{2})-\tr(V_{l}^{*}\mc{F}^{(l-1)}_{u}V_{l})^{-1}V_{l}^{*}\mc{F}^{(l-1)}_{u}(1\otimes\ac{W}^{(l-1)}_{2})\mc{F}^{(l-1)}_{u}V_{l}.
\end{align*}
From the proof of \Cref{lem:minornorms} we have
\begin{align*}
    \|(V_{l}^{*}\mc{F}^{(l-1)}_{u}V_{l})^{-1}\|&\leq C\eta_{E,t}\leq Ct.
\end{align*}
By the resolvent identity and \eqref{eq:mcFWmcFminorNorm} we have
\begin{align*}
    \|\mc{F}^{(l-1)}_{u}(1\otimes\ac{W}^{(l-1)}_{2})\mc{F}^{(l-1)}_{u}\|&\leq\left(1+|u-u_{E,t}|\|\mc{F}^{(l-1)}_{u}\|\right)^{2}\|\mc{F}^{(l-1)}(1\otimes\ac{W}^{(l-1)}_{2})\mc{F}^{(l-1)}\|\\
    &\leq\frac{C(Nt)^{1/2n_{\epsilon}}}{t^{3/2}}.
\end{align*}
Hence we find
\begin{align*}
    \tr\mc{F}^{(l)}_{u}(1\otimes\ac{W}^{(l)}_{2})&=\tr\mc{F}_{u}(1\otimes\ac{W}_{2})+O\left(\frac{(Nt)^{1/2n_{\epsilon}}}{\sqrt{t}}\right).
\end{align*}
By \eqref{eq:traceEstimate} and \eqref{eq:C3.1} we have
\begin{align*}
    \tr\mc{F}_{u}(1\otimes\ac{W}_{2})&=\tr\mc{G}_{u}(1\otimes\ac{W}_{2})+O\left(\sqrt{N}\right)\\
    &=2i\Im\tr G(u+i\eta_{E,t})\ac{W}_{2}+O\left(\sqrt{N}\right)\\
    &=O\left(\sqrt{N}\right).
\end{align*}
Estimating $\tr(\mc{F}^{(j-1)}_{u})^{2}(1\otimes\ac{W}^{(j-1)}_{2})$ in the same way we find
\begin{align*}
    b^{(j-1)}&=O\left(\frac{1}{t^{3}}\right).
\end{align*}

For $c^{(j-1)}$ we have
\begin{align*}
    c^{(j-1)}&=\frac{\tau_{N}}{2\eta_{E,t}^{2}}\tr\left(\Im(\mc{F}^{(j-1)}_{u})(1\otimes\ac{W}_{2}^{(j-1)})\right)^{2}.
\end{align*}
Write $\Im\mc{F}=\frac{1}{2i}(\mc{F}-\mc{F}^{*})$ and consider the first term in the resulting sum:
\begin{align*}
    \tr\left(\mc{F}^{(l)}_{u}(1\otimes\ac{W}^{(l)}_{2})\right)^{2}&=\tr\left(\mc{F}^{(l-1)}_{u}(1\otimes\ac{W}^{(l-1)}_{2})\right)^{2}\\
    &-2\tr\left(V_{l}^{*}\mc{F}^{(l-1)}_{u}V_{l}\right)^{-1}V_{l}^{*}\left(\mc{F}^{(l-1)}_{u}(1\otimes\ac{W}^{(l-1)}_{2})\right)^{2}\mc{F}^{(l-1)}_{u}V_{l}\\
    &+\tr\left(\left(V_{l}^{*}\mc{F}^{(l-1)}_{u}V_{l}\right)^{-1}V_{l}^{*}\mc{F}^{(l-1)}_{u}(1\otimes\ac{W}^{(l-1)}_{2})\mc{F}^{(l-1)}_{u}V_{l}\right)^{2}.
\end{align*}
Using \eqref{eq:minorNorm1}, \eqref{eq:mcFWmcFminorNorm} and \eqref{eq:mcFW2mcFminorNorm} to bound the error terms we find
\begin{align*}
    \tr\left(\mc{F}^{(l)}_{u}(1\otimes\ac{W}^{(l-1)}_{2})\right)^{2}&=\tr\left(\mc{F}_{u}(1\otimes\ac{W}_{2})\right)^{2}+O\left(\frac{(Nt)^{1/n_{\epsilon}}}{t}\right).
\end{align*}
Doing the same for the remaining terms we find
\begin{align*}
    c^{(j-1)}&=\frac{\tau_{N}}{2\eta_{E,t}^{2}}\tr\left(\Im(\mc{F}_{u})(1\otimes\ac{W}_{2})\right)^{2}+O\left(\frac{1}{t^{2}(Nt)^{1-1/n_{\epsilon}}}\right)\\
    &=\frac{\tau_{N}}{2\eta_{E,t}^{2}}\tr\left(\Im(\mc{G}_{u}(1\otimes\ac{W}_{2})\right)^{2}+O\left(\frac{1}{t^{2}\sqrt{Nt}}\right)\\
    &=\frac{\tau_{N}}{\eta_{E,t}^{2}}\tr(\Im(G)\ac{W}_{2})^{2}+O\left(\frac{\log N}{t^{2}\sqrt{Nt}}\right)\\
    &=\left[1+O\left(\frac{\log N}{\sqrt{Nt}}\right)\right]\frac{\beta_{E,t}}{\eta_{E,t}^{2}}.
\end{align*}
For the second equality we used \eqref{eq:traceEstimate}; for the third we used the resolvent identity and \eqref{eq:C3.2'} to replace $u$ with $u_{E,t}$; the fourth is the definition of $\beta_{E,t}$.

In summary, we have the bounds
\begin{align*}
    a^{(j-1)}&\geq\frac{CN}{t^{3}},\quad b^{(j-1)}\leq\frac{C}{t^{3}},\quad c^{(j-1)}\geq\frac{C}{t^{2}},\\
    \alpha&\leq\frac{C}{t^{4}},\quad\beta\leq\frac{C}{(Nt)^{1/4}t^{3}},\quad\gamma\leq\frac{C}{(Nt)^{1/2}t^{2}}.
\end{align*}
By elementary manipulations of the exponent on the right hand side of \eqref{eq:qFormRatio}, we deduce that in the region in which $t\log N\leq|v|<C$ and $\sqrt{\frac{t^{3}}{N}}\log N\leq |p|<C$, we have
\begin{align*}
    \left|\det^{-1}\left(1+iM^{(j-1)}_{u}\right)\right|&\leq e^{-C\log^{2}N}.
\end{align*}
Inside the region, we can truncate the Taylor series to second order:
\begin{align*}
    \det^{-1}\left(1+iM^{(j-1)}_{u}\right)&=\exp\left\{-i\tr M^{(j-1)}_{u}-\frac{1}{2}\tr(M^{(j-1)}_{u})^{2}+O\left(\frac{\log^{3}N}{Nt}\right)\right\}.
\end{align*}
Note that due to the relations in \eqref{eq:relations} and the fact that $\tr\rho=0$, we have
\begin{align*}
    \langle(\mc{G}_{u}(\sigma_{Y}\otimes\ac{W}_{2}))^{2n+1}\mc{G}_{u}(1\otimes\ac{W}_{2})\rangle&=0
\end{align*}
for all $n\in\mbb{N}$. From this, the resolvent identity and the bounds in \eqref{eq:Dnepsilon}, \eqref{eq:C3.1} and \eqref{eq:C3.2} we obtain
\begin{align*}
    \langle \mc{F}_{u}(1\otimes\ac{W}_{2}\rangle&=\sum_{n=0}^{n_{\epsilon}-1}\tau^{n}_{N}\langle(\mc{G}_{u}(\sigma_{Y}\otimes\ac{W}_{2}))^{2n}\mc{G}_{u}(1\otimes\ac{W}_{2})+O\left(\frac{1}{N^{23}}\right)\\
    &=O\left(\frac{N^{\epsilon/2}}{N\sqrt{t}}\right).
\end{align*}
Using this, interlacing and \eqref{eq:traceEstimate} we are now able to estimate $\tr M^{(j-1)}$:
\begin{align*}
    \tr M^{(j-1)}_{u}&=p\tr H^{(j-1)}_{u}+\sqrt{\tau_{N}}v\tr H^{(j-1)}_{u}\ac{W}^{(j-1)}_{2}\\
    &=\frac{p}{2\eta_{E,t}}\Im\tr\mc{F}^{(j-1)}_{u}+\frac{\sqrt{\tau_{N}}v}{2}\Im\tr\mc{F}^{(j-1)}_{u}(1\otimes\ac{W}^{(j-1)}_{2})\\
    &=\frac{Np}{t}+\frac{(u-u_{E,t})p}{\eta_{E,t}}\Im\tr G^{2}+O\left(\frac{1}{(Nt)^{1/4}}\right).
\end{align*}
For $\tr(M^{(j-1)}_{u})^{2}$ we have already seen that
\begin{align*}
    \tr(M^{(j-1)}_{u})^{2}&=a^{(j-1)}p^{2}+2b^{(j-1)}pv+c^{(j-1)}v^{2}\\
    &=\frac{N(1-t\Re\langle G^{2}\rangle)}{2\eta^{2}_{E,t}t}p^{2}+\frac{\beta_{E,t}v^{2}}{\eta^{2}_{E,t}}+O\left(\frac{\log^{2}N}{\sqrt{Nt}}\right).
\end{align*}
Thus we find
\begin{align}
    I_{j}(u)&=\int_{|p|<\sqrt{\frac{t^{3}}{N}}\log N}\int_{|v|<t\log N}\left[1+O\left(\frac{1}{(Nt)^{1/4}}\right)\right]\nonumber\\&\times\exp\left\{-\frac{N(1-t\Re\langle G^{2}\rangle)}{4\eta_{E,t}^{2}t}p^{2}-\frac{i(u-u_{E,t})p}{\eta_{E,t}}\Im\tr G^{2}-\frac{\tau_{E,t}}{2\eta_{E,t}^{2}}v^{2}-\frac{iy_{j}v}{\eta_{E,t}}\right\}dpdv,\label{eq:I3}
\end{align}
uniformly in $|u-u_{E,t}|<\sqrt{\frac{t}{N}}\log N$.

For $f_{j}$, we first note that when $|u-u_{E,t}|<\sqrt{\frac{t}{N}}\log N$ the exponent in the first line of \eqref{eq:f} is $O\left(\frac{\log N}{\sqrt{Nt}}\right)$. Thus we have
\begin{align*}
    f_{j}(u)&=\left[1+O\left(\frac{1}{(Nt)^{\frac{1}{2}\left(1-\frac{1}{n_{\epsilon}}\right)}}\right)\right]\frac{e^{\alpha_{E,t}+\beta_{E,t}}}{\det\left(1+i\sqrt{\tau_{N}}\mc{G}_{u}(\sigma_{Y}\otimes\ac{W}_{2})\right)}\prod_{l=1}^{j-1}\frac{\det V_{l}^{*}\mc{G}^{(l-1)}V_{l}}{\det V_{l}^{*}\mc{F}^{(l-1)}_{u}V_{l}}.
\end{align*}
By \eqref{eq:det1} we have
\begin{align*}
    \det\left(1+\sqrt{\tau_{N}}\mc{G}_{u}(\sigma_{Y}\otimes\ac{W}_{2})\right)&=\left[1+O\left(\frac{1}{Nt}\right)\right]\exp\left\{N\tau_{N}\langle G_{u+i\eta_{E,t}}\ac{W}_{2}G^{*}_{u+i\eta_{E,t}}\ac{W}_{2}\rangle\right\}.
\end{align*}
By the resolvent identity and \eqref{eq:C3.2'} we can replace $u$ by $u_{E,t}$:
\begin{align*}
    N\tau_{N}\langle G_{u+i\eta_{E,t}}\ac{W}_{2}G^{*}_{u+i\eta_{E,t}}\ac{W}_{2}\rangle&=N\tau_{N}\langle G\ac{W}_{2}G^{*}\ac{W}_{2}\rangle+O\left(\frac{\log N}{\sqrt{Nt}}\right)\\
    &=\alpha_{E,t}+\beta_{E,t}+O\left(\frac{\log N}{\sqrt{Nt}}\right).
\end{align*}
For each term in the product over $l$, we argue as before using the resolvent identity, \Cref{lem:minorresolvent} and \Cref{lem:minornorms}:
\begin{align*}
    \frac{\det V_{l}^{*}\mc{G}^{(l-1)}V_{l}}{\det V_{l}^{*}\mc{F}^{(l-1)}_{u}V_{l}}&=\frac{\det\left[1-(u-u_{E,t})(V_{l}^{*}\mc{F}^{(l-1)}_{u}V_{l})^{-1}V_{l}^{*}\mc{F}^{(l-1)}(\sigma_{X}\otimes1)\mc{F}^{(l-1)}_{u}V_{l}\right]}{\det\left[1+i\sqrt{\tau_{N}}(V_{l}^{*}\mc{G}^{(l-1)}V_{l})^{-1}V_{l}^{*}\mc{G}^{(l-1)}(\sigma_{Y}\otimes\ac{W}^{(l-1)}_{2})\mc{F}^{(l-1)}V_{l}\right]}\\
    &=1+O\left(\frac{1}{(Nt)^{\frac{1}{2}\left(1-\frac{1}{n_{\epsilon}}\right)}}\right).
\end{align*}
Hence
\begin{align}
    f_{j}(u)&=1+O\left(\frac{1}{(Nt)^{\frac{1}{2}\left(1-\frac{1}{n_{\epsilon}}\right)}}\right).\label{eq:f2}
\end{align}

Inserting \eqref{eq:I3} and \eqref{eq:f2} into $\wt{K}_{j}$ and using the fact that
\begin{align*}
    \phi(u)-\phi(u_{E,t})&=\frac{2(1-t\Re\langle G^{2}\rangle)}{t}(u-u_{E,t})^{2}+O\left(\frac{\log^{3}N}{\sqrt{Nt}}\right),
\end{align*}
we obtain
\begin{align*}
    &\psi_{j}(A^{(j-1)})\wt{K}_{j}(z_{j};A^{(j-1)})=\\&p_{N}|1-t\langle G^{2}\rangle|\int_{|u-u_{E,t}|<\sqrt{\frac{t}{N}}\log N}\int_{|p|<\sqrt{\frac{t^{3}}{N}}\log N}\int_{|v|<t\log N}\left[1+O\left(\frac{1}{(Nt)^{\frac{1}{4}}}\right)\right]\\
    &\times\exp\left\{-\frac{N(1-t\Re\langle G^{2}\rangle|^{2})}{t}(u-u_{E,t})^{2}-\frac{N(1-t\Re\langle G^{2}\rangle)}{4\eta_{E,t}^{2}t}p^{2}\right.\\
    &\left.-\frac{i(u-u_{E,t})p}{\eta_{E,t}}\Im\tr G^{2}-\frac{\tau_{E,t}}{2\eta_{E,t}^{2}}v^{2}-\frac{iy_{j}v}{\eta_{E,t}}\right\}dpdvdu.
\end{align*}
Taking the error outside the integral, extending to the whole real line and calculating the resulting Gaussian integrals we obtain \eqref{eq:KAsymp}.
\end{proof}

\begin{proof}[Proof of \Cref{lem:nuConc}]
We recall the statement of \Cref{lem:nuConc}: with 
\begin{align*}
    \mc{E}_{j}&=\left\{\left|\left|\frac{2\eta_{E,t}\mbf{v}_{j}^{*}G^{(j-1)}\mbf{v}_{j}}{1-t\langle G^{2}\rangle}\right|^{2}-1\right|<\frac{C\log N}{\sqrt{Nt}}\right\},
\end{align*}
we have
\begin{align*}
    \nu_{j}(S^{N-j}\setminus\mc{E}_{j})\leq e^{-C'\log^{2}N}.
\end{align*}
This will follow if we can show that
\begin{align}
    \mbf{v}_{j}^{*}\Re(G^{(j-1)})\mbf{v}_{j}&=\left[1+O\left(\frac{\log N}{\sqrt{Nt}}\right)\right]\frac{t}{2\eta_{E,t}}\Im\langle G^{2}\rangle,\label{eq:realPart}\\
    \mbf{v}_{j}^{*}\Im(G^{(j-1)})\mbf{v}_{j}&=\left[1+O\left(\frac{\log N}{\sqrt{Nt}}\right)\right]\frac{1-t\Re\langle G^{2}\rangle}{2\eta_{E,t}},\label{eq:imagPart}
\end{align}
with probability $1-e^{-C\log^{2}N}$. We will prove \eqref{eq:imagPart}; the proof of \eqref{eq:realPart} is exactly the same.

Let $0<r<N\eta_{E,t}^{2}/2t$. Define the matrix $M^{(j-1)}_{u}$ by
\begin{align*}
    M^{(j-1)}_{u}&=\frac{\eta_{E,t}t}{N}\sqrt{H^{(j-1)}_{u}}\Im G^{(j-1)}\sqrt{H^{(j-1)}_{u}}
\end{align*}
and let $M^{(j-1)}=M^{(j-1)}_{u_{E,t}}$. Note that by our restriction on $r$ we have
\begin{align}
    |r|\|M^{(j-1)}_{u}\|&<\frac{1}{2}.\label{eq:MNorm}
\end{align}
We will bound the moment generating function
\begin{align*}
    m_{I}(r)&:=e^{-r\tr M^{(j-1)}}\mbb{E}_{j}\left[e^{r\eta\mbf{v}_{j}^{*}\Im G^{(j-1)}\mbf{v}_{j}}\right].
\end{align*}
Following the same steps as in the proof of \eqref{eq:K}, we have the equality
\begin{align*}
    m_{I}(r)&=\frac{1}{K_{j}(z_{j};A^{(j-1)})}e^{-r\tr M^{(j-1)}+\frac{N}{t(1+\tau_{N})}\eta_{E,t}^{2}+\frac{t}{(1-\tau_{N})N\eta^{2}_{E,t}}y_{j}^{2}}\\
    &\int_{-\infty}^{\infty}\int_{-\infty}^{\infty}e^{-\frac{N}{t(1-\tau_{N})}(u-\Re w_{j})^{2}-\frac{N\tau_{N}}{2t(1-\tau_{N}^{2})}v^{2}-\frac{iy_{j}v}{(1-\tau_{N})\eta_{E,t}}}\hat{m}_{I}(u,v;r)dvdu,
\end{align*}
where
\begin{align*}
    &\hat{m}_{I}(u,v;r)=\\&\left(\frac{N}{\pi t(1+\tau_{N})}\right)^{N-j}\int_{S^{N-j}}e^{-\frac{N}{t(1+\tau_{N})}\mbf{v}_{j}^{*}\left(\eta_{E,t}^{2}+|\ac{A}^{(j-1)}-u|^{2}+iv\Im \ac{A}^{(j-1)}-\frac{rt\eta_{E,t}}{N}\Im G^{(j-1)}\right)\mbf{v}_{j}}dS_{N-j}(\mbf{v}_{j}).
\end{align*}
Let
\begin{align}
    \xi_{j}(u)&=e^{\frac{C|u-u_{E,t}|}{t}}\int_{-\infty}^{\infty}e^{\frac{iNp}{t(1+\tau_{N})}}\det^{-1}\left(1+ip\sqrt{H^{(j-1)}_{u}}(1-rM^{(j-1)}_{u})^{-1}\sqrt{H^{(j-1)}_{u}}\right)dp.\label{eq:xi}
\end{align}
Then by inserting absolute values and rewriting the integral over $\mbf{v}_{j}$ in terms of an integral over $p\in\mbb{R}$ as we did in the proof of \Cref{lem:KBound} we have
\begin{align*}
    m_{I}(r)&\leq\frac{CN}{t^{5/2}K_{j}(z_{j};A^{(j-1)})}e^{\frac{N}{t}\eta_{E,t}^{2}+\frac{1+\tau_{N}}{Nt(1-\tau_{N})}y^{2}}\int_{-\infty}^{\infty}e^{-N\phi(u)}\xi_{j}(u)\det^{-1}(1-rM^{(j-1)}_{u})du,
\end{align*}
where $\phi$ was defined in \eqref{eq:phi}.

The bound
\begin{align*}
    |\xi_{j}(u)|&\leq C\sqrt{\frac{t^{3}}{N}}
\end{align*}
follows in the same way as the equivalent bound \eqref{eq:hBound} on $h_{j}$. To see this we note that since $\|rM^{(j-1)}_{u}\|<1/2$ we have
\begin{align*}
    \tr\left(\sqrt{H^{(j-1)}_{u}}(1-rM^{(j-1)}_{u})^{-1}\sqrt{H^{(j-1)}_{u}}\right)^{2}&\geq\frac{1}{(1+\|rM^{(j-1)}_{u}\|)^{2}}\tr(H^{(j-1)}_{u})^{2}\\
    &\geq\frac{CN}{t^{3}},
\end{align*}
so that the integral in \eqref{eq:xi} can be restricted to the region $|p|<\sqrt{\frac{t^{3}}{N}}\log N$.

For the term $\det^{-1}(1-rM^{(j-1)}_{u})$ we use $-\log(1-x)\leq x+\frac{x^{2}}{1-x}$ to obtain
\begin{align*}
    e^{-r\tr M^{(j-1)}}\det^{-1}\left(1-rM^{(j-1)}_{u}\right)&\leq\exp\left\{r\tr\left(M^{(j-1)}_{u}-M^{(j-1)}\right)+2r^{2}\tr\left(M^{(j-1)}_{u}\right)^{2}\right\}.
\end{align*}
We have used the fact that $\left\|rM^{(j-1)}_{u}\right\|<1/2$. For the first term we have by Taylor's theorem
\begin{align*}
    \left|\tr\big(M^{(j-1)}_{u}-M^{(j-1)}_{u_{E,t}}\big)\right|&\leq\left|u-u_{E,t}\right|\max_{0\leq s\leq 1}\left|\frac{\partial}{\partial u(s)}\tr M^{(j-1)}_{u(s)}\right|,
\end{align*}
where $u(s)=(1-s)u_{E,t}+su$. Since 
\begin{align*}
    \partial_{u}H^{(j-1)}_{u}&=2H^{(j-1)}_{u}\left(W^{(j-1)}_{1}-u\right)H^{(j-1)}_{u},
\end{align*}
and
\begin{align*}
    \left\|H^{(j-1)}_{u}\left(W^{(j-1)}_{1}-u\right)\right\|&\leq\frac{1}{2\eta_{E,t}},
\end{align*}
we obtain
\begin{align*}
    \left|\tr\big(M^{(j-1)}_{u}-M^{(j-1)}_{u_{E,t}}\big)\right|&\leq\frac{r\left|u-u_{E,t}\right|t}{N\eta_{E,t}}\max_{0\leq s\leq 1}\tr H^{(j-1)}_{u(s)}\\
    &\leq\frac{C\left|u-u_{E,t}\right|}{t}.
\end{align*}
The last inequality follows from \eqref{eq:traceEstimate} and \eqref{eq:C1.1}. Since we have $\phi(u)-\phi(u_{E,t})\geq C(u-u_{E,t})^{2}/t$, we deduce that we can again restrict $u$ to the region $|u-u_{E,t}|<\sqrt{\frac{t}{N}}\log N$. In this region we have the explicit bound
\begin{align*}
    \left|\tr\big(M^{(j-1)}_{u}-M^{(j-1)}_{u_{E,t}}\big)\right|&\leq2\sqrt{\frac{t}{N}}\frac{\log N}{\eta_{E,t}},
\end{align*}
which follows in the usual way by replacing $H^{(j-1)}$ with $H$ by interlacing, $H$ with $|G|^{2}$ by \eqref{eq:traceEstimate}, $u$ with $u_{E,t}$ by the resolvent identity, and finally using $\tr|G|^{2}=N/t$. For the trace of $(M^{(j-1)}_{u})^{2}$ we have likewise
\begin{align*}
    \tr\big(M^{(j-1)}_{u}\big)^{2}&=\frac{\eta^{2}_{E,t}t^{2}}{N^{2}}\tr\big(H^{(j-1)}_{u}\Im G^{(j-1)}\big)^{2}\\
    &\leq\frac{t^{2}}{N^{2}\eta_{E,t}^{2}}\tr H^{(j-1)}_{u}\\
    &\leq\frac{2t}{N\eta_{E,t}^{2}}.
\end{align*}

Inserting this into the expression for $m_{I}(r)$ we find
\begin{align*}
    m_{I}(r)&\leq\frac{Ct^{5/2}}{\sqrt{N}K_{j}(z_{j};A^{(j-1)})}\exp\left\{2\sqrt{\frac{t}{N}}\frac{\log N}{\eta_{E,t}}r+\frac{2t}{N\eta_{E,t}^{2}}r^{2}\right\}.
\end{align*}
Since we have assumed that $\mbf{v}_{l}\in\mc{E}_{l}$ for $l=1,...,j-1$, we can use the asymptotics of $\wt{K}_{j}$ in \Cref{lem:KAsymp} to deduce the lower bound
\begin{align*}
    K_{j}(z_{j};A^{(j-1)})&\geq\frac{C\sqrt{N}}{t^{3}},
\end{align*}
and hence
\begin{align*}
    m_{I}(r)&\leq\frac{C}{\sqrt{t}}\exp\left\{2\sqrt{\frac{t}{N}}\frac{\log N}{\eta_{E,t}}r+\frac{2t}{N\eta_{E,t}^{2}}r^{2}\right\}.
\end{align*}
Translating this into a probability bound via Markov's inequality we find
\begin{align*}
    \nu_{j}\left(\left\{\left|\eta_{E,t}\mbf{v}_{l}^{*}\Im G^{(j-1)}\mbf{v}_{l}-\eta_{E,t}t\Tr{H^{(j-1)}\Im G^{(j-1)}}\right|>r+2\sqrt{\frac{t}{N}}\frac{\log N}{\eta_{E,t}}\right\}\right)&\leq \frac{C}{\sqrt{t}}e^{-\frac{N\eta_{E,t}^{2}}{8t}r^{2}}.
\end{align*}
Setting $r=2\sqrt{\frac{t}{N\eta_{E,t}^{2}}}\log N$ we obtain
\begin{align*}
    \mbf{v}_{j}^{*}G^{(j-1)}\mbf{v}_{j}&=t\langle H^{(j-1)}\Im G^{(j-1)}\rangle+O\left(\frac{\log N}{\sqrt{Nt^{3}}}\right).
\end{align*}
To estimate the leading term we rewrite it in terms of $\mc{F}^{(j-1)}$ and $\mc{G}^{(j-1)}$ and then carry out the usual steps of replacing $A^{(j-1)}$ by $A$ using \cref{lem:minorresolvent}, replacing $\mc{F}$ by $\mc{G}$ using the series expansion of $\mc{F}=(1+i\sqrt{\tau_{N}}\mc{G}(\sigma_{Y}\otimes\ac{W}_{2}))^{-1}\mc{G}$ and using \eqref{eq:C3.1'} and \eqref{eq:C3.2'} to estimate each resulting term:
\begin{align*}
    t\langle H^{(j-1)}\Im G^{(j-1)}\rangle&=\frac{t}{2i\eta_{E,t}}\Tr{\mc{F}^{j-1)}\Im(\mc{G}^{(j-1)})\begin{pmatrix}1&0\\0&0\end{pmatrix}}\\
    &=\frac{t}{2i\eta_{E,t}}\Tr{\mc{F}\Im(\mc{G})\begin{pmatrix}1&0\\0&0\end{pmatrix}}+O\left(\frac{1}{Nt^{2}}\right)\\
    &=\frac{t}{2i\eta_{E,t}}\Tr{\mc{G}\Im(\mc{G})\begin{pmatrix}1&0\\0&0\end{pmatrix}}+O\left(\frac{1}{Nt^{2}}\right)\\
    &=\frac{1-t\Re\Tr{G^{2}}}{2\eta_{E,t}}+O\left(\frac{1}{Nt^{2}}\right)\\
    &=\left[1+O\left(\frac{1}{Nt}\right)\right]\frac{1-t\Re\Tr{G^{2}}}{2\eta_{E,t}}.
\end{align*}
Since this is greater than $C/t$, \eqref{eq:imagPart} follows.
\end{proof}

\section{Proof of \Cref{thm2}}
Fix $k>0$, $0<\epsilon<1/2$, $t=N^{-1+2\epsilon}$ and $E\in(-2,2)$. Let $W_{1},W_{2}$ be independent Wigner matrices with atom distribution $\nu_{L}$ satisfying \eqref{eq:nu1} and \eqref{eq:nu2}, $A=W_{1}+i\sqrt{\tau_{N}}W_{2}$, and $M_{t}=A+\sqrt{t}B$, where $B$ is a Gaussian elliptic matrix. For $f\in C_{c}(\mbb{C}^{k})$, let
\begin{align*}
    f_{N,E}(\mbf{z})&=f\left(N\pi\rho_{sc}(E)(z_{1}-E),...,N\pi\rho_{sc}(E)(z_{k}-E)\right),
\end{align*}
and define the statistic
\begin{align*}
    S^{(k)}_{f,E}(M_{t})&=\sum_{i_{1}\neq\cdots\neq i_{k}}f_{N,E}(z_{i_{1}}(M_{t}),...,z_{i_{k}}(M_{t})).
\end{align*}
Recall that if $(W_{1},W_{2})\in\mc{W}_{N,\epsilon}$, then by \Cref{thm1} we can define $\lambda_{E,t}=u_{E,t}+i\eta_{E,t}$ by $\lambda_{E,t}=E+t\langle G_{\lambda_{E,t}}\rangle$. Then we can define
\begin{align}
    \tau_{E,t}&=\beta_{N}(\lambda_{E,t})+\frac{N\tau_{N}\eta_{E,t}^{2}}{t},
\end{align}
and the sequence of events
\begin{align}
    \mc{A}_{N}&=\left\{(W_{1},W_{2})\in\mc{W}_{N,\epsilon}\right\}\bigcap\left\{\tau_{E,t}=\tau_{E}+o(1)\right\}\nonumber\\&\bigcap\left\{\frac{\pi t\rho_{sc}(E)}{\eta_{E,t}}=1+o(1)\right\}\bigcap\left\{N^{1/2}|\langle W_{2}\rangle|=o(1)\right\},\label{eq:A}
\end{align}
where $\tau_{E}=\pi\rho_{sc}(E)\lim_{N\to\infty}N\tau_{N}$. If for any $D>0$ $\mc{A}_{N}$ holds with probability $1-N^{-D}$ for sufficiently large $N$, we will have
\begin{align*}
    \mbb{E}\left[S^{(k)}_{f,E}(M_{t})\right]&=\mbb{E}\left[1_{\mc{A}_{N}}S^{(k)}_{f,E}(M_{t})\right]+O(N^{k-D}).
\end{align*}
We can then choose $D>k$ so that the error can be neglected. Now we write
\begin{align}
    \mbb{E}\left[S^{(k)}_{f,E}(M_{t})\right]&=\mbb{E}\left[1_{\mc{A}_{N}}\mbb{E}\left[S^{(k)}_{f,E}(M_{t})\Big|A\right]\right]+o(1)\nonumber\\
    &=\mbb{E}\left[1_{\mc{A}_{N}}\left(\frac{\pi t\rho_{sc}(E)}{\eta_{E,t}}\right)^{2k}\int_{\mbb{C}^{k}}\wt{f}(\mbf{z})\wt{\rho}^{(k)}_{N}\left(\mbf{z};A\right)d\mbf{z}\right]+o(1),\label{eq:ES}
\end{align}
where 
\begin{align*}
    \wt{f}(\mbf{z})&=f\left(\frac{\pi t\rho_{sc}(E)}{\eta_{E,t}}\mbf{z}+iN\sqrt{\tau_{N}}\pi\rho_{sc}(E)\langle W_{2}\rangle\right)
\end{align*}
and $\wt{\rho}^{(k)}_{N}$ was defined in \eqref{eq:rhoTilde}. Then from the proof of \Cref{thm1} we have
\begin{align*}
    \wt{\rho}^{(k)}_{N}(\mbf{z};A)&=\left[1+O\left(\frac{1}{(Nt)^{1/4}}\right)\right]\rho^{(k)}_{\tau_{E,t}}(\mbf{z}),
\end{align*}
uniformly in compact subsets of $\mbb{C}^{k}$. On the event $\mc{A}_{N}$ we have by definition $\tau_{E,t}=\tau_{E}+o(1)$, $\pi t\rho_{sc}(E)/\eta_{E,t}=1+o(1)$ and $N\sqrt{\tau_{N}}\langle W_{2}\rangle=o(1)$. Since $f\in C_{c}(\mbb{C}^{k})$, we obtain
\begin{align*}
    \wt{f}(\mbf{z})&=f(\mbf{z})+o(1),
\end{align*}
uniformly in $\mbb{C}^{k}$. Hence we can take the limit on the right hand side of \eqref{eq:ES} and obtain
\begin{align*}
    \lim_{N\to\infty}\mbb{E}\left[S^{(k)}_{f,E}(M)\right]&=\int_{\mbb{C}^{k}}f(\mbf{z})\rho^{(k)}_{\tau_{E}}(\mbf{z})d\mbf{z}.
\end{align*}
Thus we need to prove the following lemma.
\begin{lemma}\label{lem:prob}
Let $\epsilon>0$ and $t\geq N^{-1+2\epsilon}$. Let $(W_{1},W_{2})$ be a pair of independent Wigner matrices and define the event $\mc{A}_{N}$ as in \eqref{eq:A}. Then for any $D>0$, we have
\begin{align}
    P(\mc{A}_{N})&\geq 1-N^{-D},
\end{align}
for sufficiently large $N$.
\end{lemma}
\begin{proof}
This follows from local laws for Wigner matrices. Let $m_{sc}$ denote the Stieltjes transform of the semicircle law. The following is the well-known single resolvent averaged local law (see e.g. Theorem 2.2 in \cite{erdos_random_2019}):
\begin{align}
    \left|\langle G_{z}\rangle-m_{sc}(z)\right|&\leq\frac{N^{\epsilon/2}}{N|\Im z|},
\end{align}
uniformly in $z\in S_{\epsilon}$ with probability $1-N^{-D}$. This can be extended to hold simultaneously for all $z\in S_{\epsilon}$ by the usual grid argument. The bounds in \eqref{eq:C0}, \eqref{eq:C1.1}, \eqref{eq:C1.2} and \eqref{eq:C1.3} are standard consequences of this local law. For the remaining conditions that define $\mc{W}_{N,\epsilon}$, we require a multi-resolvent local law. Such a law has been proven by Cipolloni, Erd\H{o}s and Schr\"{o}der \cite{cipolloni_optimal_2022}; their Theorem 2.5 states that there is a deterministic matrix $M(z_{1},B_{1},...,B_{n-1},z_{n})$ such that for any $D>0$ we have
\begin{align}
    \left|\left\langle\prod_{j=1}^{n}G_{j}B_{j}\right\rangle-\langle M(z_{1},B_{1},...,B_{n-1},z_{n})B_{n}\rangle\right|&\leq\frac{N^{\epsilon/2}}{N\eta^{n-a/2}},\label{eq:multiLL}
\end{align}
with probability $1-N^{-D}$ for sufficiently large $N$, where $z_{j}\in S_{\epsilon}$, $\eta=\min|\Im z_{j}|$ and $B_{j}$ are deterministic matrices such that $\|B_{j}\|<C$ and $0\leq a\leq n$ are traceless. As with the single resolvent local law, we can extend this to hold simultaneously for all $\mbf{z}\in S^{n}_{\epsilon}$ by a grid argument. By adjusting $D$ and taking a union bound this holds simultaneously for all $n\leq 4n_{\epsilon}=4\lceil48/\epsilon\rceil$ with probability $1-N^{-D}$. The particular form of $M(z_{1},B_{1},...,B_{n-1},z_{n})$, although explicit in \cite{cipolloni_optimal_2022}, is not relevant for us, except for the following special cases:
\begin{align}
    M(z_{1})&=m_{sc}(z_{1})1_{N},\\
    M(z_{1},B_{1},z_{2})&=m_{sc}(z_{1})m_{sc}(z_{2})B_{1},\label{eq:M2}
\end{align}
for $\tr B_{1}=0$. In the general case we need the following bound from Lemma 2.4 in the same paper.
\begin{align}
    \left|\langle M(z_{1},B_{1},...,B_{n-1},z_{n})B_{n}\rangle\right|&\leq\frac{C}{\eta^{n-1-\lceil a/2\rceil}}.
\end{align}
This is larger than the error in \eqref{eq:multiLL} by a factor $N^{1-\epsilon/2}\eta$. Conditioning on $\ac{W}_{2}$ and setting $B_{j}=\ac{W}_{2}$, we have $a=n$ traceless matrices and so 
\begin{align*}
    \left|\langle G_{z}\ac{W}_{2}\rangle\right|&\leq\frac{N^{\epsilon/2}\|\ac{W}_{2}\|}{N\sqrt{\eta}},\\
    \left|\left\langle\prod_{j=1}^{n}G_{z_{j}}\ac{W}_{2}\right\rangle\right|&\leq\frac{C\|\ac{W}_{2}\|^{n}}{\eta^{n/2-1}},
\end{align*}
for all $n\leq 4n_{\epsilon}$ with probability $1-N^{-D}$. Since $\|\ac{W}_{2}\|<C$ probability $1-N^{-D}$, \eqref{eq:C3.1} and \eqref{eq:C3.2} hold with probability $1-N^{-D}$.

Using \eqref{eq:M2} we find that, conditioned on $W_{2}$,
\begin{align*}
    \langle(\Im(G_{z})\ac{W}_{2})^{2}\rangle&=(\Im m_{sc}(z))^{2}\langle\ac{W}_{2}^{2}\rangle+O\left(\frac{N^{\epsilon/2}}{Nt}\right),
\end{align*}
with probability $1-N^{-D}$. By concentration we have
\begin{align*}
    \left|\langle W_{2}\rangle\right|&\leq\frac{C}{N^{\frac{1+\epsilon}{2}}},\\
    \langle\ac{W}_{2}^{2}\rangle&=1+o(1),
\end{align*}
with probability $1-N^{-D}$. Therefore
\begin{align*}
    \beta_{N}(z)&=N\tau_{N}(\Im m_{sc}(z))^{2}+o(1)
\end{align*}
and \eqref{eq:C2} holds with probability $1-N^{-D}$. Moreover, when $z=\lambda_{E,t}$ we have $\Re z=u_{E,t}=E+O(t)$ and $\Im z=\eta_{E,t}=O(t)$, and so $\Im m_{sc}(\lambda_{E,t})=\pi\rho_{sc}(E)+O(t)$. Since we also have
\begin{align*}
    \frac{\eta_{E,t}}{t}&=\Im\langle G_{\lambda_{E,t}}\rangle\\
    &=\pi\rho_{sc}(E)+o(1),
\end{align*}
we conclude that
\begin{align*}
    \tau_{E,t}&=\beta_{N}(\lambda_{E,t})+\frac{N\tau_{N}\eta_{E,t}^{2}}{t}\\
    &=N\tau_{N}\pi\rho_{sc}(E)\left(1+\pi t\rho_{sc}(E)\right)+o(1)\\
    &=\tau_{E}+o(1)
\end{align*}    
and
\begin{align*}
    \frac{\pi t\rho_{sc}(E)}{\eta_{E,t}}&=\frac{\pi\rho_{sc}(E)}{\Im\langle G_{z}\rangle}\\
    &=1+o(1),
\end{align*}
with probability $1-N^{-D}$. 
\end{proof}

Now we use the method of time reversal from \cite{erdos_bulk_2010}. The essence of this method is to approximate the time-reversed Dyson Brownian motion by the operator
\begin{align*}
    T_{n}&=\sum_{m=0}^{n-1}\frac{(-1)^{m}}{m!}(tL)^{m},
\end{align*}
where $L$ is the generator of the Ornstein-Uhlenbeck process. Given an atom distribution $\nu$ with density $f$, it is shown in Proposition 2.1 in \cite{erdos_bulk_2010} that
\begin{align}
    \int\frac{|e^{tL}T_{n}f-f|^{2}}{e^{tL}T_{n}f}e^{-x^{2}}dx&\leq Ct^{2n},\label{eq:error}
\end{align}
if $f$ is sufficiently smooth. On the other hand, $e^{tL}g$ is the density of a Gauss-divisible matrix element, and if $f$ is sufficiently smooth then $g=T_{n}f$ is regular enough to ensure that the resulting matrix satisfies the local law. Therefore we can apply \Cref{thm1} to deduce that the correlation functions of the matrix with atom density $e^{tL}T_{n}f$ are universal, and then appproximate these correlation functions with those of the matrix with atom density $f$. The latter can be done by noting that the error in the approximation of the joint density of all matrix elements is an upper bound to the error in the approximation of marginals, and in particular the approximation of correlation functions. To adapt this to the current setting, we simply note that since $W_{1}$ and $W_{2}$ are independent, the probability measure of $A=W_{1}+i\sqrt{\tau_{N}}W_{2}$ factorises, $dP(A)\propto dP_{1}(W_{1})dP_{2}(W_{2})$, and we can approximate $P_{1}$ and $P_{2}$ separately.

\paragraph{\small Acknowledgements}
\small{The author is grateful to Anna Maltsev for helpful discussions. This work was supported by the Royal Society, grant number RF/ERE210051.}

\bibliographystyle{plain}
\bibliography{references.bib}

\end{document}